\documentclass[
	final,
	paper=a4,
	pagesize=auto,
	fontsize=10pt,
	version=last,
]{scrartcl}

\usepackage[utf8]{inputenc}
\usepackage[
	english,
]{babel}

\KOMAoptions{%
   DIV=10,% (Size of Text Body, higher values = greater textbody)
}%
\KOMAoptions{%
   numbers=noenddot
}%
\KOMAoptions{% (most options are for package typearea)
   twoside=semi,  % two side layout (non alternating margins!)
   twocolumn=false, % (true)
   headinclude=false,%
   footinclude=false,%
   mpinclude=false,%
   headlines=2.1,%
   headsepline=true,%
   footsepline=false,%
   cleardoublepage=empty %plain, headings
}%
\KOMAoptions{%
   parskip=relative, % change indentation according to fontsize (recommanded)
}%
\KOMAoptions{%
}%  
\KOMAoptions{%
}%  
\KOMAoptions{%
   bibliography=totoc % add Bibliography to TOC without number
}%
\KOMAoptions{%
}%
\KOMAoptions{%
}%
\KOMAoptions{% 	
}%

\makeatletter
\newcommand{\LoadPackagesNow}{}
\newcommand{\LoadPackageLater}[2][]{%
   \g@addto@macro{\LoadPackagesNow}{%
      \usepackage[#1]{#2}%
   }%
}
\makeatother

\usepackage{xspace}
\usepackage{xargs}
\usepackage{ifthen}
\usepackage{etoolbox}

\usepackage{calc}
\usepackage[
	a4paper,
	margin=3cm,
	]{geometry}
\usepackage{framed}
\usepackage{mdframed}
\usepackage{pbox}
\usepackage{enumitem}
\usepackage{array}
\usepackage{multirow}
\usepackage{longtable}
\usepackage{hhline}
\usepackage{afterpage}
\usepackage{pdflscape}
\usepackage{rotating}
\usepackage{booktabs}
\usepackage{setspace}
\usepackage[
   bottom,      % Footnotes appear always on bottom. This is necessary
   stable,      % Make footnotes stable in section titles
   multiple,    % rearrange multiple footnotes intelligent in the text.
   flushmargin,
   hang,
]{footmisc}
\usepackage{mathpazo}
\usepackage[%
   headsepline,
   automark,
]{scrlayer-scrpage}

\clearscrheadings
\clearscrplain
\lehead{\pagemark}
\rohead{\pagemark}
\rehead{\headmark}
\lohead{\shortauthor}
\automark[section]{section} %[rechts]{links}
\usepackage{relsize}
\usepackage[normalem]{ulem}      
\usepackage{soul}		           
\usepackage{url}
\usepackage{lipsum}
\usepackage[table]{xcolor}
\usepackage{tikz}
\usetikzlibrary{arrows,shapes,calc,decorations.pathreplacing,intersections,quotes,angles,positioning,fit,petri,through,matrix}
\tikzset{
	pics/carc/.style args={#1:#2:#3}{
		code={
			\draw[pic actions] (#1:#3) arc(#1:#2:#3);
		}
	}
}
\tikzstyle{blackdot}=[shape=circle,fill=black,minimum size=1mm,inner sep=0pt,outer sep=0pt]
\usepackage[%
]{graphicx}
\usepackage{epstopdf}
\usepackage{wrapfig}
\usepackage{caption}
\usepackage{subcaption}
\usepackage[
   tbtags,    % 'Top-or-bottom tags': For a split equation, place equation numbers level
   sumlimits,  %(default) Place the subscripts and superscripts of summation
   nointlimits, % (default) Opposite of intlimits.
   namelimits, % (default) Like sumlimits, but for certain 'operator names' such as
   reqno,     % Place equation numbers on the right.
]{amsmath} %

\usepackage{amsfonts}
\usepackage{mathrsfs} %% Schreibschriftbuchstaben für den Mathematiksatz (nur Großbuchstaben)
\usepackage{dsfont}
\usepackage{amssymb}
\usepackage{units}
\LoadPackageLater{amsthm}
\LoadPackageLater{thmtools}
\usepackage[fixamsmath,disallowspaces]{mathtools}
\mathtoolsset{showonlyrefs}
\mathtoolsset{centercolon=true}
\usepackage{bm} 
\makeatletter
\g@addto@macro\bfseries{\boldmath}
\makeatother
\allowdisplaybreaks[4]
\numberwithin{equation}{section}

\usepackage[title]{appendix}
 \usepackage[%
 	square,	% for square brackets;
 	numbers,	% for numerical citations;
 	sort&compress,    % as sort but in addition multiple numerical citations
 ]{natbib}
\usepackage[textsize=tiny,english,colorinlistoftodos,
	disable,
	]{todonotes}
\definecolor{pdfurlcolor}{rgb}{0,0,0.6}
\definecolor{pdffilecolor}{rgb}{0.7,0,0}
\definecolor{pdflinkcolor}{rgb}{0,0,0.6}
\definecolor{pdfcitecolor}{rgb}{0,0,0.6}
\usepackage[
  colorlinks=true,         % Links erhalten Farben statt Kaeten
  urlcolor=pdfurlcolor,    % \href{...}{...} external (URL)
  filecolor=pdffilecolor,  % \href{...} local file
  linkcolor=pdflinkcolor,  %\ref{...} and \pageref{...}
  citecolor=pdfcitecolor,  %
  raiselinks=true,			 % calculate real height of the link
  breaklinks,              % Links berstehen Zeilenumbruch
  verbose,
  hyperindex=true,         % backlinkex index
  linktocpage=true,        % Inhaltsverzeichnis verlinkt Seiten
  hyperfootnotes=false,     % Keine Links auf Fussnoten
  bookmarks=true,          % Erzeugung von Bookmarks fuer PDF-Viewer
  bookmarksopenlevel=2,    % Gliederungstiefe der Bookmarks
  bookmarksopen=true,      % Expandierte Untermenues in Bookmarks
  bookmarksnumbered=true,  % Nummerierung der Bookmarks
  bookmarkstype=toc,       % Art der Verzeichnisses
  plainpages=false,        % Anchors even on plain pages ?
  pageanchor=true,         % Pages are linkable
  pdfdisplaydoctitle=true, % Dokumententitel statt Dateiname im Fenstertitel
  pdfstartview=FitH,       % Dokument wird Fit Width geaefnet
  pdfpagemode=UseOutlines, % Bookmarks im Viewer anzeigen
  pdfpagelabels=true,           % set PDF page labels
  pdfpagelayout=OneColumn, % zweiseitige Darstellung: ungerade Seiten
]{hyperref}

\LoadPackagesNow

\newcommand{\ifargdef}[3][{}]{\ifthenelse{\equal{#2}{}}{#1}{#3}}

\newlength{\hangind}
\newcommand{\myhangindent}[1]{\settowidth{\hangind}{\widthof{#1}}\hangindent=\the\hangind}

\setkomafont{section}{\Large\rmfamily\bfseries}
\setkomafont{subsection}{\large\rmfamily\bfseries}
\setkomafont{subsubsection}{\rmfamily\bfseries}
\setkomafont{paragraph}{\rmfamily\bfseries}
\setkomafont{pageheadfoot}{\smaller\scshape\rmfamily}

\usepackage{tabularx}
\hypersetup{
	pdfauthor={Martin Genzel and Christian Kipp},
	pdftitle={Generic Error Bounds for the Generalized Lasso with Sub-Exponential Data},
	pdfsubject={Article},
	pdfcreator={PDF-LaTeX},
}
\newenvironment{highlight}{\begin{quote}\itshape}{\end{quote}}

\newenvironment{listing}
{\begin{itemize}[itemindent=0em,leftmargin=1.2em]}
{\end{itemize}}

\newenvironment{rmklist}
{\begin{enumerate}[label={(\arabic*)},itemindent=2em,leftmargin=0em]}
{\end{enumerate}}

\newenvironment{thmproperties}
{\begin{enumerate}[label={(\roman*)}]}
{\end{enumerate}}

\newtheoremstyle{claim}
	{\topsep}{\topsep}%
	{\itshape}%         Body font
	{}%         Indent amount (empty = no indent, \parindent = para indent)
	{}% Thm head font
	{}%        Punctuation after thm head
	{.5em}%     Space after thm head (\newline = linebreak)
	{{\bfseries\boldmath\thmname{#1} \thmnumber{#2}} \thmnote{(#3)}}%         Thm head spec

\newtheoremstyle{definition}
	{\topsep}{\topsep}%
	{}%         Body font
	{}%         Indent amount (empty = no indent, \parindent = para indent)
	{}% Thm head font
	{}%        Punctuation after thm head
	{.5em}%     Space after thm head (\newline = linebreak)
	{\textbf{\thmname{#1} \thmnumber{#2}} \thmnote{(#3)}}%         Thm head spec
	
\newtheoremstyle{algorithm}
	{\topsep}{\topsep}%
	{}%         Body font
	{}%         Indent amount (empty = no indent, \parindent = para indent)
	{\bfseries\boldmath}% Thm head font
	{}%        Punctuation after thm head
	{\newline}%     Space after thm head (\newline = linebreak)
	{\thmname{#1} \thmnumber{#2} \thmnote{(#3)}}%         Thm head spec

\mdfdefinestyle{boxed}{innertopmargin =6pt, splittopskip = \topskip, skipbelow=6pt, skipabove=12pt}

\mdfdefinestyle{emphframe}{linecolor=black, innertopmargin = 3pt, splittopskip = \topskip, skipbelow=6pt, skipabove=6pt}

\declaretheorem[style=claim,numberwithin=section]{theorem}
\declaretheorem[style=claim,sibling=theorem]{lemma}
\declaretheorem[style=claim,sibling=theorem]{fact}
\declaretheorem[style=claim,sibling=theorem]{corollary}
\declaretheorem[style=claim,sibling=theorem]{proposition}
\declaretheorem[style=definition,sibling=theorem]{definition}

\declaretheorem[style=definition,sibling=theorem]{assumption}

\declaretheorem[style=claim,sibling=theorem]{problem}

\declaretheorem[style=definition,sibling=theorem,qed=$\Diamond$]{remark}

\declaretheorem[style=algorithm,sibling=theorem,%
	preheadhook={\begin{mdframed}[style=emphframe] \setcounter{mpfootnote}{\value{footnote}}},%
	postfoothook=\end{mdframed}]{experiment}
\declaretheorem[style=algorithm,sibling=theorem,%
	preheadhook={\begin{mdframed}[style=emphframe] \setcounter{mpfootnote}{\value{footnote}}},%
	postfoothook=\end{mdframed}]{algorithm}
\declaretheorem[style=definition,sibling=theorem,%
	preheadhook={\begin{mdframed}[style=boxed] \setcounter{mpfootnote}{\value{footnote}}},%
	postfoothook=\end{mdframed}]{recipe}
\newcommand{\opleft}[1]{\mathopen{}\left#1}
\newcommand{\opright}[1]{\right#1\mathclose{}}
\newcommandx{\braces}[4]{%
\ifstrequal{#3}{normal}{#1#4#2}{%
\ifstrequal{#3}{auto}{\left#1#4\right#2}{%
\ifstrequal{#3}{opauto}{\opleft#1#4\opright#2}{%
#3#1#4#3#2}}}%
}
\newcommandx{\opannot}[3][3=\downarrow]{\stackrel{\mathclap{\substack{#1 \\ #3 \vspace{2pt}}}}{#2}}
\newcommandx{\lineannot}[3][3=\rightarrow]{\mathllap{\boxed{\text{\textsmaller{#1}}} #3} #2}
\newcommandx{\multilineannot}[4][4=\rightarrow]{\mathllap{\boxed{\parbox{#1}{\RaggedRight\textsmaller{#2}}} #4} #3}
\newcommand{\N}{\mathbb{N}} % natural numbers
\newcommand{\R}{\mathbb{R}} % real numbers
\newcommand{\suchthat}[1][normal]{\ifstrequal{#1}{normal}{\mid}{#1|}} % seperator in sets (#1op = size)

\newcommandx{\intvcl}[3][1=normal]{\braces{[}{]}{#1}{#2, #3}} % closed interval (#1op=size, #2=left bound, #3=right bound)
\newcommandx{\intvop}[3][1=normal]{\braces{(}{)}{#1}{#2, #3}} % open interval
\newcommandx{\intvclop}[3][1=normal]{\braces{[}{)}{#1}{#2, #3}} % half-open interval (right)
\newcommandx{\intvopcl}[3][1=normal]{\braces{(}{]}{#1}{#2, #3}} % half-open interval (left)
\DeclareMathOperator{\sign}{sign}
\newcommandx{\abs}[2][1=normal]{\braces{\lvert}{\rvert}{#1}{#2}} % absolute value
\newcommandx{\ceil}[2][1=normal]{\braces{\lceil}{\rceil}{#1}{#2}} % ceil
\newcommandx{\floor}[2][1=normal]{\braces{\lfloor}{\rfloor}{#1}{#2}} % floor
\newcommandx{\round}[2][1=normal]{\braces{[}{]}{#1}{#2}} % round
\newcommandx{\der}[1]{D^{#1}} % differential operator (#1 = multiindex)
\newcommandx{\gradient}{\nabla} % gradient
\newcommandx{\partder}[4][1={},4={}]{\frac{\partial^{#4} #2}{\partial #3^{#4}}\ifargdef{#1}{\Big|_{#1}}} % partial derivative (#1=point of evaluation, #2=function, #3=variable, #4=order)
\newcommandx{\integ}[4][1={},2={}]{\int_{#1}^{#2} #3 \, #4} % integral (#1op=lower bound, #2op=upper bound, #3=integrand, #4=differential form)
\newcommandx{\asympffaster}[2][1=normal]{o\braces{(}{)}{#1}{#2}} % asymptotically faster (proper) (#1op=size)
\newcommandx{\asympfaster}[2][1=normal]{O\braces{(}{)}{#1}{#2}} % asymptotically faster
\newcommandx{\asympeq}[2][1=normal]{\Theta\braces{(}{)}{#1}{#2}} % asymptotically equal
\newcommandx{\asympsslower}[2][1=normal]{\omega\braces{(}{)}{#1}{#2}} % asymptotically slower (proper)
\newcommandx{\asympslower}[2][1=normal]{\Omega\braces{(}{)}{#1}{#2}} % asymptotically slower

\newcommandx{\norm}[2][1=normal]{\braces{\|}{\|}{#1}{#2}} % norm
\renewcommandx{\sp}[3][1=normal]{\braces{\langle}{\rangle}{#1}{#2, #3}} % inner product (#1op=size, #2=left, #3=right)
\newcommandx{\End}[2][2={}]{\mathcal{L}\opleft( #1 \ifargdef{#2}{, #2} \opright)} % endomorphism (#1=from, #2op=to)
\DeclareMathOperator{\spann}{\operatorname{span}} % span
\newcommandx{\measure}[2][1=normal]{\operatorname{vol}\braces{(}{)}{#1}{#2}} % Lebesgue-measure/volume of a set
\DeclareMathOperator{\supp}{supp} % support
\newcommandx{\Leb}[3][1={},3=normal]{L^{#2}\ifargdef{#1}{\braces{(}{)}{#3}{#1}}{}} % Lebesgue spaces (#1op=set, #2=exponent)
\newcommandx{\Lebnorm}[4][1=normal,3={2},4={}]{\norm[#1]{#2}_{#3}} % Lebesgue norm (#1op=size, #2=content, #3op=exponent, #4op=set)
\renewcommandx{\l}[3][1={},3=normal]{\ell^{#2}\ifargdef{#1}{\braces{(}{)}{#3}{#1}}} % lp sequence spaces (#1op=set, #2=exponent)
\newcommandx{\lnorm}[4][1=normal,3={2},4={}]{\norm[#1]{#2}_{#3}} % lp norm (#1op=size, #2=content, #3op=exponent, #4op=set)
\newcommandx{\Smooth}[4][1={},3={},4=normal]{C_{#3}^{#2}\ifargdef{#1}{\braces{(}{)}{#4}{#1}}} % space of differentiable functions (#1op=set, #2=order, #3op=modifier)
\newcommandx{\Schwartz}[2][1={},2=normal]{\mathscr{S}\ifargdef{#1}{\braces{(}{)}{#2}{#1}}} % space of Schwartz functions
\newcommandx{\Schwartzpoly}[2][1=normal]{\braces{\langle}{\rangle}{#1}{\abs[#1]{#2}} } % Schwartz polynomial
\newcommandx{\Tempdistr}[2][1={},2=normal]{\mathscr{S}'\ifargdef{#1}{\braces{(}{)}{#2}{#1}}} % tempered distributions
\newcommandx{\distrinp}[3][1=normal]{\braces{\langle}{\rangle}{#1}{#2, #3}} % evaluation of a tempered distribution (#1op=size, #2=distribution, #3=Schwartz function)
\newcommandx{\ft}[3][1=default,2=auto]{
\ifstrequal{#1}{default}{\widehat{#3}}{
\ifstrequal{#1}{long}{{\braces{(}{)}{#2}{#3}}^{\wedge}}{}}} % Fourier transform (hat-notation) (#1op=long expression mode, #2op=size, #3=content)
\newcommandx{\ift}[3][1=default,2=auto]{
\ifstrequal{#1}{default}{\check{#3}}{
\ifstrequal{#1}{long}{{\braces{(}{)}{#2}{#3}}^{\vee}}{}}} % inverse Fourier transform (hat-notation) (#1op=long expression mode, #2op=size, #3=content)
\let\oldcdot\cdot
\renewcommand{\cdot}{\oldcdot\protect\nolinebreak}

\DeclareMathOperator*{\convh}{conv}

\newcommand{\gennorm}{{\|\cdot\|}}
\newcommandx{\SCPR}[2][1=normal]{\braces{\langle}{\rangle}{#1}{#2}} % norm
\DeclarePairedDelimiter\scpr{\langle}{\rangle}
\newcommand{\Prob}{\mathbb{P}}
\newcommand{\E}{\mathbb{E}}
\newcommand{\Sphere}{\mathbb{S}}

\newcommand{\MID}[1][normal]{\ifstrequal{#1}{normal}{\mid}{\,#1|\,}} % seperator in sets (#1op = size)

\DeclareMathOperator*{\op}{op}
\DeclareMathOperator*{\tr}{tr}

\newcommand{\sset}{K}

\newcommandx{\prob}[2][1={},2=normal]{\mathbb{P}\ifargdef{#1}{\braces{[}{]}{#2}{#1}}}

\newcommandx{\mean}[2][1={},2=normal]{\mathbb{E}\ifargdef{#1}{\braces{[}{]}{#2}{#1}}}
\newcommandx{\var}[2][1={},2=normal]{\mathbb{V}\ifargdef{#1}{\braces{[}{]}{#2}{#1}}}
\newcommandx{\Normdistr}[3][1=normal]{\mathcal{N}\braces{(}{)}{#1}{#2, #3}} % Normal distribution
\newcommandx{\normsubg}[2][1=normal]{\norm[#1]{#2}_{\psi_2}} % sub-Gaussian
\newcommandx{\anorm}[3][1=normal,3={\sset}]{\norm[#1]{#2}_{#3}} % atomic norm 
\newcommandx{\opnorm}[2][1=normal]{\norm[#1]{#2}_{\operatorname{op}}} % atomic norm 
\newcommandx{\pospart}[2][1=auto]{\braces{[}{]}{#1}{#2}_+}

\newcommandx{\ball}[2][1={},2={}]{B_{#1}^{#2}} % lp unit ball
\newcommand{\cone}[1]{\operatorname{cone}(#1)} % cone
\graphicspath{{images/}}

\begin{document}

% Symbolic footnotes for authors
\renewcommand*{\thefootnote}{\fnsymbol{footnote}}
\pagestyle{scrheadings}

% title
\begin{center}
	\bfseries\larger[3]{Generic Error Bounds for the Generalized Lasso \\ with Sub-Exponential Data}
\end{center}

\vspace{1\baselineskip}
\begin{addmargin}[2em]{2em}
\begin{center}
\noindent{\normalsize\textbf{Martin Genzel}\footnote{Utrecht University, Mathematical Institute, Utrecht, Netherlands (e-mail:~\href{mailto:m.genzel@uu.nl}{\texttt{m.genzel@uu.nl}}); this work was done at Technische Universit\"at Berlin.} \qquad \textbf{Christian Kipp}\footnote{Technische Universit\"at Berlin, Department of Mathematics, Berlin, Germany (e-mail:~\href{mailto:kipp@math.tu-berlin.de}{\texttt{kipp@math.tu-berlin.de}}).}}
\end{center}

% \linenumbers

% \begin{addmargin}[3em]{3em}
\vspace{1\baselineskip}
{\smaller
\noindent\textbf{Abstract.}
This work performs a non-asymptotic analysis of the generalized Lasso under the assumption of sub-exponential data.
Our main results continue recent research on the benchmark case of (sub-)Gaussian sample distributions and thereby explore what conclusions are still valid when going beyond.
While many statistical features remain unaffected (e.g., consistency and error decay rates), the key difference becomes manifested in how the complexity of the hypothesis set is measured.
It turns out that the estimation error can be controlled by means of two complexity parameters that arise naturally from a generic-chaining-based proof strategy.
The output model can be non-realizable, while the only requirement for the input vector is a generic concentration inequality of Bernstein-type, which can be implemented for a variety of sub-exponential distributions. 
This abstract approach allows us to reproduce, unify, and extend previously known guarantees for the generalized Lasso.
In particular, we present applications to semi-parametric output models and phase retrieval via the lifted Lasso.
Moreover, our findings are discussed in the context of sparse recovery and high-dimensional estimation problems.

\vspace{.5\baselineskip}

\vspace{.5\baselineskip}
\noindent\textbf{MSC 2010 classes.}
60D05, 62F30, 62F35, 90C25.

}
\end{addmargin}
\newcommand{\shortauthor}{Genzel and Kipp: The Generalized Lasso with Sub-Exponential Data}

% Normal arabic footnotes
\renewcommand*{\thefootnote}{\arabic{footnote}}
\setcounter{footnote}{0}

% \linenumbers

% \maketitle
\thispagestyle{plain}

\section{Introduction}
\label{sec:intro}

This paper is concerned with the following common inference problem in statistical learning: Let $(x_1,y_1), \dots, (x_n,y_n) \in \R^{p} \times \R$ be samples of a random input-output pair $(x,y) \in \R^{p} \times \R$, whose joint probability distribution is unknown. What information about the relationship between~$x$ and~$y$ can we retrieve only based on the knowledge of $(x_1,y_1), \dots, (x_n,y_n)$?

A classical instance of this problem is \emph{linear regression}, where $y$ depends linearly on $x$, say $y = \scpr{x,\beta_0} + \nu$ for an unknown parameter vector $\beta_0 \in \R^p$ and independent additive noise $\nu$. While the resulting task of estimating $\beta_0$ is nowadays fairly well understood in the low-dimensional regime $n \geq p$, it is still subject of ongoing research in the high-dimensional regime $n \ll p$.
In the latter scenario, it is indispensable to impose additional conditions on the input-output model. A typical assumption is that $\beta_0$ belongs to a known, convex \emph{hypothesis set} $\sset \subset \R^p$ that is of low complexity in a certain sense.
In such a model setup, a natural estimation procedure is based on solving the \emph{generalized Lasso}:\footnote{We adopt the common term `generalized Lasso' from the literature (e.g., see \citep{pv16}), as a tribute to the original \emph{Lasso} estimator introduced by \citet{tib96}, where the hypothesis set corresponds to a scaled $\l{1}$-ball, serving as a convex relaxation of \emph{sparse} parameter vectors. Taking the viewpoint of statistical learning, \eqref{LS_K} is a specific instance of \emph{(constraint) empirical risk minimization}, but this terminology appears somewhat too general for the purpose of this paper.}
\begin{equation} \label{LS_K} \tag{LS$_K$}
	\min_{\beta \in K} \ \tfrac{1}{n} \sum_{i=1}^n (y_i-\scpr{x_i,\beta})^2.
\end{equation}
The popularity of Lasso-type estimators is due to several desirable properties. Perhaps most importantly, many efficient algorithmic implementations are available for \eqref{LS_K} due to the convexity of $\sset$ (e.g., see \citep{ehjt04,zh05,tib11}), accompanied by the suitability for a statistical analysis due to its simple variational formulation (e.g., see the textbooks \citep{bg11,fh13,htw15}).
A more astonishing feature of the generalized Lasso \eqref{LS_K} is its ability to deal with \emph{non-linear} relations between~$x$ and~$y$. In fact, inspired by a classical result of \citet{bri82}, a recent work of \citet{pv16} shows that for Gaussian input vectors, \eqref{LS_K} yields a consistent estimator for \emph{single-index models}, i.e., $y = f(\scpr{x, \beta_0})$ with an unknown, non-linear distortion function $f : \R \to \R$.
This finding has triggered a lot of related and follow-up research, e.g., see~\citep{tah15,os16,gen16,tr19,gmw18,so19,tr18,gj19}. We note that these works form only a small fraction of a whole research area on non-linear observation models, lying at the interface of statistics, learning theory, signal processing, and compressed sensing. A comprehensive review of the literature goes beyond the scope of this paper, and we refer the reader to \citep[Sec.~4.2]{gen19} and the references therein for more details in that regard.

The present work is inspired by the general framework developed in \citep{gen19}, which enables a theoretical analysis of \eqref{LS_K} for a large class of \emph{semi-parametric observation models} (see also the technical report \citep{gen18}).
More specifically, we will not assume an explicit functional relationship between $x$ and $y$, such as for single-index models; note that a similar viewpoint is taken by \citet{so19}, who refrain from a \emph{realizable model} connecting the input and output. Adopting this abstract setup, we intend to address the following parameter estimation problem:
\begin{problem}\label{intro:problem}
	Under what conditions is the generalized Lasso \eqref{LS_K} capable of estimating a certain \emph{target vector} $\beta^\natural \in K$ that carries information about the relationship between~$x$ and~$y$? What is the impact of the sample size $n$ and the complexity of the (convex) hypothesis set $K \subset \R^p$?
\end{problem}
For the moment, it is convenient to assume that the vector which provides the desired `information' is an \emph{expected risk minimizer}, i.e., we set $\beta^\natural \coloneqq \beta^*$ where $\beta^* \in K$ is a solution to the expected risk minimization problem (on $K$):
\begin{equation}\label{LS_K_exp}
	\min_{\beta \in K} \ \E[(y-\scpr{x,\beta})^2].
\end{equation}
Indeed, this simplification reduces the above problem to a well-known challenge in statistical learning theory, namely finding the best possible (linear) predictor of $y$ by empirical risk minimization.\footnote{To simplify the presentation even further, one might simply assume a noisy linear model $y = \scpr{x, \beta_0} + \nu$. In this case, the key messages of our main results remain valid; in particular, our conclusions on heavier tailed input data are also of interest to high-dimensional linear regression.}
However, we wish to emphasize that the absolute magnitude of the prediction error is only of minor importance to our approach, since the predictive capacity of the Lasso is likely to be poor unless $y$ follows a linear model.
On the other hand, one can still hope for a satisfactory outcome in the sense of Problem~\ref{intro:problem}, which explains why all guarantees presented in this paper concern the (parameter) estimation error.
More details on relevant scenarios, where~$\beta^\natural$ is not necessarily equal to the expected risk minimizer, are discussed later in the context of semi-parametric (non-linear) output models; see Subsection~\ref{subsec:discussion:semipara}.

The first author's dissertation~\citep{gen19} gives a far-reaching answer to Problem~\ref{intro:problem} for \mbox{\emph{(sub-)Gaus\-sian}} input vectors---an assumption that is made in almost all of the above-mentioned works on the generalized Lasso and related estimators.
However, the setup of \cite{gen19} does not adequately address output models with even non-linearities, most prominently, the \emph{phase retrieval problem} where $y = |\scpr{x, \beta_0}|$.\footnote{More specifically, if $x$ is standard Gaussian, then \eqref{LS_K} becomes a useless estimator in this case, since the expected risk minimizer $\beta^*$ in \eqref{LS_K_exp} is just the zero vector.}
The initial motivation of the present article was to address this fundamental shortcoming by applying the ``phase lift trick'' \citep{cesv13,csv13} to the Lasso.
A key feature of the lifted Lasso is that the input vector $x \in \R^p$ is replaced by the tensor product $xx^T$, see Subsection~\ref{subsec:discussion:liftedlasso} for details. 
If $x$ is an isotropic sub-Gaussian random vector, then $xx^T$ is a (generally anisotropic) \emph{sub-exponential} random matrix satisfying a mixed-tail inequality of Bernstein type (whose precise form depends on the exact assumptions on~$x$). 
The strategy we follow here is tailored to this problem, but abstracts away from all irrelevant details.
Therefore, going clearly beyond non-asymptotic error bounds for phase-retrieval-like models, we obtain a general solution to Problem~\ref{intro:problem} if the sample data have sub-exponential tails.
Indeed, research on this subject is still in its infancy, and we intend to make progress by presenting a unified approach.
This particularly includes the derivation of new and extension of known results on statistical estimation with sub-exponential data (see Section~\ref{sec:discussion}).
Before outlining our general approach in Subsection~\ref{subsec:intro:contrib}, we would like to familiarize the reader with the setup of this paper by presenting a prototypical estimation guarantee for sub-exponential input vectors.

\subsection{A Simple Error Bound for Sub-Exponential Input Vectors}
\label{subsec:intro:simple}

Let us begin with the formal definition of sub-Gaussian and sub-exponential random variables:
\begin{definition}[Sub-Gaussian/sub-exponential random variables]\label{def_sg_se}
For $\alpha \in \{1,2\}$, we define the \emph{exponential Orlicz norm} of a random variable $Z: \Omega \to \R$ by\footnote{We do not explicitly mention the underlying probability space here. In fact, our analysis does not require any treatment of measure theoretic issues and we simply assume that the probability space is rich enough to model all random quantities and processes that we are interested in.}
\begin{equation}
	\norm{Z}_{\psi_\alpha} \coloneqq \inf \Big\{ t>0 \MID[\big] \E\Big[\exp\Big(\tfrac{|Z|^\alpha}{t^\alpha}\Big)\Big] \leq 2 \Big\}.
\end{equation}
The \emph{exponential Orlicz space} $L_{\psi_\alpha}$ is then denoted by
\begin{equation}
	L_{\psi_\alpha} \coloneqq \{Z: \Omega \to \R \MID \norm{Z}_{\psi_\alpha}<\infty\}.
\end{equation}
The elements of the exponential Orlicz spaces $L_{\psi_1}$ and $L_{\psi_2}$ are called \emph{sub-exponential} and \emph{sub-Gaussian random variables}, respectively.
\end{definition}
The notions of sub-exponentiality and sub-Gaussianity impose restrictions on the tails of a random variable, which must not be ``too heavy''. This intuition gives rise to several equivalent versions of Definition~\ref{def_sg_se}, which are summarized in Proposition~\ref{prop_sg_se_properties} in Appendix~\ref{sec:app:facts}; for a more detailed introduction, we refer to \citep[Chap.~2~\&~3]{ver18}.
%The role model of sub-Gaussian random variables is clearly the normal distribution, with which they share many important properties, especially ``rotation invariance'' (in the sense of \emph{Hoeffding's inequality}, see Theorem~\ref{thm_hoeff_ineq}). The latter does not remain valid for the larger class of sub-exponential random variables, for which one may show a version of \emph{Bernstein's inequality} instead (see Theorem~\ref{thm_bern_ineq}). This fact particularly indicates that working with sub-exponential sample data is more challenging than in the sub-Gaussian case.

The definition of sub-Gaussian and sub-exponential random vectors is characterized by their one-dimensional marginals (i.e., projections onto one-dimensional subspaces):
\begin{definition}[Sub-Gaussian/sub-exponential random vectors]\label{def_sg_se_randvect}
For a random vector $x \in\nobreak \R^p$ and $\alpha \in \{1,2\}$, we set
\begin{equation}
	\norm{x}_{\psi_\alpha} \coloneqq \sup_{v \in \Sphere^{p-1}} \norm{\scpr{x,v}}_{\psi_\alpha}.
\end{equation}
If $\norm{x}_{\psi_2}<\infty$, we say that $x$ is \emph{(uniformly) sub-Gaussian}, and if $\norm{x}_{\psi_1}<\infty$, we say that $x$ is \emph{(uniformly) sub-exponential}.
\end{definition}

The following result states a non-asymptotic error bound for the generalized Lasso \eqref{LS_K} with sub-exponential input vectors. 
Its proof is provided in Subsection~\ref{subsec:proofs:prop_intro}, being a ``by-product'' of one of our main results, Corollary~\ref{cor_eg_global} in Subsection~\ref{subsec:mainresults:globalandconic}.
For the sake of simplicity, we restrict ourselves to a polytopal hypothesis set $\sset$ here, as this allows for explicit bounds on the complexity parameters.
Moreover, it is worth emphasizing that for linear models, i.e., if $y = \scpr{x, \beta_0}$, we simply obtain an estimation guarantee for $\beta^* = \beta_0$.
\begin{proposition}\label{prop_intro}
Let $(x,y) \in \R^{p} \times \R$ be a joint random pair such that $y \in \R$ is sub-exponential and $x \in \R^p$ is isotropic and sub-exponential with $\norm{x}_{\psi_1} \leq \kappa$ for some $\kappa > 0$.
Let $K \subset \R^p$ be a convex polytope with $D$ vertices and Euclidean diameter $\Delta_2(K)$, and let $\beta^* \in K$ be the expected risk minimizer on $K$, i.e., a solution to \eqref{LS_K_exp}. Finally, let the observed sample pairs $(x_1,y_1),\dots,(x_n,y_n) \in \R^{p} \times \R$ be independent copies of $(x,y)$. 
Then there exists a universal constant $C>0$ such that for every $u \geq 8$, the following holds true with probability at least $1-5\exp(-C\cdot u^2)-2\exp(-C \cdot \sqrt{n})$:
If the sample size obeys
\begin{equation}\label{prob_eg_condition_n}
	n \gtrsim \Big(\kappa^{10}\cdot\Delta_2(K)\cdot\Big(\tfrac{\log(D)}{\sqrt{n}}+\sqrt{\log(D)}\Big)+ \kappa^{6} \cdot u\Big)^2,
\end{equation}
then every minimizer $\hat{\beta}$ of \eqref{LS_K} satisfies
\begin{equation}\label{pro_eg_bound_t}
	\norm{\hat{\beta}-\beta^*}_2 \lesssim \kappa^{18} \cdot \max\{1,u^2 \cdot \sigma(\beta^*) \}\cdot  \frac{\sqrt{\kappa \cdot\Delta_2(K)\cdot\log(D)}}{n^{1/4}},
\end{equation}
where $\sigma(\beta^*) \coloneqq \norm{y-\scpr{x,\beta^*}}_{\psi_1}$.
\end{proposition}
Informally speaking, Proposition~\ref{prop_intro} shows that estimation of the expected risk minimizer succeeds with overwhelmingly high probability as long as $n \gg \Delta_2(K)^2\cdot \log(D)^2$. Such a statement is particularly appealing to high-dimensional problems such as sparse recovery.
Another remarkable conclusion is that the estimator \eqref{LS_K} essentially performs as well as if the sample data were sub-Gaussian (cf.~\citep[Thm.~4.3]{gen19}).
Our main results in Section~\ref{sec:mainresults} confirm this observation in much greater generality, but they will also reveal several important differences to the sub-Gaussian case;
first and foremost, we will be concerned with defining appropriate complexity measures for~$K$, which do not explicitly appear in the polytopal setting of Proposition~\ref{prop_intro}.
In this respect, it is important to note that there are relevant special cases of sub-exponential vectors (e.g., those with independent coordinates) for which the above estimate is too pessimistic and can be improved. The elaboration of this aspect is a key concern of this article and motivates the introduction of a generic tail condition that takes the underlying ``geometry'' of the problem into account.
Apart from this, let us also emphasize that the simplifications of Proposition~\ref{prop_intro} come along with a suboptimal behavior regarding, (a), the error decay rate $O(n^{-1/4})$, (b), the sub-exponential parameter~$\kappa$, and (c), the model deviation parameter $\sigma(\beta^*)$.

To the best of our knowledge, Proposition~\ref{prop_intro} is a new result, but it bears resemblance with a recent finding of \citet[Thm.~3.4]{so19}, who consider a similar model setup with sub-exponential input vectors.
Their analysis focuses on the projected gradient descent method, as an algorithmic implementation of \eqref{LS_K}, and is therefore related to our estimation guarantees; see Subsection~\ref{subsec:discussion:se} for a more detailed comparison.

\subsection{Contributions and Overview}
\label{subsec:intro:contrib}

The main purpose of this work is to shed more light on the estimation capacity of the generalized Lasso \eqref{LS_K} when the sample data are not sub-Gaussian.
While Proposition~\ref{prop_intro} already gives a first glimpse into the prototypical situation of sub-exponential input vectors, we intend to address this problem in a more systematic and abstract way (cf.~Problem~\ref{intro:problem}).
At the heart of our statistical analysis stands the so-called \emph{generic Bernstein concentration}, which is introduced in Subsection~\ref{subsec:mainresults:preliminaries} (see Definition~\ref{def_gen_bern}). 
This concept is the outcome of a somewhat uncommon proof strategy: Instead of assuming a specific (sub-exponential) distribution for $x$, we study the associated excess risk of \eqref{LS_K} in an abstract sense, relying on an advanced \emph{generic chaining} argument due to \citet{men16}.
Consequently, the key step of our approach is to understand the increment behavior of the underlying stochastic processes, and in fact, this precisely leads to generic Bernstein concentration as a natural condition for $x$. In that way, we are able to explore \eqref{LS_K} for a whole class of input distributions and thereby to refine the assumption of uniform sub-exponentiality in Proposition~\ref{prop_intro}.
%among which (uniformly) sub-exponential vectors are just a special case.
Another important outcome of our analysis are two general complexity parameters for the hypothesis set $K$ (see Definition~\ref{def_q_complexity} and~\ref{def_m_complexity}), which are compatible with the notion of generic Bernstein concentration.

With these preliminaries at hand, we formulate our main result in Subsection~\ref{subsec:mainresults:localerrorbound} (see Theorem~\ref{thm_eg_local}), which provides a novel, non-asymptotic error bound for \eqref{LS_K} under generic Bernstein concentration.
However, a direct application of this guarantee to specific model situations is not always straightforward, since the aforementioned complexity parameters are of local nature, implicitly depending on the desired precision level.
For this reason, we present two more easily accessible corollaries of Theorem~\ref{thm_eg_local} in Subsection~\ref{subsec:mainresults:globalandconic}. These results are based on simplified complexity parameters (see Definition~\ref{def_conic_q_m_complexity} and~\ref{def_global_q_m_complexity}, respectively), but come with the price of looser error bounds and sample-size conditions.
%Apart from that, it is worth pointing out that our proof techniques are amenable to several extensions that we have not elaborated for the sake of clarity; most notably, one may measure the estimation error by an arbitrary semi-norm or replace the squared loss in \eqref{LS_K} by a different convex loss function (see Appendix~\ref{subsec:app:rem_main}\ref{rmk_local_extensions}).

While the purpose of Section~\ref{sec:mainresults} is to develop a unified analysis for the generalized Lasso \eqref{LS_K}, Section~\ref{sec:discussion} is devoted to various applications and examples of our findings.
We begin with a brief discussion on semi-parametric modeling in Subsection~\ref{subsec:discussion:semipara}, demonstrating how our general results may be applied to specific parameter estimation problems.
This is followed by several relevant examples of generic Bernstein concentration (see Subsection~\ref{subsec:discussion:sg}--\ref{subsec:discussion:se}), leading to off-the-shelf guarantees for \eqref{LS_K} with sub-exponential and sub-Gaussian sample data; these parts also provide a comparison to related approaches in the literature.
In Subsection~\ref{subsec:discussion:liftedlasso}, we then revisit our motivating case study on the \emph{lifted Lasso} for phase-retrieval-like problems---a scenario where sub-exponential distributions arise naturally.
Finally, Subsection~\ref{subsec:discussion:compl} contains a more detailed discussion of the complexity parameters from Section~\ref{sec:mainresults}. In this context, it will become clearer that measuring complexity beyond sub-Gaussianity is a delicate issue and comes along with unexplored difficulties.
Nevertheless, we are able to establish simple bounds in the prototypical case of $\l{1}$-balls, making our error bounds applicable to high-dimensional estimation and sparse recovery.
Some concluding remarks are made in Section~\ref{sec:conclusion}.

\subsection{Differentiation from Previous Works}

Apart from enabling heavier-tailed data, a crucial feature of generic Bernstein concentration is that it does \emph{not require any type of isotropy}.
Instead, the ``geometric'' behavior of the input vectors is captured by selecting two appropriate semi-norms (see Definition~\ref{def_gen_bern}).
This relaxation is key to the applicability of our results, as it allows us to handle structured, but anisotropic input vectors, such as arising in the phase lift approach (see Subsection~\ref{subsec:discussion:liftedlasso}).
It is not clear to us how this important challenge could be addressed with other techniques, especially those suggested in \citep{gen19} and related articles mentioned above.\footnote{Unfortunately, the idea of \citep{gen19} to capture anisotropic and correlated structures through \emph{mixing matrices} is no remedy in our situations of interest, e.g., phase-lifted input vectors.}
The concept of generic Bernstein concentration therefore also presents a novel and systematic solution to this open problem, which is arguably one of the most significant achievements compared to previous works.

We close this part with another clarification: The present article is concerned with the generalized Lasso \eqref{LS_K} when the sample data are heavier tailed than sub-Gaussian, in particular, the underlying distribution may be \emph{unbounded}.
An alternative strategy is to first truncate the raw data at an appropriate threshold and then to apply \eqref{LS_K} or a similar estimator. 
In fact, the latter approach is quite common in practice, but it also facilitates a theoretical study due to the boundedness of the involved random variables, e.g., see~\citep{ybkl17,ybl17,ybwl17,gmw18,wei18} for related results on non-linear observation models.
However, the (concentration-based) machinery for bounded sample data is certainly not applicable to the model setup of the present paper.
Instead, we rather follow the conceptual ideas of \citet{men15,men18}, who points out the downsides of the bounded framework and develops a general theory for heavy-tailed problems.
%Having said this, we stress that although inspired by Mendelson's works, our findings are not a by-product of those. Indeed, we focus on much more concrete model assumptions, implying quite different mathematical challenges.
Although our analysis is concerned with more specific model assumptions, it is just general enough to allow for a rigorous understanding of estimation with sub-exponential data, thereby unifying and improving a series of previously known results from the literature (see Section~\ref{sec:discussion}). Thus, to a certain degree, our work can be seen as a ``connecting piece'' between such highly customized approaches and the abstract theory of \citet{men15,men18}.

\subsection{Notation}
\label{subsec:intro:notation}

The letter $C$ is reserved for constants, whose values could change from time to time, and we say that $C$ is \emph{universal} if its value does not depend on any other involved parameter. If an inequality holds true up to a universal constant $C > 0$, we usually write $A \lesssim B$ instead of $A \leq C \cdot B$; the notation $A \asymp B$ means that both $A \lesssim B$ and $B \lesssim A$ hold true. Furthermore, the \emph{positive part} of a real number $s \in \R$ is denoted by $\pospart{s} \coloneqq \max\{s, 0\}$.

The \emph{cardinality} of a finite set $I$ is denoted by $|I|$. The \mbox{$j$-th} entry of a vector $v \in \R^p$ is denoted by~$v_j$ and the \emph{support} of $v$ is defined as $\supp(v) \coloneqq \{ j \suchthat v_j \neq 0 \}$. The cardinality of $\supp(v)$ is referred to as the \emph{sparsity} of $v$ and we write $\norm{v}_0 \coloneqq |\supp(v)|$.
For $1 \leq q \leq \infty$, we denote the \emph{$\l{q}$-norm} on $\R^p$ by $\gennorm_{q}$ and the associated \emph{unit ball} by $B_q^p$.
The \emph{Euclidean unit sphere} is given by $\Sphere^{p-1} \coloneqq \{ v \in \R^p \MID \norm{v}_2 = 1 \}$. 
The \emph{Frobenius norm} is denoted by $\gennorm_F$ and the \emph{spectral norm} by $\gennorm_{\op}$. We write $I_p \in \R^{p \times p}$ for the \emph{identity matrix}. 

Let $L \subset \R^p$. By $\spann(L)$, $\cone{L}$, and $\convh(L)$, we denote the \emph{linear hull}, \emph{conic hull}, and \emph{convex hull}, respectively.
The \emph{diameter} of $L$ with respect to a (pseudo-)metric $d$ is defined as $\Delta(L) \coloneqq \sup_{v_1,v_2 \in L} d(v_1,v_2)$.

Let $x$ be a random vector in $\R^p$. We say that $x$ is \emph{centered} if $\E[x] = 0$, it is \emph{symmetric} if $x$ has the same distribution as $-x$, and it is \emph{isotropic} if $\E[xx^T] = I_p$. The $L^q$-norm of a real-valued random variable $Z$ is $\norm{Z}_{L^q} \coloneqq (\E[\abs{Z}^q])^{1/q}$ for $1 \leq q < \infty$.
Moreover, we write $g \sim \mathcal{N}(0,I_p)$ if $g$ is a \emph{standard Gaussian random vector} in $\R^p$.

For $v \in \R^p$, we use the notation $v^*$ for the linear functional $\scpr{\cdot,v}$, i.e., $v^*$ is the image of~$v$ under the Riesz isomorphism; analogously, we write $A^*$ for the image of a subset $A \subset \R^p$ under the Riesz isomorphism. Furthermore, if $\R^p$ is equipped with a probability measure $\mu$, we can interpret $v^*$ as a random variable, i.e., $v^* = \scpr{x, v}$, where $x$ is distributed according to $\mu$. 
In particular, we have that $\norm{v^*}_{L^q}^q = \E[|\scpr{v,x}|^q]$ for $1 \leq q < \infty$.

\section{Main Results}
\label{sec:mainresults}

This section presents the main results of this work. We begin with several technical preliminaries in Subsection~\ref{subsec:mainresults:preliminaries}, including the central concept of generic Bernstein concentration (see Definition~\ref{def_gen_bern}) as well as the related complexity parameters (see Definition~\ref{def_q_complexity} and~\ref{def_m_complexity}).
The most general estimation guarantee is then formulated and discussed in Subsection~\ref{subsec:mainresults:localerrorbound} (see Theorem~\ref{thm_eg_local}).
This is followed by two corollaries in Subsection~\ref{subsec:mainresults:globalandconic}, employing simplified variants of our complexity parameters.
Note that all proofs for this section are postponed to Section~\ref{sec:proofs}.

\subsection{Preliminaries and Generic Bernstein Concentration}
\label{subsec:mainresults:preliminaries}

An error bound for the generalized Lasso \eqref{LS_K} is a statement about the minimizer of the following function:
\begin{definition}[Empirical risk, excess risk]\label{def_excrisk}
The objective function minimized in \eqref{LS_K}, i.e.,
\begin{equation}
	\bar{\mathcal{L}}(\beta) \coloneqq \tfrac{1}{n} \sum_{i=1}^n (y_i-\scpr{x_i,\beta})^2,
\end{equation}
is called the \emph{empirical risk} of $\beta \in K$. Given $\beta, \beta^\natural \in K$, we call
\begin{equation}
	\mathcal{E}(\beta,\beta^\natural) \coloneqq \bar{\mathcal{L}}(\beta)-\bar{\mathcal{L}}(\beta^\natural)
\end{equation}
the \emph{excess risk} of $\beta$ over $\beta^\natural$.
\end{definition}
Since the map $\beta \mapsto \bar{\mathcal{L}}(\beta)$ depends on the random pairs $(x_i,y_i)$, it can be seen as a stochastic process on the hypothesis set $K$. If the excess risk is strictly positive on a subset of $K$, the minimizer must be outside of this subset. In other words, we can localize the empirical risk minimizer in a certain set $L\subset K$ if we have a positive lower bound for the excess risk on $K \setminus L$ (see Fact~\ref{fact_excess} below). 
A powerful technique for proving such lower bounds is \emph{generic chaining} for stochastic processes (see \citep{tal06,tal14}). 

The following definition introduces a generic concentration inequality for linear functions on the parameter space, which leads to an increment condition for the involved stochastic processes. Based on this condition, we will use chaining arguments to derive a generic error bound for \eqref{LS_K} (see Theorem~\ref{thm_eg_local} and its proof in Subsection~\ref{subsec:proofs:thm_eg_local}). Estimation guarantees for specific classes of input vectors can be then obtained by considering concrete instances of this condition (see Section~\ref{sec:discussion}).
\begin{definition}[Generic Bernstein concentration]
\label{def_gen_bern}
Let $x\in \R^p$ be a random vector and let $\gennorm_g$ and $\gennorm_e$ be two semi-norms on $\R^p$. We say that $x$ exhibits \emph{generic Bernstein concentration} with respect to $(\gennorm_g,\gennorm_e)$ if for every $v \in \R^p$ and every $t \geq 0$, we have that
\begin{equation} \label{def_gen_bern_conc_ineq}
\Prob (|\scpr{x,v}| \geq t ) \leq 2 \exp\Big(- \min\Big\{\tfrac{t^2}{\norm{v}_g^2},\tfrac{t}{\norm{v}_e}\Big\}\Big),
\end{equation}
where $\exp(-\infty) \coloneqq 0$ and 
\begin{equation}
	\frac{t}{0} \coloneqq \begin{cases}\infty & \text{for } t>0, \\ 0 & \text{for } t=0.\end{cases}
\end{equation}
\end{definition}
The prototypical instance of generic Bernstein concentration is a centered random vector $x = (x_1, \dots, x_p) \in \R^p$ with independent, sub-exponential coordinates: indeed, such an $x$ exhibits generic Bernstein concentration with respect to $(\tfrac{R}{\sqrt{C_B}}\gennorm_2, \tfrac{R}{C_B}\gennorm_{\infty})$ where $R \coloneqq \max_{1 \leq j \leq p} \norm{x_j}_{\psi_1}$ and a universal constant $C_B > 0$. In this case, \eqref{def_gen_bern_conc_ineq} simply corresponds to the classical Bernstein's inequality (see Theorem~\ref{thm_bern_ineq}), justifying the terminology of Definition~\ref{def_gen_bern}. 
More generally, \eqref{def_gen_bern_conc_ineq} can be seen as an example of \emph{mixed-tail conditions}, which are quite common in the generic chaining literature, e.g., see~\citep[Thm.~3.5]{dir15} or \citep[Thm.~2.2.23]{tal14}. To be more specific, the semi-norm $\gennorm_g$ governs the Gaussian-like (`g') tail, while $\gennorm_e$ governs the exponential-like (`e') tail.

The central idea of generic chaining is that the expected infimum (or supremum) of a stochastic process depends on the ``size'' of the underlying index set, which is equipped with a \mbox{(pseudo-)}met\-ric that reflects the increment behavior of the stochastic process. For certain classes of canonical processes, the appropriate way of measuring the size is given by the well-known $\gamma$-functional:
\begin{definition}[\protect{$\gamma$-functional; \citep[Def.~2.2.19]{tal14}}]\label{def_gamma}
	Let $L$ be a set equipped with a pseudo-metric $d$. We call a sequence $(\mathcal{A}_s)_{s \in \N}$ of partitions\footnote{As usual, by a partition of $L$, we mean a family of pairwise disjoint, non-empty subsets of $L$ whose union is $L$.} of $L$ an \emph{admissible partition sequence} if $|\mathcal{A}_0|=1$ and $|\mathcal{A}_s| \leq 2^{2^s}$ for $s \geq 1$ and if the sequence is increasing, i.e., for every $A \in \mathcal{A}_{s+1}$ there is some $B \in \mathcal{A}_{s}$ with $A \subset B$. For $\alpha \in \{1,2\}$, we set
	\begin{equation}
		\gamma_\alpha(L,d) \coloneqq \inf \ \sup_{v \in L} \sum_{s \in \N} 2^{s/\alpha} \Delta(A_s(v)),
	\end{equation}
	where $A_s(v)$ is the unique set in $\mathcal{A}_s$ containing $v$ and the infimum is taken over all admissible partition sequences. In this work, we will only deal with pseudo-metrics induced by semi-norms. Hence, we may write $\gamma_\alpha(L,\gennorm) \coloneqq \gamma_\alpha(L,d_{\gennorm})$ where $d_{\gennorm}$ is the pseudo-metric induced by a semi-norm $\gennorm$.
\end{definition}

Returning to the issue of finding an error bound for \eqref{LS_K}, let us now fix some precision level $t > 0$ and an arbitrary \emph{target vector} $\beta^\natural \in K$ (see also Problem~\ref{intro:problem}). At the present level of abstraction, it is beneficial to leave the notion of the `estimation error' as general as possible. For the sake of mental convenience, $\beta^\natural$ can be seen as a desirable outcome of an estimation procedure (e.g., the expected risk minimizer on $K$), but this interpretation is mathematically irrelevant. The error measure that will concern us in this section is the Euclidean distance $\norm{\hat{\beta}-\beta^\natural}_2$, where $\hat{\beta} \in K$ is the estimate of the generalized Lasso, i.e., a minimizer of \eqref{LS_K}. Since $\mathcal{E}(\cdot,\beta^\natural)$ is a convex function (on $K$), one can make use of the following basic, yet important fact:
\begin{fact}\label{fact_excess}
	Let $K \subset \R^p$ be a convex set. For $\beta^\natural \in K$ and $t > 0$, we set
	\begin{equation}
		K_{\beta^\natural,t} \coloneqq \{\beta \in K \MID \norm{\beta-\beta^\natural}_2=t\} = K \cap (t\Sphere^{p-1} + \beta^\natural).
	\end{equation}
	If $\mathcal{E}(\beta,\beta^\natural)>0$ for all $\beta \in K_{\beta^\natural,t}$, then every minimizer $\hat{\beta}$ of \eqref{LS_K} satisfies the error bound $\norm{\hat{\beta}-\beta^\natural}_2<t$.
\end{fact}

Consequently, it suffices to control $\mathcal{E}(\cdot,\beta^\natural)$ on the spherical subset $K_{\beta^\natural,t}$ of radius $t$ around~$\beta^\natural$.
To this end, we loosely follow the approach of \citet{men15} and decompose the excess risk as follows:
\begin{align}
\mathcal{E}(\beta,\beta^\natural) &=\tfrac{1}{n}\sum_{i=1}^n (y_i - \scpr{x_i,\beta})^2-\tfrac{1}{n}\sum_{i=1}^n(y_i - \scpr{x_i,\beta^\natural})^2 \\*
	&=\underbrace{\tfrac{1}{n} \sum_{i=1}^n\scpr{x_i,\beta-\beta^\natural}^2}_{\eqqcolon \mathcal{Q}(\beta-\beta^\natural)} + \underbrace{\tfrac{2}{n} \sum_{i=1}^n(\scpr{x_i,\beta^\natural}-y_i) \scpr{x_i, \beta-\beta^\natural}}_{\eqqcolon \mathcal{M}(\beta,\beta^\natural)}. \label{decomp_excess}
\end{align}
In this decomposition, the excess risk is expressed as a sum of two empirical processes $\mathcal{Q}(\beta-\beta^\natural)$ and $\mathcal{M}(\beta,\beta^\natural)$, both indexed by $\beta \in K_{\beta^\natural,t}$, which we call the \emph{quadratic process} and the \emph{multiplier process}, respectively. Note that this corresponds to a second-order Taylor expansion of $\mathcal{E}(\cdot,\beta^\natural)=\bar{\mathcal{L}}(\cdot)-\bar{\mathcal{L}}(\beta^\natural)$:
the quadratic process is the second-order term\footnote{Since the Hessian matrix $H_{\bar{\mathcal{L}}}(\beta^\natural) \in \R^{p \times p}$ is actually independent of $\beta^\natural$ for the squared loss, the quadratic process is translation-invariant in the sense that it only depends on $\beta-\beta^\natural$. Hence, we write $\mathcal{Q}(\beta-\beta^\natural)$ rather than $\mathcal{Q}(\beta, \beta^\natural)$.}
\begin{equation}
\mathcal{Q}(\beta-\beta^\natural)=\tfrac{1}{2}(\beta-\beta^\natural)^T H_{\bar{\mathcal{L}}}(\beta^\natural)(\beta-\beta^\natural)
\end{equation}
and the multiplier process is the first-order term
\begin{equation}
\mathcal{M}(\beta,\beta^\natural)=\scpr{(\nabla \bar{\mathcal{L}})(\beta^\natural),\beta-\beta^\natural}.
\end{equation}
With this notation at hand, the desired uniform lower bound $\mathcal{E}(\beta,\beta^\natural)>0$ amounts to the event that $\mathcal{Q}(\beta-\beta^\natural)$ dominates $-\mathcal{M}(\beta,\beta^\natural)$ on the whole index set $K_{\beta^\natural,t}$.

Based on the $\gamma$-functional, we now define two general complexity parameters, which are adapt\-ed to the analysis of the quadratic process and the multiplier process, respectively. Both parameters are tailored to the above notion of generic Bernstein concentration and have in common that they measure the complexity of a set \emph{locally}, i.e., at a certain scale $t>0$. This reflects the fact that we are only interested in the behavior of the empirical processes on $K_{\beta^\natural,t}$ and not on the full hypothesis set $K$.
\begin{definition}[Local q-complexity]\label{def_q_complexity} Let $L \subset \R^p$ and let $\gennorm_g$ and $\gennorm_e$ be semi-norms on $\R^p$. For $t>0$, we define the \emph{local q-complexity} of $L$ at scale $t$ and sample size $n$ with respect to $(\gennorm_g,\gennorm_e)$ by
\begin{equation}
	q_{t,n}^{(g,e)}(L) \coloneqq \frac{1}{t}\inf\Big\{\tfrac{\gamma_1(S, \gennorm_e)}{\sqrt{n}}+\gamma_2(S, \gennorm_g+\gennorm_e)\MID[\big] S \subset \R^p, \convh(S) \supset L \cap t\Sphere^{p-1} \Big\}.
\end{equation}
\end{definition}
\begin{figure}
	\hspace{\fill}\begin{tikzpicture}
	\path[fill, color=lightgray] 
	(0:3) to (0:0) to (60:3) to[bend left] (0:3);
	\draw[thick]  pic[black]{carc=0:60:1.5cm};
	\draw (0:3) to (0:0) to (60:3);
	
	\node at (0:0)[left,below] {$0$};
	\node at (30:2.25)[above] {$L$};
	\node at (0:2)[below] {$L \cap t \Sphere^{p-1}$};
	\draw[->] (-4:2) to[bend right] (30:1.575);
	\begin{scope}[shift={(-0.05,0.05)}]
	\draw [decorate,decoration={brace},rotate=0] (0:0) -- (60:1.5);
	\node at (80:0.53)[left,above] {$t$};
	\end{scope}
	\node at (0:0) [circle,fill,inner sep=0.2pt]{};
	\node at (90:2)[above] {(a)};
	\node at (-12:3) [circle,inner sep=0pt]{};
	\end{tikzpicture}\hspace{\fill}
	\begin{tikzpicture}
	\path[fill, color=lightgray] 
	(2:1.3*1.5) to (-5:0.8*1.5) to (62:0.8*1.5) to (63:1.3*1.5) to (2:1.3*1.5);
	\node at (2:1.3*1.5) [circle,fill,inner sep=0.5pt]{};
	\node at (-5:0.8*1.5) [circle,fill,inner sep=0.5pt]{};
	\node at (62:0.8*1.5) [circle,fill,inner sep=0.5pt]{};
	\node at (63:1.3*1.5) [circle,fill,inner sep=0.%pt]{};
	\node at (0:0) [circle,fill,inner sep=0.5pt]{};
	\draw[thick]  pic[black]{carc=0:60:1.5cm};
	\draw (0:3) to (0:0) to (60:3);
	
	\node at (0:0)[left,below] {$0$};
	\node at (90:2)[above] {(b)};
	\node at (-12:3) [circle,inner sep=0pt]{};
	
	\end{tikzpicture}\hspace{\fill}
	\begin{tikzpicture}
	
	\node at (2:1.3*1.5) [circle,fill,inner sep=0.5pt]{};
	\node at (-5:0.8*1.5) [circle,fill,inner sep=0.5pt]{};
	\node at (62:0.8*1.5) [circle,fill,inner sep=0.5pt]{};
	\node at (63:1.3*1.5) [circle,fill,inner sep=0.5pt]{};
	\draw[lightgray] (0:3) to (0:0) to (60:3);
	
	\node at (0:0) [circle,fill,inner sep=0.5pt]{};
	\node at (0:0)[left,below] {$0$};
	\node at (30:1.7*1.5) {$S$};
	\draw[thick]  pic[lightgray]{carc=0:60:1.5cm};
	\draw[->] (30:1.5*1.5) to[bend right] (4:1.28*1.5);
	\draw[->] (30:1.5*1.5) to[bend right] (-3:0.8*1.5);
	\draw[->] (30:1.5*1.5) to[bend left] (60:0.8*1.5);
	\draw[->] (30:1.5*1.5) to[bend left] (61:1.28*1.5);
	\node at (90:2)[above] {(c)};
	\node at (-12:3) [circle,inner sep=0pt]{};
	\end{tikzpicture}\hspace{\fill}
	\caption{An illustration of the local q-complexity $q_{t,n}^{(g,e)}(L)$ from Definition~\ref{def_q_complexity}: (a)~In order to measure the complexity of $L$ locally at scale $t$, we consider the set $L \cap t \Sphere^{p-1}$. (b)~$L \cap t \Sphere^{p-1}$ is contained in the convex hull of the four points indicated by small black dots. (c)~Defining $S$ as these four points, the quantity $\frac{1}{t} \cdot \big(\gamma_1(S, \gennorm_e) / \sqrt{n}+\gamma_2(S, \gennorm_g+\gennorm_e)\big)$ is an upper bound for the local q-complexity of $L$ at scale $t$, which is defined as the infimum over all such upper bounds.}
	\label{fig:skeleton}
\end{figure}
Remarkably, we do not simply measure the size of the set $L \cap t\Sphere^{p-1}$ in Definition~\ref{def_q_complexity}, but optimize over all ``skeletons'' $S$ of this set; see Figure~\ref{fig:skeleton} for an illustration and \citep[Appx.~A]{oym19} for a related approach in the literature.
\begin{definition}[Local m-complexity]\label{def_m_complexity}
Let $L \subset \R^p$ and let $\gennorm_g$ and $\gennorm_e$ be semi-norms on~$\R^p$. For $t>0$, we define the \emph{local m-complexity} of $L$ at scale $t$ with respect to $(\gennorm_g,\gennorm_e)$ by
\begin{equation}
	m_t^{(g,e)}(L) \coloneqq \frac{1}{t}\inf\Big\{\gamma_1(S, \gennorm_e)+\gamma_2(S, \gennorm_g)\MID[\big] S \subset \R^p, \convh(S) \supset (L \cap t \Sphere^{p-1}) \cup \{0\} \Big\}.
\end{equation}
\end{definition}
It is worth noting that in the well-understood case of sub-Gaussian sample data, the \mbox{q-}com\-plex\-ity and \mbox{m-}com\-plex\-ity can be both identified with the notion of \emph{local Gaussian width}; see also Subsection~\ref{subsec:discussion:sg} for more details.
In general, however, this simple geometric interpretation is no longer valid and the behavior of both parameters is highly non-trival. We will return to this important issue later in Subsection~\ref{subsec:discussion:compl}.

In order to control the quadratic process (in Subsection~\ref{subsec:proofs:quadraticprocess}), we will apply the \emph{small-ball method}, which is a powerful tool to establish uniform lower bounds for non-negative empirical processes (see \citep{km15,men15,men18}).
For this purpose, the notion of a small-ball function is required:
\begin{definition}[{Small-ball function; \citep[p.~12995]{km15}}]\label{def_q_theta}
Let $L \subset\nobreak \R^p$ and let $x$ be a random vector in $\R^p$. For $\theta \geq 0$, we define the \emph{small-ball function}
\begin{equation}
	Q_\theta(L,x) \coloneqq \inf_{v \in L} \Prob(|\scpr{x,v}| \geq \theta).
\end{equation}
\end{definition}

Since we are aiming at an error bound relative to an arbitrary target vector $\beta^\natural \in K$, it is natural that this error bound depends on how well the associated linear hypothesis $\scpr{x,\beta^\natural}$ predicts the actual output variable $y$ (which may depend on $x$ in a non-linear way). In other words, the estimation performance of \eqref{LS_K} is also affected by the behavior of the \emph{model mismatch} $y-\scpr{x,\beta^\natural}$, measuring how much $y$ deviates from the linear model $\scpr{x,\beta^\natural}$. 
The following parameters allow us to make this precise:
\begin{definition}[Mismatch parameters]\label{def_mismatch} Given $\beta^\natural \in \R^p$ and a random pair $(x,y) \in \R^{p} \times \R$, the \emph{mismatch deviation} of $\beta^\natural$ is defined by
	\begin{equation}
	\sigma(\beta^\natural) \coloneqq \norm{y-\scpr{x,\beta^\natural}}_{\psi_1}
	\end{equation}
	and the \emph{(global) mismatch covariance} of $\beta^\natural$ by
	\begin{equation}
	\rho(\beta^\natural) \coloneqq \norm[\big]{\E\big[(y-\scpr{x,\beta^\natural}) x\big]}_2.
	\end{equation}
	Moreover, for $t \geq 0$ and $K \subset \R^p$, we define the \emph{(local) mismatch covariance} of $\beta^\natural$ at scale $t$ by
	\begin{equation}
	\rho_t(\beta^\natural) \coloneqq \sup_{v \in K^t} \SCPR[\big]{\E\big[(y-\scpr{x,\beta^\natural}) x\big], v },
	\end{equation}
	where $K^t \coloneqq \tfrac{1}{t} (K - \beta^\natural) \cap \Sphere^{p-1}$ for $t > 0$ and $K^0 \coloneqq \cone{K-\beta^\natural} \cap \Sphere^{p-1}$.
\end{definition}
As the name suggests, the mismatch covariance captures the covariance between the input vector $x$ and the model mismatch $y-\scpr{x,\beta^\natural}$. 
Inspired by linear regression problems, it is useful to think of the model mismatch as ``noise'' that perturbs the linear model $\scpr{x,\beta^\natural}$. In particular, if $\E[(y-\scpr{x,\beta^\natural}) x] = 0$, this noise is uncorrelated with all input variables (but not necessarily independent), implying that $\rho(\beta^\natural) = \rho_t(\beta^\natural) = 0$.
In contrast, the mismatch deviation measures the sub-exponential tail behavior of the model mismatch.
Note that in the noisy linear case, i.e., $y = \scpr{x,\beta^\natural} + \nu$, we simply have that $\sigma(\beta^\natural) = \norm{\nu}_{\psi_1}$.
The interested reader is referred to Appendix~\ref{subsec:app:rem_on_mc} for further remarks on the above notions of the mismatch covariance.

\subsection{A Local Error Bound for \texorpdfstring{\eqref{LS_K}}{(LS\_K)}}
\label{subsec:mainresults:localerrorbound}

Before stating the error bound, let us formally summarize our assumptions about the sampling process:
\begin{assumption}[Model setup] \label{model_setup} Let $(x,y) \in \R^{p} \times \R$ be a joint random pair where $x \in \R^p$ satisfies generic Bernstein concentration with respect to $(\gennorm_g,\gennorm_e)$ and $y \in \R$ is sub-exponential. Moreover, let $K \subset \R^p$ be a convex hypothesis set. We define the set $K^\Delta \coloneqq \spann(K-K) \cap \Sphere^{p-1}$ and assume that $x$ satisfies the small-ball condition
\begin{equation}\label{small_ball_condition}
Q_{2\tau}(K^\Delta,x) > 0
\end{equation}
for some $\tau>0$. Finally, we assume that the observed sample pairs $(x_1,y_1),\dots,(x_n,y_n)$ are independent copies of $(x,y)$.
\end{assumption}
Although it can be helpful to imagine a semi-parametric relationship between $x$ and $y$ (see Subsection~\ref{subsec:discussion:semipara}), such an assumption is not required at the current level of abstraction. Indeed, our main result, which is presented next, provides a generic error bound for the generalized Lasso \eqref{LS_K} without any specific observation model.
\begin{theorem}[General error bound for \eqref{LS_K}, local version] \label{thm_eg_local} 
	Let Assumption~\ref{model_setup} be satisfied and fix a vector $\beta^\natural \in K$. Then there exists a universal constant $C>0$ such that for every $u \geq 8$ and $t \geq 0$, the following holds true with probability at least $1-5\exp(-C\cdot u^2)-2\exp(-C \cdot \sqrt{n})$: If the sample size obeys\footnote{For the case of exact recovery, i.e., $t = 0$, the corresponding complexity parameters $q_{0,n}^{(g,e)}(K-\beta^\natural)$ and $m_0^{(g,e)}(K-\beta^\natural)$ are introduced further below in Definition~\ref{def_conic_q_m_complexity} in Subsection~\ref{subsec:mainresults:globalandconic}.}
	\begin{equation} \label{thm_eg_local_condition_n}
		n \gtrsim \bigg(\frac{q_{t,n}^{(g,e)}(K-\beta^\natural)+ \tau \cdot u}{\tau\cdot Q_{2\tau}(K^\Delta,x)}\bigg)^2
	\end{equation}
	and we have that
	\begin{equation} \label{thm_eg_local_condition_t}
		t \gtrsim \frac{1}{(\tau\cdot Q_{2\tau}(K^\Delta,x))^2}\cdot\pospart[\bigg]{\rho_t(\beta^\natural) + u^2 \cdot \sigma(\beta^\natural)\cdot \frac{m_t^{(g,e)}(K-\beta^\natural)}{\sqrt{n}}},
	\end{equation}
	then every minimizer $\hat{\beta}$ of \eqref{LS_K} satisfies $\norm{\hat{\beta}-\beta^\natural}_2 \leq t$.
\end{theorem}
The interpretation of the error bound established in Theorem~\ref{thm_eg_local} is not straightforward, since the right-hand side of \eqref{thm_eg_local_condition_t} depends on the precision level $t$ and the right-hand side of \eqref{thm_eg_local_condition_n} depends on both $t$ and $n$. 
But regardless of these implicit dependencies, the above statement has almost the same syntactic form as in the case of sub-Gaussian sample data, e.g., see~\citep[Thm.~3.6]{gen19}, and we can rely on the interpretation suggested there.
The following way of reading Theorem~\ref{thm_eg_local} is quoted from \citep[p.~41]{gen19}, except that the mathematical terms and the equation numbers have been altered accordingly:
\begin{quote}
	A convenient way to read the above statement is as follows: First, fix an estimation accuracy $t$ that can be tolerated. Then adjust the sample size $n$ and $\beta^\natural \in K$ such that \eqref{thm_eg_local_condition_n} and \eqref{thm_eg_local_condition_t} are both fulfilled (if possible at all). In particular, if $n$ is chosen such that \eqref{thm_eg_local_condition_t} just holds with equality (up to a constant), we obtain an error bound of the form
	\begin{equation}\label{thm_eg_local_condition_error}
		\norm{\hat{\beta}-\beta^\natural}_2 \lesssim \frac{1}{(\tau\cdot Q_{2\tau}(K^\Delta,x))^2}\cdot\pospart[\bigg]{\rho_t(\beta^\natural) + u^2 \cdot \sigma(\beta^\natural)\cdot \frac{m_t^{(g,e)}(K-\beta^\natural)}{\sqrt{n}}}.
	\end{equation}
\end{quote}
With that in mind, one might be tempted to think that not much changes when going beyond sub-Gaussianity---but this is far from being true.
The key difference becomes manifested in our generalized complexity parameters $q_{t,n}^{(g,e)}(K-\beta^\natural)$ and $m_t^{(g,e)}(K-\beta^\natural)$. In fact, their behavior can be significantly more complicated than in the sub-Gaussian case. We defer a more detailed discussion to Subsection~\ref{subsec:discussion:compl}, but also the applications in Subsection~\ref{subsec:discussion:sg}--\ref{subsec:discussion:liftedlasso} can be helpful for a better understanding of this issue.
Finally, several additional remarks on Theorem~\ref{thm_eg_local} and possible refinements can be found in Appendix~\ref{subsec:app:rem_main}.

\subsection{Global and Conic Error Bounds for \texorpdfstring{\eqref{LS_K}}{(LS\_K)}}
\label{subsec:mainresults:globalandconic}

The local complexity parameters in Theorem~\ref{thm_eg_local} lead to a fairly strong, but implicit error bound for \eqref{LS_K}.
%but the implicit nature of the error bound makes it hard to interpret and implement in concrete model setups. 
In this subsection, we state two corollaries of Theorem~\ref{thm_eg_local} which achieve a better interpretability at the price of suboptimality. The first one replaces the local complexity terms by their (more pessimistic) conic versions:
\begin{definition}[Conic q- and m-complexity]\label{def_conic_q_m_complexity} 
	Let $L \subset \R^p$ and let $\gennorm_g$ and $\gennorm_e$ be semi-norms on $\R^p$. We define the \emph{conic q-complexity} of $L$ at sample size $n$ with respect to $(\gennorm_g,\gennorm_e)$ by
\begin{equation}
	q_{0,n}^{(g,e)}(L) \coloneqq \inf\Big\{\tfrac{\gamma_1(S, \gennorm_e)}{\sqrt{n}}+\gamma_2(S, \gennorm_g+\gennorm_e)\MID[\big] S \subset \R^p,\convh(S) \supset \cone{L} \cap  \Sphere^{p-1} \Big\}.
\end{equation}
Similarly, we define the \emph{conic m-complexity} of $L$ with respect to $(\gennorm_g,\gennorm_e)$ by
\begin{equation}
m_0^{(g,e)}(L) \coloneqq \inf\Big\{\gamma_1(S, \gennorm_e)+\gamma_2(S, \gennorm_g)\MID[\big] S \subset \R^p, \convh(S) \supset (\cone{L} \cap  \Sphere^{p-1}) \cup \{0\} \Big\}.
\end{equation}
\end{definition}
The subscript `0' in $q_{0,n}^{(g,e)}(L)$ and $m_0^{(g,e)}(L)$ indicates that one can imagine the conic q- and m-complexity as the limit case $t = 0$ of their local counterparts from Definition~\ref{def_q_complexity} and~\ref{def_m_complexity}.
%; indeed, by Lebesgue's dominated convergence theorem, one can show that $q_{t,n}^{(g,e)}(L) \to q_{0,n}^{(g,e)}(L)$ and $m_t^{(g,e)}(L) \to m_0^{(g,e)}(L)$ as $t \to 0$.

The conic complexity parameters allow us to remove the dependence of the right-hand sides of \eqref{thm_eg_local_condition_n} and \eqref{thm_eg_local_condition_t} on $t$ in Theorem~\ref{thm_eg_local}:
\begin{corollary}[General error bound for \eqref{LS_K}, conic version] \label{cor_eg_conic} 
	The assertion of Theorem~\ref{thm_eg_local} remains valid if $q_{t,n}^{(g,e)}(K-\beta^\natural)$ is replaced by $q_{0,n}^{(g,e)}(K-\beta^\natural)$ in \eqref{thm_eg_local_condition_n}, while $m_t^{(g,e)}(K-\nobreak\beta^\natural)$ and $\rho_t(\beta^\natural)$ are replaced by $m_0^{(g,e)}(K-\beta^\natural)$ and $\rho_0(\beta^\natural)$ in \eqref{thm_eg_local_condition_t}, respectively.
\end{corollary}
While leading to an explicit error bound for the generalized Lasso (cf.~\eqref{thm_eg_local_condition_error}), Corollary~\ref{cor_eg_conic} has the following drawback: If $\beta^\natural$ is an interior point of $K$, then $\cone{K-\beta^\natural}=\R^p$, and the complexity terms $q_{0,n}^{(g,e)}(K-\beta^\natural)$ and $m_0^{(g,e)}(K-\beta^\natural)$ are equal to $q_{0,n}^{(g,e)}(\R^p)$ and $m_0^{(g,e)}(\R^p)$, respectively, i.e., they no longer reflect any complexity reduction due to the restricted hypothesis set~$K$.
Hence, unless the hypothesis set $K$ is perfectly tuned such that $\beta^\natural$ is located on the boundary of~$K$, Corollary~\ref{cor_eg_conic} fails to provide a useful estimation guarantee in the high-dimensional regime $p \gg n$. Evidently, this tuning problem affects the local error bound in Theorem~\ref{thm_eg_local} as well, but the situation is much less severe there, at least when $\beta^\natural$ is close to the boundary of $K$ (more precisely, if $\inf_{\beta \in \R^p\setminus K}\norm{\beta^\natural - \beta}_2<t$). This fact particularly explains why \eqref{LS_K} is a \emph{stable} estimator (cf.~\citep[Prop.~2.6 and Cor.~3.15]{gen19}).

Our second approach to simplify Theorem~\ref{thm_eg_local} is to measure the complexity of the hypothesis set ``globally'', rather than in a local neighborhood of $\beta^\natural$.
\begin{definition}[Global q- and m-complexity]\label{def_global_q_m_complexity}
Let $L \subset \R^p$ and let $\gennorm_g$ and $\gennorm_e$ be semi-norms on $\R^p$. We define the \emph{global q-complexity} of $L$ at sample size $n$ with respect to $(\gennorm_g,\gennorm_e)$ by
\begin{equation}
	q_n^{(g,e)}(L) \coloneqq \inf\Big\{\tfrac{\gamma_1(S, \gennorm_e)}{\sqrt{n}}+\gamma_2(S, \gennorm_g+\gennorm_e) \MID[\big] S \subset \R^p, \convh(S)\supset L\Big\}.
\end{equation}
Similarly, we define the \emph{global m-complexity} of $L$ with respect to $(\gennorm_g,\gennorm_e)$ by
\begin{equation}
	m^{(g,e)}(L) \coloneqq \inf\Big\{\gamma_1(S, \gennorm_e)+\gamma_2(S, \gennorm_g) \MID[\big] S \subset \R^p, \convh(S)\supset L\Big\}.
\end{equation}
\end{definition}
The following lemma provides some basic facts about the global complexity parameters and relates them to their local counterparts.
\begin{lemma}\label{lemma_q_and_m_global}
Let $L \subset \R^p$, $v \in \R^p$, and $t>0$. Then we have the following:
\begin{thmproperties}
	\item\label{lemma_q_and_m_global_1} $q_{t,n}^{(g,e)}(L) \leq \frac{1}{t} q_n^{(g,e)}(L)$ and $m_t^{(g,e)}(L) \leq \frac{1}{t} m^{(g,e)}(L \cup \{0\})$,
	\item\label{lemma_q_and_m_global_2}  $q_n^{(g,e)}(L)= q_n^{(g,e)}(L+v)$ and $m^{(g,e)}(L)=m^{(g,e)}(L+v)$,
	\item\label{lemma_q_and_m_global_3} $q_n^{(g,e)}(L) \lesssim m^{(g,e)}(L)$,
	\item\label{lemma_q_and_m_global_4} $q_{0,n}^{(g,e)}(L) = q_n^{(g,e)}(\cone{L}\cap \Sphere^{p-1})$ and $m_{0}^{(g,e)}(L) = m^{(g,e)}((\cone{L}\cap \Sphere^{p-1}) \cup \{0\})$,
	\item\label{lemma_q_and_m_global_5} $q_{t,n}^{(g,e)}(L) = q_n^{(g,e)}(\tfrac{1}{t}L\cap \Sphere^{p-1})$ and $m_{t}^{(g,e)}(L) = m^{(g,e)}((\tfrac{1}{t}L\cap \Sphere^{p-1}) \cup \{0\})$.
\end{thmproperties}
\end{lemma}
The second claim of Lemma~\ref{lemma_q_and_m_global} states that the global complexity parameters are translation-in\-vari\-ant.
This allows us to decouple the complexity terms in Theorem~\ref{thm_eg_local} from $\beta^\natural$, leading to the following error bound:
\begin{corollary}[General error bound for \eqref{LS_K}, global version] \label{cor_eg_global} 
Let Assumption~\ref{model_setup} be satisfied and fix a vector $\beta^\natural \in K$. Then there exists a universal constant $C>0$ such that for every $u \geq 8$, the following holds true with probability at least $1-5\exp(-C\cdot u^2)-2\exp(-C \cdot \sqrt{n})$:
If the sample size obeys
\begin{equation}\label{cor_eg_global_condition_n}
	n \gtrsim \bigg(\frac{q_n^{(g,e)}(K)+ \tau \cdot u}{\tau\cdot Q_{2\tau}(K^\Delta,x)}\bigg)^2,
\end{equation}
then every minimizer $\hat{\beta}$ of \eqref{LS_K} satisfies
\begin{equation}\label{cor_eg_global_bound_t}
	\norm{\hat{\beta}-\beta^\natural}_2 \lesssim \max\Big\{1,\big(\tau\cdot Q_{2\tau}(K^\Delta,x)\big)^{-2}\Big\} \cdot \pospart[\bigg]{\rho_0(\beta^\natural) + \max\{1,u^2 \cdot \sigma(\beta^\natural)\} \cdot \frac{\sqrt{m^{(g,e)}(K)}}{n^{1/4}}}.
\end{equation}
\end{corollary}

If the complexity terms $q^{(g,e)}_n(K)$ and $m^{(g,e)}(K)$ are sufficiently small, then \eqref{cor_eg_global_bound_t} provides a useful error bound in the high-dimensional regime $p \gg n$, independently of the location of~$\beta^\natural$ in~$K$. Note that a prototypical application of Corollary~\ref{cor_eg_global} was already presented in Proposition~\ref{prop_intro} where $K$ is a convex polytope (see also Proposition~\ref{prop_polytopes} in Subsection~\ref{subsec:discussion:compl}). However, the simplification of Corollary~\ref{cor_eg_global} has its price: the second summand in the error bound \eqref{cor_eg_global} exhibits a decay rate $O(n^{-1/4})$, which is substantially worse that the rate of $O(n^{-1/2})$ achieved in Theorem~\ref{thm_eg_local} and Corollary~\ref{cor_eg_conic}. Moreover, the dependence on $\sigma(\beta^\natural)$ is suboptimal in the ``low-noise'' regime, i.e., when $\sigma(\beta^\natural) \ll 1$.

\section{Applications and Examples}
\label{sec:discussion}

This section is devoted to specific applications of the generic error bounds presented in Section~\ref{sec:mainresults}.
We begin with a discussion of semi-parametric estimation problems in Subsection~\ref{subsec:discussion:semipara}, in particular, how the generalized Lasso \eqref{LS_K} performs with non-linear output models.
In Subsection~\ref{subsec:discussion:sg}--\ref{subsec:discussion:liftedlasso}, we then demonstrate that generic Bernstein concentration covers a whole ``spectrum'' of relevant distributions, where (uniformly) sub-exponential and sub-Gaussian input vectors  appear just as marginal cases.
Finally, we continue our discussion on the notions of q- and m-complexity in Subsection~\ref{subsec:discussion:compl}, thereby focusing on the prototypical situation of sparse recovery via $\l{1}$-constraints.

\subsection{Semi-Parametric Estimation Problems and the Mismatch Principle}
\label{subsec:discussion:semipara}

We intentionally did not make a concrete choice of the target vector $\beta^\natural$ in Section~\ref{sec:mainresults}.
This strategy has led to very flexible (generic) error bounds for \eqref{LS_K}, but it does not address any specific estimation problem.
As already pointed out subsequently to the initial Problem~\ref{intro:problem}, a valid choice of $\beta^\natural$ is the expected risk minimizer.
Indeed, assuming that $x$ is isotropic and $\beta^\natural \coloneqq \E[yx] \in K$, then $\beta^\natural$ is the expected risk minimizer (on both~$K$ and~$\R^p$) and we have that $\rho(\beta^\natural) = \rho_t(\beta^\natural) = 0$ (see Appendix~\ref{subsec:app:rem_on_mc} and Figure~\ref{fig:mismatchcovar} there). Hence, according to Theorem~\ref{thm_eg_local} (or its corollaries), \eqref{LS_K} yields a consistent estimator of $\beta^\natural$.

While such a statement is common in statistical learning, a much less obvious phenomenon is the capability of \eqref{LS_K} to solve \emph{semi-parametric estimation problems}.
In the context of this article, we may express a semi-parametric observation model as follows:
\begin{equation}\label{eq_semipara}
	y = F(x, \beta_0),
\end{equation}
where $\beta_0 \in \R^p$ is an unknown \emph{parameter vector} and $F : (\R^p \times \R^p) \to \R$ a scalar \emph{output function} which can be non-linear, random, and unknown.
Agreeing on this model setup, the ultimate hope is now that \eqref{LS_K} is a (consistent) estimator of $\beta_0$.\footnote{If $F$ is non-linear, this can be very different from asking for the expected risk minimizer, which would simply yield the best linear predictor of $y$.}
It turns out that this is often possible at least to a certain extent, even though fitting a linear model to non-linear observations might appear counterintuitive at first sight.
A typical example is the simple classification rule $y = \sign(\scpr{x, \beta_0})$, where there is still hope to recover the direction of $\beta_0$, but not its magnitude. This limitation gives rise to a relaxed estimation problem:
\begin{problem}\label{problem_target}
	 Is the generalized Lasso \eqref{LS_K} capable of estimating any element from a certain \emph{target set} $T_{\beta_0} \subset \R^p$, which contains all those parameter vectors that allow us to extract the information of interest?
\end{problem}
Similarly to the more general formulation of Problem~\ref{intro:problem}, the term `information' is left unspecified here and depends on what a user considers as a desirable outcome of an estimation procedure. In the above example of binary classification, a natural choice of target set would be $T_{\beta_0} \coloneqq \spann(\{\beta_0\})$, if one is interested in the recovery of \emph{any} scalar multiple of $\beta_0$.

Our guarantees from Section~\ref{sec:mainresults} allow us to tackle Problem~\ref{problem_target} in a very systematic way:
\begin{highlight}
	Select $\beta^\natural \in T_{\beta_0} \cap K$ such that the mismatch covariance $\rho(\beta^\natural)$ becomes as small as possible.\footnote{For the sake of clarity, we only consider the global mismatch covariance here, which is easier to interpret and forms an upper bound for $\rho_t(\beta^\natural)$ according to Appendix~\ref{subsec:app:rem_on_mc}; but refinements are certainly possible when analyzing $\rho_t(\beta^\natural)$ instead of $\rho(\beta^\natural)$.}\textsuperscript{,}\footnote{If $x$ is isotropic, this selection procedure has a nice geometric interpretation due to Appendix~\ref{subsec:app:rem_on_mc}: The mismatch covariance $\rho(\beta^\natural)$ is minimized on $T_{\beta_0} \cap K$ if and only if $\beta^\natural$ is a Euclidean projection of the (global) expected risk minimizer $\E[yx]$ onto $T_{\beta_0} \cap K$.}
	Then apply Theorem~\ref{thm_eg_local} (or one of its corollaries) to obtain an error bound for the estimation error $\norm{\hat{\beta}-\beta^\natural}_2$.
\end{highlight}
This strategy ensures that the resulting target vector $\beta^\natural$ encodes the desired information, while the (asymptotic) bias of \eqref{LS_K} is brought under control.
In particular, if $\rho(\beta^\natural) = 0$, we achieve a consistent estimator of $\beta^\natural$; note that the corresponding mismatch deviation~$\sigma(\beta^\natural)$ can still be large, but its size only affects the variance of the error $\norm{\hat{\beta}-\beta^\natural}_2$.
The approach just described was developed by~\citet[Chap.~4]{gen19}, where it is referred to as the \emph{mismatch principle} (see also the technical report~\citep{gen18}).
It is worth pointing out that there is an important conceptual difference to the ``naive'' idea of first explicitly computing the expected risk minimizer (on~$K$) and then finding the closest point on the target set~$T_{\beta_0}$: indeed, we measure the complexity of $K$ locally at $\beta^\natural$, which enables us to exploit beneficial geometric features directly on $T_{\beta_0}$.

We refer the reader to \citep[Chap.~4]{gen19} for a more extensive discussion of the mismatch principle and various applications to semi-parametric estimation problems.
In the present work, we confine ourselves to an illustration in the prototypical situation of single-index models.
\begin{proposition}[\protect{\citep[Prop.~4.6]{gen19}}] \label{prop_mismatch_sim}
	Let $x \in \R^p$ be a centered, isotropic random vector. We assume that $y$ obeys a \emph{single-index model} of the form
	\begin{equation}\label{prop_mismatch_sim:sim}
		y = f(\scpr{x, \beta_0}) + \nu,
	\end{equation}
	where $\beta_0 \in \R^p\setminus \{0\}$ is an unknown parameter vector, $f: \R \to \R$ is a scalar output function, and $\nu$ is independent noise with $\E[\nu] = 0$.
	Moreover, we choose $T_{\beta_0} \coloneqq \spann(\{\beta_0\})$ as target set. Then $\beta^\natural = \mu \beta_0$ with
	\begin{equation}\label{prop_mismatch_sim:mu}
		\mu \coloneqq \tfrac{1}{\norm{\beta_0}_2^2} \cdot \E[f(\scpr{x, \beta_0}) \cdot \scpr{x, \beta_0}]
	\end{equation}
	minimizes the (global) mismatch covariance over $T_{\beta_0}$ and we have that
	\begin{equation}
		\rho(\beta^\natural) = \norm[\big]{\E[f(\scpr{x, \beta_0}) P_{\beta_0}^\perp x]}_2,
	\end{equation}
	where $P_{\beta_0}^\perp \in \R^{p \times p}$ is the projection onto the orthogonal complement of $\spann(\{\beta_0\})$. In particular, if $x$ is a standard Gaussian random vector, we have that $\rho(\beta^\natural) = 0$.
\end{proposition}
In the special case of a Gaussian input vector, Proposition~\ref{prop_mismatch_sim} reproduces the original finding of \citet{pv16}:
despite an unknown, non-linear distortion, the generalized Lasso  still allows for consistent estimation of the parameter vector, or at least a scalar multiple of it.
When combining Proposition~\ref{prop_mismatch_sim} with the results from Section~\ref{sec:mainresults} (for an appropriately tuned hypothesis set $K$), we observe that their conclusion remains essentially valid for non-Gaussian inputs as long as the mismatch covariance $\rho(\beta^\natural)$ vanishes or gets sufficiently small.
%we observe that their conclusion remains essentially valid as long as the mismatch covariance vanishes or is sufficiently small.
%When combining Proposition~\ref{prop_mismatch_sim} with the results from Section~\ref{sec:mainresults} (for an appropriately tuned hypothesis set $K$), we observe that their conclusion essentially remains valid for non-Gaussian inputs, with the price of a possible asymptotic bias.
On the other hand, if $\rho(\beta^\natural)$ is too large, it can be useful to employ an adaptive estimator instead (e.g., see \citep{ywlez15,ybkl17,ybl17}), but there also exist worst-case scenarios where an asymptotic bias is inevitable, regardless of the considered estimator (see \citep{alpv14}).
For an overview of the extensive literature on single-index models as well as historical references, we refer the reader to \citep[Sec.~6]{pvy16} and \citep[Subsec.~1.2]{ybl17}.
Moreover, see~\citep[Subsec.~4.2.2]{gen19} and the references therein for related works in $1$-bit compressed sensing.

Finally, it is worth pointing out that we did not make any (explicit) assumptions on the tail behavior of the distribution of $x$ in this subsection.
Therefore, one can easily combine the above described approach with the findings of the forthcoming subsections, which investigate specific instances of generic Bernstein concentration.

\subsection{Sub-Gaussian Input Vectors}
\label{subsec:discussion:sg}

The current and subsequent subsections are devoted to several examples of generic Bernstein concentration (see Definition~\ref{def_gen_bern}).
Let us begin with the situation of (uniformly) sub-Gaussian input vectors. 
A characteristic property of sub-Gaussian random variables is that their tails are essentially not heavier than those of the normal distribution (see Proposition~\ref{prop_sg_se_properties}\ref{prop_sg_se_ci}). 
Combining this property with Definition~\ref{def_sg_se_randvect}, one can easily verify that a sub-Gaussian random vector $x \in \R^p$ with sub-Gaussian norm $\norm{x}_{\psi_2}$ exhibits generic Bernstein concentration with respect to $(C \norm{x}_{\psi_2} \gennorm_2,0)$ for a universal constant $C > 0$. 
Since the sub-exponential part of the mixed-tail condition is effectively erased here by setting $\gennorm_e \coloneqq 0$, we observe that sub-Gaussian input vectors form a degenerate limit case at the lighter-tailed end of the ``spectrum'' of generic Bernstein concentration. Regarding the q- and m-complexities, the identity $\gennorm_e = 0$ implies that the $\gamma_1$-functional effectively vanishes in their respective definitions, so that we end up with a rescaled version of the functional $\gamma_2(\cdot, \gennorm_2)$. The celebrated Majorizing Measure Theorem of \citet[Thm.~2.4.1]{tal14} relates this functional to a well-known complexity parameter:
\begin{definition}[Gaussian width]
Let $L \subset \R^p$ and let $g  \sim \mathcal{N}(0, I_p)$ be a standard Gaussian random vector. The \emph{Gaussian width} of $L$ is defined by
\begin{equation}\label{def_gaussian_width_eq}
	w(L) \coloneqq \E\Big[\sup_{v\in L} \scpr{g,v}\Big].
\end{equation}
\end{definition}
The Gaussian width originates from classical results in geometric functional analysis and asymptotic convex geometry, e.g., see~\citep{mil84,gor88,gm04}. More recently, it has emerged as a useful tool for the analysis of high-dimensional estimation problems, e.g., see~\citep{mpt07,rv08,sto09,crpw12,almt14,tro15,oh16}. The connection to our analysis, which is provided by Talagrand's Majorizing Measure Theorem, is  the fact that for every subset $L \subset \R^p$, we have
\begin{equation}\label{eq_talagrand_majorizing_measure}
\gamma_2(L, \gennorm_2) \asymp w(L).
\end{equation}
Apart from a simple geometric interpretation of the resulting complexity parameters, \eqref{eq_talagrand_majorizing_measure} implies that the optimization over the ``skeleton'' is irrelevant (up to constants) in the sub-Gaussian case, since the Gaussian width is invariant under taking the convex hull. This explains why such an optimization is uncommon in the literature dealing with sub-Gaussian input data.

The following fact summarizes the above considerations and allows us to relate the generic (global) error bound from Corollary~\ref{cor_eg_global} to the sub-Gaussian setting:
\begin{fact}\label{prop_sg}
		Let $x \in \R^p$ be a (uniformly) sub-Gaussian random vector, i.e., $\norm{x}_{\psi_2} <\nobreak \infty$.
		Then~$x$ exhibits generic Bernstein concentration with respect to $(\sqrt{C_2} \norm{x}_{\psi_2} \gennorm_2,0)$, where $C_2 > 0$ is the constant from Proposition~\ref{prop_sg_se_properties}\ref{prop_sg_se_ci}.
		The global q- and m-complexities satisfy
		\begin{align} 
			q_{n}^{(g,e)}(L) &\asymp \norm{x}_{\psi_2} \cdot \underbrace{\inf\Big\{\gamma_2(S,\gennorm_2) \MID[\big] S \subset \R^p, \convh(S) \supset L \Big\}}_{\eqqcolon q_{n}^{(2,0)}(L)}\asymp \norm{x}_{\psi_2} \cdot w(L),\\
			m^{(g,e)}(L) &\asymp \norm{x}_{\psi_2} \cdot \underbrace{\inf\Big\{\gamma_2(S,\gennorm_2) \MID[\big] S \subset \R^p, \convh(S) \supset L \Big\}}_{\eqqcolon m^{(2,0)}(L)} \asymp \norm{x}_{\psi_2} \cdot w(L). \label{prop_sg_eq}
		\end{align}
\end{fact}
The comparison of our results with existing ones is facilitated by introducing the \emph{normalized} complexities $q_{n}^{(2,0)}$ and $m^{(2,0)}$ in \eqref{prop_sg_eq}. Indeed, both parameters are unaffected by a rescaling of the input vector, which is a common feature of complexity measures defined in the literature.
In contrast, their ``unnormalized'' counterparts $q_{n}^{(g,e)}$ and $m^{(g,e)}$ ``absorb'' the norm $\norm{x}_{\psi_2}$ as a scalar pre-factor for the generic semi-norms $\gennorm_e$ and $\gennorm_g$.
% In our analysis, the norm $\norm{x}_{\psi_2}$ is `absorbed' as a scalar pre-factor for the generic semi-norms $\gennorm_e$ and $\gennorm_g$, and thereby enters into the complexity terms. In contrast, most other texts compute the complexity of a set with respect to fixed norms on $\R^p$, which are unaffected by a rescaling of the input vectors. The normalized complexity terms $q_{n}^{(2,0)}$ and $m^{(2,0)}$ have this property.

With regard to the local and conic q- and m-complexities, Talagrand's Majorizing Measure Theorem leads to similar conclusions. The (normalized) local \mbox{q-}com\-plex\-ity corresponds to the notion of \emph{local Gaussian width} (cf.~\citep{pv16,gen16}):
\begin{equation}
q_{t,n}^{(2,0)}(L) \coloneqq \frac{1}{t}\inf\Big\{\gamma_2(S,\gennorm_2) \MID[\big] S \subset \R^p, \convh(S) \supset L \cap t\Sphere^{p-1} \Big\} \asymp \frac{1}{t} w(L \cap t\Sphere^{p-1}).
\end{equation}
Since the definition of the (normalized) local \mbox{m-}com\-plex\-ity requires that $\convh(S)$ also contains the origin, it is not strictly equivalent to the local Gaussian width, but incorporates an additional constant term (cf.~Appendix~\ref{subsec:app:rem_main}\ref{rmk_local_radius}):
\begin{equation}
m_{t}^{(2,0)}(L) \coloneqq \frac{1}{t}\inf\Big\{\gamma_2(S,\gennorm_2) \MID[\big] S \subset \R^p, \convh(S) \supset (L \cap t\Sphere^{p-1}) \cup \{0\} \Big\} \asymp \frac{1}{t} w(L \cap t\Sphere^{p-1}) +1.
\end{equation}
Analogously, the (normalized) conic \mbox{q-}com\-plex\-ity and m-complexity correspond to the notion of \emph{conic Gaussian width} (cf.~\citep{crpw12,almt14}). %, whereas the (normalized) conic \mbox{m-}com\-plex\-ity incorporates the additional constant term. 
A combination of these identifications with the corresponding error bounds from Section~\ref{sec:mainresults} (Theorem~\ref{thm_eg_local}, Corollary~\ref{cor_eg_conic}, and Corollary~\ref{cor_eg_global}) allows us to reproduce known estimation guarantees for (sub-)Gaussian sample data, e.g., see~\citep{pv16,gen19}.

\subsection{Input Vectors with Independent Sub-Exponential Features}
\label{subsec:discussion:indep_se_features}

Although it is reassuring that the generic error bounds from Section~\ref{sec:mainresults} are consistent with existing results for the sub-Gaussian case, this setup does not constitute a proper example of the mixed-tail condition in Definition~\ref{def_gen_bern}. A more natural example is given by input vectors with centered, independent, sub-exponential coordinates. In this case, generic Bernstein concentration can be simply implemented by the classical Bernstein's inequality (see Theorem~\ref{thm_bern_ineq}). 
The following fact is a direct consequence of Theorem~\ref{thm_bern_ineq}.
\begin{fact}\label{prop_isef}
Let $x \in \R^p$ be a random vector with centered, independent, sub-exponential coordinates. Set $R \coloneqq \max_{1 \leq j \leq d} \norm{x_j}_{\psi_1}$. Then $x$ exhibits generic Bernstein concentration with respect to $\big(\frac{R}{\sqrt{C_B}}\norm{\cdot}_2, \frac{R}{C_B}\norm{\cdot}_\infty \big)$, where $C_B > 0$ is the constant from Theorem~\ref{thm_bern_ineq}.
The global q- and m-complexities satisfy
		\begin{align}
			q_{n}^{(g,e)}(L) & \asymp R \cdot \inf\Big\{\tfrac{\gamma_1(S, \norm{\cdot}_\infty)}{\sqrt{n}}+ \gamma_2(S, \norm{\cdot}_2) \MID[\big] S \subset \R^p, \convh(S) \supset L \Big\}, \\
			m^{(g,e)}(L) & \asymp R \cdot \underbrace{\inf\Big\{\gamma_1(S, \norm{\cdot}_\infty)+ \gamma_2(S, \norm{ \cdot}_2) \MID[\big] S \subset \R^p, \convh(S) \supset L \Big\}}_{\eqqcolon m^{(2,\infty)}(L)}.
		\end{align}
\end{fact}
Remarkably, there exists a geometric interpretation of $m^{(2,\infty)}$ that is very similar to \eqref{eq_talagrand_majorizing_measure}. To this end, let $L \subset \R^p$ and let $Y=(Y_1,\dots,Y_p)$ be a random vector with independent, symmetric coordinates satisfying $\Prob(|Y_j| \geq t) = \exp(-t)$ for all $j = 1, \dots, p$. By a result of \citet[Thm.~10.2.8]{tal14}, we then have
\begin{equation}\label{exponential_width_eq}
 \gamma_1(L,\gennorm_\infty)+\gamma_2(L, \gennorm_2) \asymp \E \Big[\sup_{v \in L} \scpr{v,Y}\Big].
\end{equation}
Inspired by the notion of `Gaussian' width, the expression on the right-hand side is referred to as the \emph{exponential width} of $L$.
Similarly to the sub-Gaussian case in Subsection~\ref{subsec:discussion:sg}, the relation \eqref{exponential_width_eq} shows that the optimization over the ``skeleton'' does not make a difference in the present scenario, at least when ignoring universal constants. Therefore, we can conclude that the normalized m-complexity is equivalent to the exponential width, i.e., $m^{(2,\infty)}(L) \asymp \E \big[\sup_{v \in L} \scpr{v,Y}\big]$. 

The idea of using the exponential width as a complexity measure for sub-exponential input vectors was proposed by \citet{sbr15}. In contrast to the Gaussian width, the exponential width is not rotation-invariant: only the sub-Gaussian component of the complexity is tied to the Euclidean structure on $\R^p$, whereas the sub-exponential component of the complexity is described by the $\l{\infty}$-norm. Based on results by Talagrand, it follows from \citep[Thm.~1]{sbr15} that for every subset $L \in \R^p$, we have
\begin{equation} \label{bound_exponential_width}
	m^{(2,\infty)}(L) \asymp \E \Big[\sup_{v \in L} \scpr{v,Y}\Big] \lesssim \sqrt{\log(p)} \cdot  w(L).
\end{equation} 
Regarding the feasibility of estimation in the high dimensions, this shows that the situation for input vectors with independent, sub-exponential features is not substantially worse than for sub-Gaussian sample data.

According to Lemma~\ref{lemma_q_and_m_global}\ref{lemma_q_and_m_global_3}, the normalized global q-complexity $q_n^{(2,\infty)}$ (defined analogously to $m^{(2,\infty)}$) can also be controlled by the exponential width. With this in mind, our lower bound for the quadratic process in Proposition~\ref{prop_quadratic_bound} resembles a result of~\citet[Thm.~3]{sbr15}, if we assume input vectors with independent, sub-exponential entries. In contrast, \citep[Thm.~3]{sbr15} just requires isotropy and uniform sub-exponentiality. Due to a lack of published proofs, we were unable to verify that the exponential width is the correct complexity measure for this more general scenario.

\subsection{(Uniformly) Sub-Exponential Input Vectors}
\label{subsec:discussion:se}

We were already concerned with the situation of (uniformly) sub-expo\-nen\-tial input vectors in Subsection~\ref{subsec:intro:simple} (see Proposition~\ref{prop_intro}). Taking the more abstract viewpoint from Section~\ref{sec:mainresults}, this setting corresponds to a degenerate limit case at the heavier-tailed end of the ``spectrum'' of generic Bernstein concentration. Indeed, an application of Proposition~\ref{prop_sg_se_properties}\ref{prop_sg_se_ci} leads to the following fact:
\begin{fact}\label{prop_se}
		Let $x \in \R^p$ be a (uniformly) sub-exponential random vector, that is, $\norm{x}_{\psi_1} <\nobreak \infty$.
		Then~$x$ exhibits generic Bernstein concentration with respect to $(0, C_1 \norm{x}_{\psi_1} \gennorm_2)$, where $C_1 > 0$ is the constant from Proposition~\ref{prop_sg_se_properties}\ref{prop_sg_se_ci}.
		The global q- and m-complexities satisfy
		\begin{align}
			q_{n}^{(g,e)}(L) & \asymp \norm{x}_{\psi_1} \cdot \underbrace{\inf\Big\{\tfrac{\gamma_1(S, \gennorm_2)}{\sqrt{n}}+ \gamma_2(S,\gennorm_2) \MID[\big] S \subset \R^p, \convh(S) \supset L \Big\}}_{\eqqcolon q_{n}^{(0,2)}(L)}, \\
			m^{(g,e)}(L) & \asymp \norm{x}_{\psi_1} \cdot \underbrace{\inf\Big\{\gamma_1(S, \gennorm_2) \MID[\big] S \subset \R^p, \convh(S) \supset L \Big\}}_{\eqqcolon m^{(0,2)}(L)}.
		\end{align}
\end{fact}
According to Talagrand's Majorizing Measure Theorem, we have that
\begin{equation} \label{eq_perturbed_width}
q_{n}^{(0,2)}(L) \asymp \inf\Big\{\tfrac{\gamma_1(S, \gennorm_2)}{\sqrt{n}}+ w(S) \MID[\big] S \subset \R^p, \convh(S) \supset L \Big\}.
\end{equation}
The right-hand side of \eqref{eq_perturbed_width} agrees with the notion of \emph{perturbed width} that was considered by \citep{oym19,so19}. Since \citet{so19} focus on the projected gradient descent as an algorithmic implementation of the generalized Lasso \eqref{LS_K}, their error bounds are not directly comparable to ours (in the special case of sub-exponential input vectors), but they bear a resemblance. A remarkable difference is that we achieve an exponentially decaying probability of failure. This is due to the fact that we handle the multiplier process by Mendelson's chaining approach (see Subsection~\ref{subsec:proofs:multiplierprocess}), which also explains why the notion of m-complexity does not appear in the results of~\citep{so19}.

Unlike in the settings of Subsection~\ref{subsec:discussion:sg} and~\ref{subsec:discussion:indep_se_features}, a simple geometric interpretation of the complexity parameters is not available in the case of uniformly sub-exponential input vectors. In particular, there is no reason to believe that the optimization over the ``skeleton'' is unnecessary in general. Compared to the $\gamma_2$-functional, which can be controlled by means of the Gaussian width, the $\gamma_1$-functional seems more mysterious and intangible.
However, it is at least possible to derive informative upper bounds in the important special case where~$K$ is a scaled $\l{1}$-ball (see Subsection~\ref{subsec:discussion:compl}).

\subsection{The Lifted Lasso and Phase Retrieval}
\label{subsec:discussion:liftedlasso}

Another relevant example of generic Bernstein concentration occurs when applying the so-called lifted Lasso to sub-Gaussian input vectors. The lifted Lasso introduced below can be seen as a variant of the phase lift approach (see \citep{cesv13,csv13}), which is tailored to the phase retrieval problem (e.g., see~\citep{seccms15} for an overview). Our statistical analysis is not limited to this specific model setup and covers the more general scenario considered by~\citet{tr19}, namely single-index models with \emph{even} output functions.
In fact, Proposition~\ref{prop_mismatch_sim} indicates that this is a highly non-trivial task: if $y$ obeys \eqref{prop_mismatch_sim:sim} with Gaussian inputs and an even function $f$, then we would simply have $\mu = 0$, so that the ordinary Lasso \eqref{LS_K} fails to recover the direction of the parameter vector $\beta_0$.

\emph{Phase lifting} follows a different approach that allows us to reduce the non-linear phase retrieval problem to a more accessible linear problem.
It is based on the simple, yet crucial observation that
\begin{equation}\label{tr_observation}
	\scpr{x,\beta}^2=\tr(xx^T \beta \beta^T)=\scpr{xx^T, \beta \beta^T}_F \quad \text{ for all $x, \beta \in \R^p$},
\end{equation}
where $\tr(\cdot)$ denotes the trace and $\scpr{\cdot, \cdot}_F$ the Hilbert-Schmidt inner product.
The \emph{lifted Lasso} then corresponds to the following convex optimization problem:
\begin{equation} \label{LLS_H} \tag{LLS$_H$}
	\min_{B \in H} \ \tfrac{1}{n} \sum_{i=1}^n \Big(y_i-\scpr{x_ix_i^T-\E[x x^T],B}_F\Big)^2,
\end{equation}
where $H \subset \R^{p \times p}$ is a convex subset of the positive semidefinite cone in $\R^{p \times p}$ which contains all ``lifted'' hypotheses, i.e., $H \supset \{\beta\beta^T \MID \beta \in K\}$ for some $K \subset \R^p$.
Note that the centering term $\E[x x^T]$ in \eqref{LLS_H} is important to achieve consistency (see also Proposition~\ref{prop_quadratic_with_noise} below). 

If the input vector $x \in \R^p$ is sub-Gaussian, then $xx^T \in \R^{p \times p}$ is sub-exponential.\footnote{Here and in the following, the matrix space $\R^{p\times p}$ is canonically identified with $\R^{p^2}$. In particular, we interpret $xx^T$ as a random vector in $\R^{p^2}$, rather than a random matrix.} In this sense, the lifted Lasso is a typical application where sub-exponential vectors  occur naturally.
The connection between the lifted setting and the notion of generic Bernstein concentration is given by the Hanson-Wright inequality (cf.~\citep[Thm.~6.2.1]{ver18}):
\begin{fact} \label{prop_hw_ineq}
Let $x \in \R^p$ be a random vector with centered, independent, sub-Gaussian coordinates. Moreover, let $B \in \R^{p \times p}$ and set $R \coloneqq \max_{1 \leq j \leq p} \norm{x_j}_{\psi_2}$. Then $xx^T-\E[x x^T] \in \R^{p \times p}$ exhibits generic Bernstein concentration with respect to $\big(\frac{R^2}{\sqrt{C}}\gennorm_F,\frac{R^2}{C} \gennorm_{\op}\big)$, where $C > 0$ is a universal constant. %, where $\gennorm_F$ denotes the Frobenius norm and $\gennorm_{\op}$ denotes the operator norm (induced by $\gennorm_2$).
In particular, the global q- and m-complexities satisfy\footnote{Note that the operator norm is absorbed by the Frobenius norm in the $\gamma_2$-part of the q-complexity, which is possible due to $\gennorm_F + \gennorm_{\op} \lesssim \gennorm_F$.}
 \begin{align}
	q_{n}^{(g,e)}(L) & \asymp R^2 \cdot \underbrace{\inf\Big\{\tfrac{\gamma_1(S, \gennorm_{\op})}{\sqrt{n}}+ \gamma_2(S,\gennorm_F) \MID[\big] S \subset \R^{p \times p}, \convh(S) \supset L \Big\}}_{\eqqcolon q_{n}^{(F,{\op})}(L)}, \\
	m^{(g,e)}(L) & \asymp R^2 \cdot \underbrace{\inf\Big\{\gamma_1(S, \gennorm_{\op}) +\gamma_2(S, \gennorm_F) \MID[\big] S \subset \R^{p \times p}, \convh(S) \supset L \Big\}}_{\eqqcolon m^{(F,{\op})}(L)}.
\end{align}
\end{fact}
%In the setting of Proposition~\ref{prop_hw_ineq}, the global q- and m-complexities satisfy
It is noteworthy that a similar result can be achieved under different assumptions on~$x$. For instance, \citet[Thm.~2.5]{ada15} proves a Hanson-Wright inequality under the assumption of a convex concentration property, while \citet[Thm.~1.5]{jlpy20} provide a somewhat different version under the additional assumption of unit variance.
Alternatively, if the input vector $x$ is just uniformly sub-Gaussian, we can make use of a result by \citet[Prop.~2.7, Rem.~2.8]{zaj19} to obtain the following:
\begin{fact}\label{prop_zajkowski}
	Let $x \in \R^p$ be a centered, (uniformly) sub-Gaussian random vector and let $B \in \R^{p \times p}$. Then $xx^T-\E[xx^T]$ exhibits generic Bernstein concentration with respect to $(0,C\norm{x}_{\psi_2}^2 \gennorm_F)$, where $C > 0$ is a universal constant.
\end{fact}

All aforementioned results can be integrated into our framework. The following error bound is a direct application of Corollary~\ref{cor_eg_global} to the setup of Fact~\ref{prop_hw_ineq}:
\begin{corollary}[Error bound for \eqref{LLS_H}, global version] \label{prop_lifted_eg_global} 
	Let $x\in \R^p$ be as in Fact~\ref{prop_hw_ineq} and let~$y$ be sub-exponential. Let the sample pairs $(x_1,y_1),\dots,(x_n,y_n)$ be independent copies of $(x,y)$. Moreover, let $H \subset \R^{p \times p}$ be convex and fix $B^\natural \in H$. We also assume that
	\begin{equation}
		Q \coloneqq Q_{2\tau}(\spann(H-H) \cap \Sphere^{p^2-1},xx^T-\E[x x^T]) > 0.
	\end{equation}
	Then there exists a universal constant $C>0$ such that for every $u \geq 8$, the following holds true with probability at least $1-5\exp(-C\cdot u^2)-2\exp(-C \cdot \sqrt{n})$:
	If the sample size obeys
	\begin{equation}\label{prop_lifted_eg_global_condition_n}
		n \gtrsim  \bigg(\frac{R^2\cdot q_n^{(F,{\op})}(H)+ \tau \cdot u}{\tau\cdot Q}\bigg)^2,
	\end{equation}
	then every minimizer $\hat{B}$ of \eqref{LLS_H} satisfies
	\begin{equation}\label{prop_lifted_eg_global_condition_t}
		\norm{\hat{B}-B^\natural}_F \lesssim \max\big\{1,(\tau\cdot Q)^{-2}\big\} \cdot \pospart[\bigg]{\rho_0(B^\natural) + \max\{1,u^2 \cdot \sigma(B^\natural)\} \cdot R \cdot  \frac{\sqrt{m^{(F,{\op})}(H)}}{n^{1/4}}},
	\end{equation}
	where the mismatch parameters $\rho_0(B^\natural)$ and $\sigma(B^\natural)$ are defined with respect to the ``lifted'' random pair $(xx^T-\E[x x^T], y)$.
\end{corollary}
Obviously, one can derive analogous estimation guarantees for the local and conic complexity parameters based on Theorem~\ref{thm_eg_local} and Corollary~\ref{cor_eg_conic}, respectively.
\begin{remark}
	From a practical viewpoint, the error bound for the lifted Lasso in Corollary~\ref{prop_lifted_eg_global} is only of indirect interest, since the actual goal is to construct an estimator for an appropriate target vector $\beta^\natural \in \R^p$ with $B^\natural = \beta^\natural (\beta^\natural)^T$.
	For this purpose, one may simply extract the rank-one component from a solution $\hat{B}$ to \eqref{LLS_H}. Indeed, let $\hat{\beta} \in \Sphere^{p-1}$ be a unit-norm eigenvector of $\hat{B}$ corresponding to the largest eigenvalue $\hat{\lambda}_1$ of $\hat{B}$ (recall that $\hat{B}$ is positive semidefinite). Then, on that same event as in Corollary~\ref{prop_lifted_eg_global}, the following error bound holds true:
	\begin{equation}
		\min\Big\{\norm{\hat{\lambda}_1\hat{\beta}-\beta^\natural}_2, \norm{\hat{\lambda}_1\hat{\beta} + \beta^\natural}_2 \Big\} \lesssim \min\Big\{\norm{\beta^\natural}_2, \tfrac{\text{Err}}{\norm{\beta^\natural}_2}\Big\},
	\end{equation}
	where $\text{Err}$ is the error term on the right-hand side of \eqref{prop_lifted_eg_global_condition_t}; see~\citep[Sec.~6]{csv13} and \citep[Subsec.~2.1]{tr19} for more details.
	In other words, $\hat{\lambda}_1\hat{\beta}$ is an estimator of either $\beta^\natural$ or $-\beta^\natural$, due to sign ambiguities.
\end{remark}
Since the error bound \eqref{prop_lifted_eg_global_condition_t} is affected by the (constant) additive term $\rho_0(B^\natural)$, it is natural to study situations where it vanishes. The following proposition concerns two model setups where this is the case. In conjunction with Corollary~\ref{prop_lifted_eg_global}, this shows that the lifted Lasso \eqref{LLS_H} can provide a consistent estimator for phase-retrieval-like problems.
\begin{proposition}
\begin{rmklist}
\item
	Let $x \in \R^p$ be a random vector with finite second moments (so that the centering term $\E[xx^T]$ is well-defined). We assume that $y$ obeys a quadratic observation model of the form
	\begin{equation}
		y = (\scpr{x,\beta_0}+\nu)^2,
	\end{equation}
	where $\beta_0 \in \R^p$ is an unknown parameter vector and $\nu$ is independent noise with $\E[\nu] = 0$. Then we have that $\rho_0(B^\natural)=0$ for $B^\natural \coloneqq \beta_0 \beta_0^T$.
\item
	Let $x \sim \mathcal{N}(0, I_p)$. We assume that $y$ obeys a single-index model the form
	\begin{equation}
		y = f(\scpr{x,\beta_0}) + \nu,
	\end{equation}
	where $\beta_0 \in \Sphere^{p-1}$ is an unknown parameter vector, $f: \R \to \R$ is a scalar output function, and $\nu$ is independent noise with $\E[\nu] = 0$.
	Set $\mu \coloneqq \tfrac{1}{2} \E[f(Z)(Z^2-1)]$ with $Z \sim \mathcal{N}(0, 1)$. Then, we have that $\rho_0(B^\natural)=0$ for $B^\natural \coloneqq \mu \beta_0 \beta_0^T$.
\end{rmklist}\label{prop_quadratic_with_noise}
\end{proposition}
For the sake of brevity, we omit a proof of Proposition~\ref{prop_quadratic_with_noise}; it is straightforward but especially the second part requires some lengthy calculations, see~\citep[Appx.~B.2]{tr19} for more details.
While the first statement of Proposition~\ref{prop_quadratic_with_noise} addresses the classical phase retrieval problem under no additional assumptions on the input vector, the second one indicates that the lifted Lasso can handle much more general non-linearities, at least in the Gaussian case.
In fact, one can achieve a consistent estimator of $\pm\beta_0$ as long as $\mu \neq 0$, which includes a large subclass of even output functions.
This observation allows us to reproduce a main result of \citet{tr19}, thereby integrating it into a more general statistical framework for the lifted Lasso.
They also investigate the important special case of sparse recovery, where $\beta_0$ is sparse and $H$ is a subset of a scaled $\l{1}$-ball in $\R^{p^2}$.
A detailed analysis of this situation goes beyond the scope of this paper, but we emphasize that the complexity bounds presented in the next subsection could be used to derive results in that regard.\footnote{However, this would produce a further example of the observation that \emph{``the same $k^2$-barrier appears in most of the algorithms that have been proposed for sparse recovery from quadratic measurements''}; quote from \citep[Subsec.~2.1]{tr19}; see also \citep{ojfeh15}.}
Finally, we refer to \citep[Subsec.~1.3]{tr19} for further reading on recent approaches to phase retrieval and related problems. 

\subsection{Sparse Recovery and the Complexity of Polytopes}
\label{subsec:discussion:compl}

In this subsection, we discuss our complexity parameters in the context of high-dimensional estimation problems where $n \ll p$, with a particular emphasis on sparse recovery; for a comprehensive introduction to high-dimensional statistics, we refer to the textbooks~\citep{bg11,fh13,htw15,ver18,wai19}.
The common ground of sparse recovery problems is the assumption that the underlying parameter vector is sparse in a certain sense.
In this part, we focus on the specific case where the target vector $\beta^\natural \in \R^p$ is \emph{$k$-sparse}, i.e., at most~$k$ of its coordinates are non-zero.
Since the set of $k$-sparse vectors in $\R^p$ is non-convex for $k<p$, it cannot be used as hypothesis set for the generalized Lasso \eqref{LS_K}, and one has to come up with an appropriate convex relaxation.
Probably the most natural choice is a scaled $\l{1}$-ball, which precisely leads to the standard Lasso studied by \citet{tib96}.

Let us begin with the situation where the hypothesis set is perfectly tuned in the sense that the target vector lies exactly on its boundary.
The following result for the $\l{1}$-ball provides bounds for the local and conic m- and q-complexities in settings of Subsection~\ref{subsec:discussion:sg}--\ref{subsec:discussion:se}.
A proof is given in Appendix~\ref{sec:app:complexities}.
\begin{proposition}\label{prop_upper_bounds_complexities}
Let $t \geq 0$ and assume that $k \lesssim p$. Let $\beta^\natural \in \R^p$ be a $k$-sparse vector and let $K \coloneqq \norm{\beta^\natural}_1 B^p_1$ (i.e., $\beta^\natural$ lies on the boundary of $K$). Then the following holds true:\footnote{The notations for the m- and q-complexities are adopted from the respective settings of Subsection~\ref{subsec:discussion:sg}--\ref{subsec:discussion:se}; see also Lemma~\ref{lemma_q_and_m_global}\ref{lemma_q_and_m_global_4} and~\ref{lemma_q_and_m_global_5} for the relation to their global counterparts.}
\begin{thmproperties}
\item\label{prop_upper_bounds_complexities_20}
	The local and conic m- and q-complexities with respect to $(\gennorm_2,0)$ satisfy
	\begin{equation}
	q_{t,n}^{(2,0)}(K-\beta^\natural) \lesssim m_t^{(2,0)}(K-\beta^\natural) \lesssim \sqrt{k \log\Big(\frac{p}{k}\Big)}.
	\end{equation}
\item\label{prop_upper_bounds_complexities_2inf}
	The local and conic m- and q-complexities with respect to $(\gennorm_2,\gennorm_\infty)$ satisfy
	\begin{equation}
	q_{t,n}^{(2,\infty)}(K-\beta^\natural) \lesssim m_t^{(2,\infty)}(K-\beta^\natural) \lesssim \sqrt{k \log\Big(\frac{p}{k}\Big)\log(p)}.
	\end{equation}
\item\label{prop_upper_bounds_complexities_02}
	The local and conic m- and q-complexities with respect to $(0, \gennorm_2)$ satisfy
	\begin{equation}
		q_{t,n}^{(0,2)}(K-\beta^\natural) \lesssim m_t^{(0,2)}(K-\beta^\natural) \lesssim \sqrt{k} \cdot \log(2 p).
	\end{equation}
\end{thmproperties}
\end{proposition}
Since all upper bounds in Proposition~\ref{prop_upper_bounds_complexities} scale only logarithmically with the ambient dimension~$p$, we can conclude that sparse recovery is feasible in the settings of Subsection~\ref{subsec:discussion:sg}--\ref{subsec:discussion:se}.
Moreover, the square-root-dependence on the sparsity $k$ is optimal in each of the above cases.
In particular, the sample-size condition \eqref{thm_eg_local_condition_n} of Theorem~\ref{thm_eg_local} takes the familiar form
\begin{equation}
	n \gtrsim k \cdot \operatorname{Polylog}(p,k),
\end{equation}
where we have ignored other model-dependent parameters for the sake of clarity.
%A remarkable conclusion from Proposition~\ref{prop_upper_bounds_complexities}\ref{prop_upper_bounds_complexities_02} is that the m-complexity scales linearly with the sparsity $k$, which is very different from part~\ref{prop_upper_bounds_complexities_20} and~\ref{prop_upper_bounds_complexities_2inf}.
%This implies an unfavorable dependence on $k$ in the error bound \eqref{thm_eg_local_condition_t} of Theorem~\ref{thm_eg_local}, in the sense that $n$ must exceed~$k^2$ if the precision level~$t$ is small.
%Interestingly, such a $k^2$-bottleneck does not occur in the error bound of \citet[Thm.~3.4]{so19}, who also study (uniformly) sub-exponential input vectors (see Subsection~\ref{subsec:discussion:se}).
%However, their guarantee does not achieve an exponentially decaying probability of failure, which is due to a very different treatment of the multiplier process.

In situations where the hypothesis set is not perfectly tuned, it can be more appropriate to apply the global error bound from Corollary~\ref{cor_eg_global} instead of Theorem~\ref{thm_eg_local} (or Corollary~\ref{cor_eg_conic}). In this context, the ``skeleton'' optimization in the m- and q-complexities proves very useful.
To this end, let us assume that $K \subset \R^p$ is a convex polytope, i.e., $K = \convh(F)$ for a finite set of vertices $F \subset \R^p$. For $\alpha \in \{1,2\}$, we then have
\begin{equation} \label{eq_complexity_finiteset}
\gamma_\alpha(F,d) \lesssim \Delta(F) \cdot \big(\log(|F|)\big)^{\frac{1}{\alpha}}.
\end{equation}
This rather crude bound is proved straightforwardly by constructing an admissible partition sequence whose partitions contain all elements of $F$ as singletons ``as soon as admissible'', and by bounding every diameter trivially by $\Delta(F)$. With \eqref{eq_complexity_finiteset} at hand, we immediately obtain the following bounds for the global complexity parameters in the case of polytopal hypothesis sets:
\begin{proposition}\label{prop_polytopes}
	Let $K \subset \R^p$ be a convex polytope with $D$ vertices. Then we have
	\begin{align}
	    q_n^{(g,e)}(K) &\lesssim \frac{\Delta_e(K)\cdot \log(D)}{\sqrt{n}}+ \big(\Delta_g(K)+\Delta_e(K)\big) \cdot \sqrt{\log(D)},\\
	    m^{(g,e)}(K)  &\lesssim \Delta_e(K)\cdot \log(D)+\Delta_g(K)\cdot \sqrt{\log(D)},
	\end{align}
	where $\Delta_e(K)$ and $\Delta_g(K)$ are the diameters of $K$ with respect to the semi-norms $\gennorm_e$ and $\gennorm_g$, respectively.
\end{proposition}
Since the $\l{1}$-ball in $\R^p$ has only $D = 2p$ vertices, Proposition~\ref{prop_polytopes} implies that sparse recovery is possible in the high-dimensional regime $n \ll p$ for all variants of generic Bernstein concentration discussed in this section, even if the hypothesis set $K$ is not perfectly tuned. For example, if $\beta^\natural$ is $k$-sparse and has unit norm, $K = \sqrt{k} B^p_1$ would be a valid choice for Corollary~\ref{cor_eg_global}.
Bypassing perfect tuning is in fact a desirable feature in statistics, but we point out that the error bound \eqref{cor_eg_global_bound_t} in Corollary~\ref{cor_eg_global} exhibits a suboptimal decay rate of $O(n^{-1/4})$.

The above findings indicate a noteworthy phenomenon of our complexity parameters, and generic chaining in general.
The argument behind Proposition~\ref{prop_polytopes} is especially effective for those polytopes with few vertices because we then only have to control the empirical processes over this small subset (cf.~\eqref{eq_complexity_finiteset}).
For the $\gamma_2(\cdot, \gennorm_2)$-functional, this simplification is irrelevant, since it is equivalent to the Gaussian width according to \eqref{eq_talagrand_majorizing_measure}.
However, the general geometric mechanisms behind this fact remain largely mysterious, e.g., see~\citep[Sec.~2.4]{tal14}.
In particular, the situation is much less understood beyond this special case and the involved $\gamma$-functionals are not necessarily invariant under taking the convex hull.
Consequently, controlling the m- and q-complexities in any specific situation is a highly non-trivial task.
%For example, we have observed in Proposition~\ref{prop_upper_bounds_complexities}\ref{prop_upper_bounds_complexities_02} that unexpected effects may already appear in relatively simple scenarios and the success of sparse recovery is no longer evident for sub-exponential sample data.

\section{Conclusion and Outlook}
\label{sec:conclusion}

Leaving aside the specific aspects and applications discussed in the previous sections, the overall conclusion of our main results reads as follows: 
The benchmark case of sub-Gaussian sample data can be seen as a ``barrier'' behind which the estimation behavior of the generalized Lasso \eqref{LS_K} can change significantly.
The key difference becomes manifested in what way the complexity of the hypothesis set $K$ is measured.
Indeed, the m- and q-complexities do not generally enjoy a simple geometric interpretation similar to the Gaussian width, and except for some specific scenarios, the underlying chaining functionals are difficult to control (see Subsection~\ref{subsec:discussion:compl}).
On the other hand, we have observed that several statistical and conceptual features remain valid beyond sub-Gaussianity.
In particular, semi-parametric estimation problems can be treated as before, since the consistency of the generalized (or lifted) Lasso is not affected by the tail behavior of the input vectors (see Subsection~\ref{subsec:discussion:semipara} and~\ref{subsec:discussion:liftedlasso}).

On the technical side of this paper stands an application of generic chaining, as a means of controlling the quadratic and the multiplier process according to the underlying geometry of their index sets.
In our specific analysis, this paradigm appears explicitly in the notion of generic Bernstein concentration: the correct way to measure complexity is determined by the tail behavior of the input vector, which is captured by two (appropriate) semi-norms.

We close our discussion with a short list of open problems and possible extensions of our approach:
\begin{listing}
\item
	\textbf{Beyond sub-exponentiality.} Are our results extendable to input vectors for which generic Bernstein concentration is simply too restrictive?
	An obvious relaxation would be that~$x$ obeys only a $\alpha$-sub-exponential distribution, i.e., $\norm{x}_{\psi_\alpha} < \infty$ for some $0 <\nobreak \alpha <\nobreak 1$ (cf.~Theorem~\ref{thm_goetze}). For instance, such distributions occur naturally when studying higher-order variants of the lifted Lasso, where the input data consist of tensor products.
	We believe that our basic proof strategy would not break down in such scenarios.
	In fact, even though $\gennorm_{\psi_\alpha}$ is just a quasi-norm for $0 < \alpha < 1$, concentration inequalities are available, similarly to the case $\alpha = 1$, e.g., see~\citep{gss19,sam20,bmp20}.
	Hence, a careful adaptation of generic Bernstein concentration and the related chaining argument might lead to similar estimation guarantees as in Section~\ref{sec:mainresults}.
	It is worth pointing out that lower bounds for the quadratic process under heavier tailed inputs are subject of recent research, e.g., see~\citep{sbr15,lm17b,kr18,kc18}, while the behavior of the multiplier process remains largely unclear.
\item
	\textbf{The multiplier process.} The conclusion of the previous point gives rise to another relevant issue: How tight is our bound for the multiplier process (in Proposition~\ref{prop_multiplier_bound})?
	Can it be improved in general or at least in specific model setups? 
	Let us be a little more precise about this concern: Our approach to controlling the multiplier process is based on a powerful concentration inequality by \citet{men16}, formulated in Theorem~\ref{thm_men_subexp_bound}.
	In contrast, the multiplier process is handled with more elementary arguments in most related works on the generalized Lasso, e.g., see~\citep{pv16,tr19,so19}.
	These approaches suffer from a more pessimistic probability of success and may lead to different error bounds in some situations.
	We suspect that there exists a certain trade-off between the probability of success and the size of the related complexity terms. 
	In this regard, a particularly interesting phenomenon is that---in contrast to the sub-Gaussian case---the complexities of the multiplier process and the quadratic process may be measured in a different way.
%	While this result holds true in much greater generality, further refinements might be possible in the specific setting of this paper.
%	For example, it would be interesting to investigate whether the unfavorable behavior of the m-complexity in Proposition~\ref{prop_upper_bounds_complexities}\ref{prop_upper_bounds_complexities_02} is rather an artifact of our analysis or an actual bottleneck.
\item
	\textbf{Beyond linearity and convexity.} The results of this paper are limited to \emph{convex} hypothesis sets consisting of \emph{linear} functions.
	Indeed, convexity is an important ingredient of Fact~\ref{fact_excess}, while linearity enables the optimization over the ``skeleton'' in the proofs of Proposition~\ref{prop_quadratic_bound} and~\ref{prop_multiplier_bound}.
	We expect that it is possible to drop the convexity assumption on $K$ by analyzing the projected gradient descent method as an algorithmic implementation of \eqref{LS_K}, e.g., see \citep{os16,ors18,so19}.
	However, it is not clear to us whether our complexity terms would still adequately capture the non-convex nature of the hypothesis set; for instance, recall that the Gaussian width is invariant under taking the convex hull.
	In general, the analysis of non-convex optimization problems is very subtle, due to the possible presence of spurious local optima or saddle points. 
	On the other hand, non-convex methods often perform better in both theory and practice, e.g., see~\citep{cm19}.
	These benefits have triggered a large amount of research in the last decade, but it is fair to say that many important issues in this field remain widely open.
%	\textbf{Beyond linearity and convexity.} The results of this paper are limited to \emph{convex} hypothesis sets consisting of \emph{linear} functions.
%	Indeed, both features are crucial for our analysis; convexity is a key ingredient of Fact~\ref{fact_excess}, while linearity enables the optimization over the ``skeleton'' in the proofs of Proposition~\ref{prop_quadratic_bound} and~\ref{prop_multiplier_bound}.
%	Omitting even one of these assumptions would require a very different proof strategy.
%	Especially the analysis of non-convex optimization problems is usually much harder, due to the possible presence of spurious local optima or saddle points. 
%	On the other hand, non-convex methods often perform better in both theory and practice, e.g., see~\citet{cm19}.
%	These benefits have triggered a large amount of research in the last decade, but it is fair to say that many important issues in this field remain widely open.
\end{listing}

\section{Proofs of the Main Results}
\label{sec:proofs}

This part is dedicated to the proofs for Section~\ref{sec:intro} (provided in Subsection~\ref{subsec:proofs:prop_intro}) and Section~\ref{sec:mainresults} (provided in Subsection~\ref{subsec:proofs:implications}--\ref{subsec:proofs:cor_eg_global}).

\subsection{Implications of Generic Bernstein Concentration}
\label{subsec:proofs:implications}

We begin with two implications of generic Bernstein concentration (see Definition~\ref{def_gen_bern}), which are required for the proof of Theorem~\ref{thm_eg_local} in the next subsection, but might be also of independent interest. The proofs of both lemmas can be found in Appendix~\ref{sec:app:implications}.
The first one concerns the $q$-th moment of the marginals of a random vector that satisfies generic Bernstein concentration; recall the notation $v^*$ from Subsection~\ref{subsec:intro:notation}.
\begin{lemma} \label{lemma_bc_norm_bound}
Let $q \geq 1$. Let $x$ be a random vector in $\R^p$ that exhibits generic Bernstein concentration with respect to $(\gennorm_g,\gennorm_e)$, and we equip $\R^p$ with the pushforward measure~$\Prob \circ x^{-1}$. Then for all $v \in \R^p$, we have that
\begin{equation}
	\norm{v^*}_{L^q}\lesssim q \cdot \norm{v}_e + \sqrt{q} \cdot \norm{v}_g.
\end{equation}
\end{lemma}
The second lemma addresses the symmetrized sum of i.i.d.\ random vectors that satisfy generic Bernstein concentration. The resulting random vector still exhibits generic Bernstein concentration but with respect to different semi-norms.
\begin{lemma} \label{lemma_bern_sum} 
	Let $x$ be a random vector in $\R^p$ that exhibits generic Bernstein concentration with respect to $(\gennorm_g,\gennorm_e)$ and let $x_1,\dots,x_n$ be independent copies of $x$. Furthermore, let $\varepsilon_1,\dots,\varepsilon_n$ be independent Rademacher random variables (also independent of the $x_i$).
	Then the rescaled symmetrized sum
	\begin{equation}
		S_n \coloneqq \tfrac{1}{\sqrt{n}}\sum_{i=1}^n \varepsilon_i  x_i
	\end{equation}
	exhibits generic Bernstein concentration with respect to $\big(C(\gennorm_g+\gennorm_e),\tfrac{C}{\sqrt{n}}\gennorm_e\big)$, where $C > 0$ is a universal constant.
\end{lemma}

\subsection{Proof of Theorem~\ref{thm_eg_local}}
\label{subsec:proofs:thm_eg_local}

Throughout this subsection and unless otherwise stated, we assume that the hypotheses of Theorem~\ref{thm_eg_local} are satisfied, especially Assumption~\ref{model_setup}.
Let us recall the decomposition of the excess risk from \eqref{decomp_excess}:
\begin{equation}
\mathcal{E}(\beta,\beta^\natural) =\underbrace{\tfrac{1}{n} \sum_{i=1}^n\scpr{x_i,\beta-\beta^\natural}^2}_{=\mathcal{Q}(\beta-\beta^\natural)} + \underbrace{\tfrac{2}{n} \sum_{i=1}^n(\scpr{x_i,\beta^\natural}-y_i) \scpr{x_i, \beta-\beta^\natural}}_{=\mathcal{M}(\beta,\beta^\natural)}.
\end{equation}
According to Fact~\ref{fact_excess}, our main goal is to show that $\mathcal{E}(\beta,\beta^\natural) > 0$ for all $\beta \in K_{\beta^\natural,t}$.
For this purpose, we will first treat the quadratic and multiplier process separately in Subsection~\ref{subsec:proofs:quadraticprocess} and~\ref{subsec:proofs:multiplierprocess} below.
The outcome of this analysis are Proposition~\ref{prop_quadratic_bound} and~\ref{prop_multiplier_bound}, respectively, which eventually allows us to derive the desired error bound in Subsection~\ref{subsec:proofs:errorbounds}.
We note that some results in Subsection~\ref{subsec:proofs:quadraticprocess} and~\ref{subsec:proofs:multiplierprocess} are presented in a slightly more general setting, considering a generic set $L \subset \R^p$ instead of specific subsets of $K$.

\subsubsection{The Quadratic Process}
\label{subsec:proofs:quadraticprocess}

We now address the issue of finding a lower bound for the quadratic process $\mathcal{Q}(\beta-\beta^\natural)$. Setting $v \coloneqq \beta - \beta^\natural \in K - K$, the square root of the quadratic process takes the form
\begin{equation}
	\sqrt{\mathcal{Q}(\beta - \beta^\natural)} = \tfrac{1}{\sqrt{n}} \bigg(\sum_{i=1}^n \scpr{x_i,v}^2\bigg)^{\frac{1}{2}}.
\end{equation}
Evidently, the quadratic process describes an interaction of the input vectors $x_i$ with the difference of two hypotheses $\beta, \beta^\natural \in K$. In this sense, it is intrinsic to $K$---it does not depend in any way on~$y$, and in particular not on the model mismatch $y - \scpr{x, \beta^\natural}$ (cf.~Appendix~\ref{subsec:app:rem_main}\ref{rmk_local_mendelson}).

Since $\mathcal{Q}(\beta - \beta^\natural)$ is a non-negative empirical process, it is suited for an application of the \emph{small-ball method}.
We state a version by \citet[Prop.~5.1]{tro15} here, but it should be emphasized that the original idea is due to Mendelson (e.g., see \citep[Thm.~5.4]{men15}); recall the notion of small-ball function from Definition~\ref{def_q_theta}.
\begin{theorem}[{Small-ball method; \citep[Prop~5.1]{tro15}}]\label{thm_msbm}
	Let $L \subset \R^p$. Let $x$ be a random vector in $\R^p$ and let $x_1, \dots,x_n$ be independent copies of $x$. Then for every $\theta > 0$ and $u > 0$, we have that 
	\begin{equation}
		\inf_{v \in L} \bigg(\sum_{i=1}^n \scpr{x_i,v}^2\bigg)^{\frac{1}{2}} \geq \theta \cdot \sqrt{n} \cdot Q_{2\theta}(L,x) - 2 W_n(L,x) - \theta \cdot u
	\end{equation}
	with probability at least $1-\exp(-u^2 / 2)$, where
	\begin{equation}\label{def_empirical_width}
		W_n(L,x) \coloneqq \E\bigg[\sup_{v \in L} \SCPR[\Big]{\tfrac{1}{\sqrt{n}} \sum_{i=1}^n \varepsilon_i x_i, v}\bigg]
	\end{equation}
	is the \emph{empirical width} of $L$ with independent Rademacher random variables $\varepsilon_1,\dots,\varepsilon_n$. 
\end{theorem}
In fact, Theorem~\ref{thm_msbm} is a remarkable result because it holds true without a strong tail assumption on $x$. However, its significance is closely linked to finding an appropriate (upper) bound for the empirical width $W_n(L,x)$, which is usually not a simple task.
In the specific context of this paper, where $x$ exhibits generic Bernstein concentration, the following \emph{generic chaining bound} by Talagrand will prove useful; recall the $\gamma$-functional from Definition~\ref{def_gamma}.
\begin{theorem}[{\citep[Thm.~2.2.23]{tal14}}] \label{thm_tal_bern_bound}
Let $d_1,d_2$ be two pseudo-metrics on a set $L$. Consider a real-valued stochastic process $(X_v)_{v \in L}$ which satisfies the increment condition
\begin{equation} \label{tal_bern_bound_ic}
	\Prob(|X_{v_1}-X_{v_2}| \geq t) \leq 2 \exp\Big(-\min\Big\{\tfrac{t^2}{d_2(v_1,v_2)^2}, \tfrac{t}{d_1(v_1,v_2)}\Big\}\Big)
\end{equation}
for all $v_1, v_2 \in L$ and all $t >0$.
Then, we have that
\begin{equation}
	\E \Big[\sup_{v_1,v_2 \in L} |X_{v_1}-X_{v_2}|\Big] \lesssim \gamma_1(L,d_1)+\gamma_2(L,d_2).
\end{equation}
Moreover, if $(X_v)_{v \in L}$ is symmetric, we have that
\begin{equation} \label{cor_tal_bern_bound_ineq} 
	\E \Big[\sup_{v \in L} X_v\Big] \lesssim \gamma_1(L,d_1)+\gamma_2(L,d_2).
\end{equation}
\end{theorem}
An appropriate combination of Theorem~\ref{thm_msbm} and Theorem~\ref{thm_tal_bern_bound} leads to the following lower bound for the quadratic process:
\begin{proposition}
\label{prop_quadratic_bound}
Let $x, x_1, \dots,x_n$, $K\subset \R^p$, and $\tau>0$ be as in Assumption~\ref{model_setup}. For $t>0$, let $L \subset (K-K) \cap t\Sphere^{p-1}$. 
Then for every $u>0$, we have that
\begin{equation}
	\inf_{v \in L} \bigg(\sum_{i=1}^n \scpr{x_i,v}^2\bigg)^{\frac{1}{2}} \geq t \cdot\Big(\sqrt{n} \cdot \tau \cdot Q_{2\tau}(K^\Delta,x) -C_Q\cdot q_{t,n}^{(g,e)}(L)-\tau \cdot u\Big)
\end{equation}
with probability at least $1-\exp(-u^2 / 2)$, where $C_Q > 0$ is a universal constant.
\end{proposition}
\begin{proof}
Let $\theta \coloneqq t\cdot \tau$. Then, Theorem~\ref{thm_msbm} states that
\begin{equation}
	\inf_{v \in L} \bigg(\sum_{i=1}^n \scpr{x_i,v}^2\bigg)^{\frac{1}{2}} \geq t \cdot \Big( \sqrt{n} \cdot \tau \cdot Q_{2\theta}(L,x)- \frac{2}{t} \cdot W_n(L,x)- \tau \cdot u \Big)
\end{equation}
holds true with probability at least $1-\exp(-u^2 / 2)$. %, where $W_n$ is the empirical width from \eqref{def_empirical_width}.
The claim of Proposition~\ref{prop_quadratic_bound} follows from the bounds on $Q_{2\theta}(L,x)$ and $W_n(L,x)$ that we establish in the following.

\paragraph{Lower bound for $Q_{2\theta}$:} We have that
\begin{align}
	Q_{2\theta}(L,x)&=\inf_{v \in L} \Prob(|\scpr{x,v}| \geq 2\theta) \\*
	& \geq\inf_{v \in (K-K)\,\cap\, (t\Sphere^{p-1})} \Prob(|\scpr{x,v}| \geq 2\theta)\\*
	&\geq\inf_{\tilde{v} \in \spann(K-K)\,\cap\, \Sphere^{p-1}} \Prob(|\scpr{x,t \tilde{v}}| \geq 2t\cdot \tau) \\*
	&=\inf_{\tilde{v} \in K^\Delta} \Prob(|\scpr{x,\tilde{v}}| \geq 2\tau) =Q_{2\tau}(K^\Delta,x).
\end{align}

\paragraph{Upper bound for $W_n$:} According to the definition of the local q-complexity, there exists a set $\tilde{S}\subset \R^p$ with
\begin{equation}
	\convh(\tilde{S}) \supset L \cap t \Sphere^{p-1} \ ( = L)
\end{equation}
that satisfies
\begin{equation}\label{prop_quadratic_bound_2_approx_property}
	\frac{\gamma_1(\tilde{S}, \gennorm_e)}{\sqrt{n}}+\gamma_2(\tilde{S}, \gennorm_g+\gennorm_e) \leq 2t\cdot q_{t,n}^{(g,e)}(L).
\end{equation}
Now let $v\in L$. Since $L \subset \convh(\tilde{S})$, the point $v$ can be expressed as a convex combination in $\tilde{S}$:
\begin{equation}\label{v_convex_combination}
	v=\sum_{j=1}^M \lambda_j s_j, \quad \text{where } \lambda_j \geq 0, \quad\sum_{j=1}^M \lambda_j =1, \quad s_j \in \tilde{S}.
\end{equation}
Conditioning on the random variables $\varepsilon_i$, $x_i$, the function
\begin{equation}
	h: \R^p \to \R, \quad w \mapsto \SCPR[\Big]{\tfrac{1}{\sqrt{n}} \sum_{i=1}^n \varepsilon_i x_i,\, w}
\end{equation}
is linear. Hence, we have that
\begin{equation}
	h(v) =\sum_{j=1}^M \lambda_j h(s_j) \in \convh\big(\{h(s_1),\dots,h(s_M)\}\big).
\end{equation}
In particular, at least one of the $h(s_i)$ is not smaller than $h(v)$. This implies
\begin{equation}
	\sup_{v \in L} \SCPR[\Big]{\tfrac{1}{\sqrt{n}} \sum_{i=1}^n \varepsilon_i x_i,\, v} \leq \sup_{v \in \tilde{S}} \SCPR[\Big]{\tfrac{1}{\sqrt{n}} \sum_{i=1}^n \varepsilon_i x_i,\, v}
\end{equation}
and therefore $W_n(L,x) \leq W_n(\tilde{S},x)$. To obtain an upper bound for $W_n(\tilde{S},x)$, we consider the associated stochastic process
\begin{equation}
	(X_v)_{v \in \tilde{S}} \coloneqq \bigg(\SCPR[\Big]{\tfrac{1}{\sqrt{n}} \sum_{i=1}^n \varepsilon_i x_i, \, v}\bigg)_{v \in \tilde{S}}
\end{equation}
and intend to apply Theorem~\ref{thm_tal_bern_bound}: Since the $x_i$ and the $\varepsilon_i$ are independent, the distribution of $(X_v)_{v \in \tilde{S}}$ only depends on the individual distributions of the $x_i$ and the $\varepsilon_i$. Observing that
\begin{equation}
	(-X_v)_{v \in \tilde{S}} = \bigg(\SCPR[\Big]{\tfrac{1}{\sqrt{n}} \sum_{i=1}^n (-\varepsilon_i) x_i,\, v } \bigg)_{v \in \tilde{S}}
\end{equation}
and that $-\varepsilon_i$ has the same distribution as $\varepsilon_i$ for each $i$, we conclude that $(X_v)_{v \in \tilde{S}}$ is indeed symmetric.
Regarding the increment condition, let $v_1, v_2 \in \tilde{S}$ and set $v \coloneqq v_1-v_2$. Then, Lemma~\ref{lemma_bern_sum} implies that
\begin{equation}
	\Prob(|X_{v_1}-X_{v_2}| \geq t) \leq 2\exp\Big(-\min\Big\{\tfrac{t^2}{C^2(\norm{v}_g+\norm{v}_e)^2},\tfrac{\sqrt{n}t}{C\norm{v}_e}\Big\}\Big).
\end{equation}
Finally, Theorem~\ref{thm_tal_bern_bound} yields
\begin{align}
	W_n(\tilde{S},x) &\lesssim \gamma_1(\tilde{S}, \tfrac{1}{\sqrt{n}}\gennorm_e)+\gamma_2(\tilde{S},\gennorm_g+\gennorm_e)\\*
	&= \frac{\gamma_1(\tilde{S}, \gennorm_e)}{\sqrt{n}}+\gamma_2(\tilde{S},\gennorm_g+\gennorm_e) \stackrel{{\eqref{prop_quadratic_bound_2_approx_property}}}{\leq} 2t\cdot q_{t,n}^{(g,e)}(L),
\end{align}
and therefore $W_n(L,x) \leq W_n(\tilde{S},x) \lesssim t\cdot q_{t,n}^{(g,e)}(L)$.
\end{proof}

\subsubsection{The Multiplier Process}
\label{subsec:proofs:multiplierprocess}

We now turn our attention to the multiplier process. Setting $v \coloneqq \beta-\beta^\natural$ and $\xi_i \coloneqq \scpr{x_i,\beta^\natural}-\nobreak y_i$, the process takes the form 
\begin{equation}
	\mathcal{M}(\beta,\beta^\natural)=\tfrac{2}{n}\sum_{i=1}^n \xi_i \cdot \scpr{x_i, v}.
\end{equation}
Unlike the quadratic process, the multiplier process is not intrinsic to the hypothesis set~$K$, but (empirically) describes an interaction of the difference of two hypotheses $\beta, \beta^\natural \in K$ with the model mismatch $\xi \coloneqq \scpr{x,\beta^\natural} - y$ (cf.~Appendix~\ref{subsec:app:rem_main}\ref{rmk_local_mendelson}).

In order to control the multiplier process, we adapt another result by \citet{men16}, which is based on a refined chaining approach: Instead of applying the traditional generic chaining to the function class $\{\xi \cdot \scpr{\cdot,\beta}\MID \beta \in K\}$, Mendelson isolates the effect of the multiplier term $\xi$, which leads to a bound in terms of $\norm{\xi}_{L^q}$ and geometric properties of the class $\{\scpr{\cdot,\beta}\MID \beta \in K\}$. In fact, his result holds true for more general (non-linear) function classes, but in view of the objectives of this article, we only recite the special case of linear functions. In order to state this result, several definitions are required.
\begin{definition}[\protect{\citep[Def.~1.6]{men16}}] 
	For a real-valued random variable $Z$ and $q \geq 1$, we define the \emph{$(q)$-norm} by
	\begin{equation}
		\norm{Z}_{(q)} \coloneqq \sup_{1 \leq r \leq q} \frac{\norm{Z}_{L^r}}{\sqrt{r}}.
	\end{equation}
\end{definition} 
It is worth comparing the above definition to the moment characterization of sub-Gaussian variables (see Proposition~\ref{prop_sg_se_properties}\ref{prop_sg_se_mom}). \citet[p.~3658]{men16} remarks that the $(q)$-norm \emph{``measure[s] the subgaussian behaviour of the functions involved, but only up to a fixed level, rather than at every level''}.
\begin{definition} \label{def_admiss_seq}
	Let $L$ be a set. We call a sequence $(L_s)_{s \in \N} \subset 2^L$ of subsets of $L$ an \emph{admissible approximation sequence} if $|L_0|=1$ and $|L_s| \leq 2^{2^s}$ for $s \geq 1$. % Given an admissible approximation sequence, we denote by $\pi_s : L \to L_s$ a map that selects for each $v \in L$ to a closest point in $L_s$ with respect to $d$.
\end{definition}
The following definition introduces a relative of Talagrand's $\gamma$-functional (see Definition~\ref{def_gamma}). For this purpose, also recall the notation of dual vectors from Subsection~\ref{subsec:intro:notation}; more precisely, we equip $\R^p$ with the pushforward measure $\Prob \circ x^{-1}$ of a (generic) random vector $x \in \R^p$, so that for every $v \in \R^p$, we have $\norm{v^*}_{L^q}^q = \E[|\scpr{x, v}|^q]$.
\begin{definition}[\protect{\citep[Def.~1.7]{men16}}] For $L \subset \R^p$ and $u \geq 1$, we define
\begin{equation}
	\tilde{\Lambda}_u(L,x) \coloneqq \inf \Big\{\sup_{v \in L} \norm{(\pi_0v)^*}_{(u^2)} + \sup_{v \in L} \sum_{s \geq 0} 2^{s/2}\cdot \norm{v^*-(\pi_s v)^*}_{(u^2 2^s)}\Big\},
\end{equation}
where the infimum is taken over all admissible approximation sequences $(L_s)_{s \in \N}$ and $(\pi_s v)^*$ is a nearest point to $v^*$ in $(L_s)^*$ with respect to the $(u^2 2^s)$-norm. %\footnote{These expressions are consistent with Definition~\ref{def_admiss_seq} when $d$ is the semi-norm on $\R^p$ induced by $\gennorm_{}$}
\end{definition}
With these definitions at hand, we can now state Mendelson's result, which provides a powerful concentration inequality for multiplier processes.
We emphasize that the feature vector $x$ and the multiplier $\xi$ are \emph{not} necessarily independent here, which is crucial for our analysis and an important difference to related results in the literature, e.g., see~\citep{hw19}.
\begin{theorem}[{\citep[Thm.~1.9]{men16}}] \label{thm_men_general_bound}
Let $L\subset \R^p$ and let $(x, \xi) \in \R^p \times \R$ be a random pair such that $\norm{\xi}_{L^q} < \infty$ for some $q > 2$. We assume that $(x_1, \xi_1), \dots, (x_n, \xi_n)$ are independent copies of $(x, \xi)$. Then there exist constants $C_0,C_1,\dots,C_4 > 0$ (only depending on $q$) such that for every $w,u > C_0$, the following holds true with probability at least $1-C_1\cdot w^{-q} \cdot n^{-(q/2)+1}\cdot \log_q(n)-4\exp(-C_2\cdot u^2)$:
\begin{equation}
	\sup_{v \in L} \bigg| \sum_{i=1}^n (\xi_i \cdot \scpr{x_i,v} - \E[\xi \cdot v^*]) \bigg| \leq C_3 \cdot w \cdot u \cdot \sqrt{n} \cdot \norm{\xi}_{L^q} \cdot \tilde{\Lambda}_{C_4 u}(L).
\end{equation}
\end{theorem}
The term $C_1 \cdot w^{-q} \cdot n^{-(q/2)+1}\cdot \log_q(n)$ in the probability of success arises from a concentration inequality for the random vector $(\xi_i)_{i=1}^n \in \R^n$, for which we only assume that the $q$-th moment of its components exists for some $q>2$. In fact, better rates can be achieved by more restrictive assumptions on the tails of $\xi$. For example, Mendelson proves a sub-Gaussian variant of Theorem~\ref{thm_men_general_bound} using Bernstein's inequality (see \citep[Thm.~4.4]{men16}). If we assume that $\xi$ is just sub-exponential (as in Assumption~\ref{model_setup}), Bernstein's inequality cannot be applied to the squared coordinates appearing in the Euclidean norm of $(\xi_i)_{i=1}^n$. However, the following recent result of \citet{gss19} allows us to derive a concentration inequality in the sub-exponential case:
\begin{theorem}[{\citep[Prop.~1.1]{gss19}}] \label{thm_goetze}
	Let $X_1,\dots,X_n$ be independent, centered random variables with $\sigma_i^2 \coloneqq \E[X_i^2] < \infty$ and $\norm{X_i}_{\psi_\alpha} \leq R$ for some $\alpha \in (0,1] \cup \{2\}$.\footnote{As discussed by \citet{gss19}, our Definition~\ref{def_sg_se} can be extended to the more general case $\alpha>0$, leading to the notion of \emph{$\alpha$-sub-exponential random variables}. But note that if $0< \alpha <1$, the exponential Orlicz ``norm'' $\gennorm_{\psi_\alpha}$ violates the triangle inequality.} Let $B = [b_{ij}] \in \R^{n \times n}$ be a symmetric matrix. Then there exists a universal constant $C > 0$ such that for every $t>0$, we have that
	\begin{equation}
		\Prob\bigg( \bigg| \sum_{i,j=1}^n b_{ij} X_i X_j - \sum_{i=1}^n \sigma_i^2 b_{ii} \bigg| \geq t \bigg) \leq 2 \exp \Big( - C \cdot \min\Big\{\tfrac{t^2}{R^4 \norm{B}^2_F}, \Big( \tfrac{t}{R^2 \norm{B}_{\op}}\Big)^{\frac{\alpha}{2}} \Big\} \Big).
	\end{equation}
\end{theorem}
\begin{corollary} \label{cor_xi_conc}
Let $\xi_1,\dots,\xi_n$ be i.i.d.\ sub-exponential random variables. Then there exists a universal constant $C > 0$ such that with probability at least $1-2\exp(-C \cdot \sqrt{n})$, we have that
\begin{equation}
	\norm{(\xi_i)_{i=1}^n}_2=\bigg(\sum_{i=1}^n \xi_i^2 \bigg)^{\frac{1}{2}} \lesssim\sqrt{n} \cdot \norm{\xi_1}_{\psi_1}.
\end{equation}
\end{corollary}
\begin{proof}
We apply Theorem~\ref{thm_goetze} for $X_i \coloneqq \xi_i - \E[\xi_i]$, $B \coloneqq I_n$, $t \coloneqq R^2\cdot n$, and obtain
\begin{align}
	\Prob(E) \coloneqq \Prob \bigg( \bigg|\sum_{i=1}^n (\xi_i - \E[\xi_i])^2 - n\cdot \sigma_1^2\bigg| \geq n \cdot \norm{\xi_1 - \E[\xi_1]}^2_{\psi_1} \bigg) &\leq 2 \exp(- C \cdot \min\{n,\sqrt{n} \} ) \\*
	&=2 \exp(- C \cdot \sqrt{n} ).
\end{align}
Let us assume that the complement of the event $E$ has occurred. Then, using Proposition~\ref{prop_sg_se_properties}\ref{prop_sg_se_mom}, it follows that
\begin{equation}
\sum_{i=1}^n (\xi_i - \E[\xi_i])^2 < n\cdot( \norm{\xi_1 - \E[\xi_1]}^2_{\psi_1}+\sigma_1^2) \lesssim n \cdot \norm{\xi_1 - \E[\xi_1]}^2_{\psi_1} \lesssim  n \cdot \norm{\xi_1}^2_{\psi_1},
\end{equation}
where the last step is due to 
\begin{equation}
	\norm{X-\E[X]}_{\psi_\alpha} \lesssim \norm{X}_{\psi_\alpha}
\end{equation}
for $\alpha \in \{1,2\}$; see~\citep[Lem.~2.6.8]{ver18} for a proof, which also works for $\alpha=1$.
Consequently, we have that
\begin{equation}
	\norm{(\xi_i)_{i=1}^n}_2 \leq \norm{(\xi_i- \E[\xi_i])_{i=1}^n}_2+\norm{(\E[\xi_i])_{i=1}^n}_2 \lesssim \sqrt{n} \cdot \norm{\xi_1}_{\psi_1}.
\end{equation}
\end{proof}

The bound of Corollary~\ref{cor_xi_conc} leads to the following sub-exponential version of Theorem~\ref{thm_men_general_bound}:
\begin{theorem} \label{thm_men_subexp_bound}
	Let $L\subset \R^p$ and let $(x, \xi) \in \R^p \times \R$ be a random pair such that $\norm{\xi}_{\psi_1} < \infty$. We assume that $(x_1, \xi_1), \dots, (x_n, \xi_n)$ are independent copies of $(x, \xi)$.
	Then there exist universal constants $C, C', \tilde{C}>\nolinebreak 0$ such that for every $u \geq 8$, the following holds true with probability at least $1-2\exp(-C\cdot \sqrt{n})-4\exp(-C\cdot u^2)$:
	\begin{equation}
		\sup_{v \in L} \bigg| \sum_{i=1}^n (\xi_i \cdot \scpr{x_i,v} - \E[\xi \cdot v^*]) \bigg| \leq C' \cdot u \cdot \sqrt{n} \cdot \norm{\xi}_{\psi_1} \cdot \tilde{\Lambda}_{\tilde{C} u}(L,x).
	\end{equation}
\end{theorem}
\begin{proof}
Analogously to the proof of \citep[Thm.~4.4]{men16}, it is enough to adapt the last step of the proof of \citep[Thm.~1.9]{men16}. To this end, we set the variables arising in the proof of \citep[Thm.~4.4]{men16} to the values $q \coloneqq 6$ and $w \coloneqq 1$, which entails $r=r'=2$ and $q_1= 8$. Then, with probability at least $1-2\exp(-C\cdot \sqrt{n})$, Corollary~\ref{cor_xi_conc} implies that
\begin{equation}
	\norm{(\xi_i)_{i=1}^n}_2 \lesssim \sqrt{n} \cdot \norm{\xi_1}_{\psi_1}.
\end{equation}
Since $\norm{\xi}_{L^6} \lesssim \norm{\xi}_{\psi_1}$, we also have that
\begin{equation}
	\bigg( \sum_{i \geq j_0}(\xi^*_i)^{2r} \bigg)^{\frac{1}{2r}} = \bigg( \sum_{i \geq j_0}(\xi^*_i)^{4} \bigg)^{\frac{1}{4}} \lesssim \norm{\xi}_{L^6} \cdot n^{1/4} \lesssim \norm{\xi}_{\psi_1} \cdot n^{1/4},
\end{equation}
with probability at least $1-2 \exp(-C \cdot u^2)$, where $\xi^*$ and $j_0$ are objects defined in the proof of \citep[Thm.~1.9]{men16}. The rest of the proof remains unchanged.
\end{proof}

The following lemma is a centerpiece of our statistical analysis, as it allows us to control the complexity term $\tilde{\Lambda}_u(L,x)$ via generic Bernstein concentration:
\begin{lemma} \label{lemma_lambda_tilde}
Let $L \subset \R^p$ with $0 \in \convh(L)$ and $u \geq 1$. Let $x$ be a random vector in $\R^p$ that exhibits generic Bernstein concentration with respect to $(\gennorm_g,\gennorm_e)$. Then, we have that
\begin{equation}
	\tilde{\Lambda}_u(L,x) \lesssim u \cdot \gamma_1(L,\gennorm_e) + \gamma_2(L,\gennorm_g).
\end{equation}
\end{lemma}
\begin{proof}
	From Lemma~\ref{lemma_bc_norm_bound}, we obtain
	\begin{equation} \label{lemma_lambda_tilde_norm_bound}
		\norm{v^*}_{(u^2 2^s)} = \sup_{1 \leq q \leq u^2 2^s} \frac{\norm{v^*}_{L^q}}{\sqrt{q}} \lesssim \sup_{1 \leq q \leq u^2 2^s} \frac{q \cdot \norm{v}_e + \sqrt{q} \cdot \norm{v}_g}{\sqrt{q}} = u 2^{s/2} \cdot \norm{v}_e + \norm{v}_g 
	\end{equation}
	for every $v \in \R^p$.
	Adopting the notation from \citep[Def.~1.7]{men16}, we set
	\begin{equation}
		\Lambda_u(L,x) \coloneqq \inf \ \sup_{v \in L} \sum_{s \geq 0} 2^{s/2}\cdot \norm{v^*-(\pi_s v)^*}_{(u^2 2^s)},
	\end{equation}
	which implies that
	\begin{equation} \label{lemma_lambda_tilde_lambda}
		\tilde{\Lambda}_u(L,x) \leq \Lambda_u(L,x) + \sup_{v \in L} \norm{v^*}_{(u^2)}.
	\end{equation}
	According to \eqref{lemma_lambda_tilde_norm_bound}, the second summand of \eqref{lemma_lambda_tilde_lambda} can bounded as follows:
	\begin{align}
		\sup_{v \in L} \norm{v^*}_{(u^2)} &\lesssim \sup_{v \in L} (u \cdot \norm{v}_e + \norm{v}_g) \leq \sup_{v \in L} (u \cdot \norm{v}_e) +\sup_{v \in L} \norm{v}_g\\*
		& \stackrel{\mathllap{0 \in \convh(L)}}{\leq} u \cdot \Delta_e(L) + \Delta_g(L) \leq u \cdot \gamma_1(L,\gennorm_e) + \gamma_2(L,\gennorm_g),
	\end{align}
	where $\Delta_e(L)$ and $\Delta_g(L)$ are the diameters of $L$ with respect to $\gennorm_e$ and $\gennorm_g$, respectively. To handle the first summand of \eqref{lemma_lambda_tilde_lambda}, we apply \eqref{lemma_lambda_tilde_norm_bound} once again:
	\begin{align}
		\Lambda_u(L,x) &=\inf \ \sup_{v \in L} \sum_{s \geq 0} 2^{s/2}\cdot \norm{v^*-(\pi_s v)^*}_{(u^2 2^s)} \\*
		&\leq \inf \ \sup_{v \in L} \sum_{s \geq 0} 2^{s/2}\cdot \norm{v^*-(\pi^{+}_s v)^*}_{(u^2 2^s)} \\*
		&\lesssim \inf \ \sup_{v \in L} \sum_{s \geq 0} \big(u 2^s \cdot \norm{v-\pi^{+}_s v}_e + 2^{s/2} \cdot \norm{v-\pi^{+}_s v}_g\big),
	\end{align}
	where $\pi^{+}_s : L \to L_s$ is an arbitrary map (depending on the respective admissible approximation sequence $(L_s)_{s \in \N}$ indexed by the infimum);
	note that we can indeed replace $\pi_s$ by $\pi^+_s$ here, since by definition, $(\pi_s)_{s \in \N}$ is an optimal (functional-minimizing) sequence of projections with respect to the $(u^2 2^s)$-norms.
	
	We now show that the above expression is upper bounded by 
	\begin{equation}
		5\cdot \Big(u\cdot \gamma_1(L,\gennorm_e) + \gamma_2(L,\gennorm_g)\Big),
	\end{equation}
	which would imply the claim of Lemma~\ref{lemma_lambda_tilde}. For this purpose, let $(\mathcal{E}_s)_{s \in \N}$ and $(\mathcal{G}_s)_{s \in \N}$ be two admissible partition sequences which approximate $\gamma_1(L,\gennorm_e)$ and $\gamma_2(L,\gennorm_g)$ up to a factor of~$2$, respectively. Furthermore, let $(\mathcal{F}_s)_{s \in \N}$ be given by $\mathcal{F}_0=\{L\}$ and
	\begin{equation} \label{tal_bern_bound_def_f}
		\mathcal{F}_s \coloneqq \{E \cap G \MID E \in \mathcal{E}_{s-1}, G \in \mathcal{G}_{s-1} \} \quad \text{ for } s \geq 1.
	\end{equation}
	It is not hard to see that $(\mathcal{F}_s)_{s \in \N}$ is indeed an admissible partition sequence.
	
	Next, we use the sequence $(\mathcal{F}_s)_{s \in \N}$ to construct an admissible approximation sequence $(L_s)_{s \in \N}$ and a corresponding sequence of maps $(\pi_s^\circ)_{s \in \N}$: For each $s \in \N$, the set $L_s \subset L$ is obtained by selecting exactly one (arbitrary) point $v_F$ from every $F \in \mathcal{F}_s$, while $\pi_s^{\circ}$ maps every point in $F$ to the respective $v_F$. This construction ensures that for $s \geq 1$ and $v \in L$, we have
	\begin{equation}
		\norm{v-\pi^{\circ}_sv}_e \leq \Delta_e(F_s(v)) \leq \Delta_e(E_{s-1}(v))
	\end{equation}
	and
	\begin{equation}
		\norm{v-\pi^{\circ}_sv}_g \leq \Delta_g(F_s(v)) \leq \Delta_g(G_{s-1}(v)).
	\end{equation}
	This implies
	\begin{align}
		\Lambda_u(L,x) &\lesssim \sup_{v \in L} \sum_{s \geq 0} \big(u 2^s \cdot \norm{v-\pi^{\circ}_s v}_e + 2^{s/2} \cdot \norm{v-\pi^{\circ}_s v}_g\big) \\*
		&\leq u \cdot \sup_{v \in L} \sum_{s \geq 0} 2^s \cdot \norm{v-\pi^{\circ}_s v}_e +\sup_{v \in L} \sum_{s \geq 0} 2^{s/2} \cdot \norm{v-\pi^{\circ}_s v}_g \\
		&\leq u \cdot \Big(\Delta_e(L)+ \sup_{v \in L} \sum_{s \geq 1} 2^{s} \cdot \Delta_e(E_{s-1}(v))\Big) + \Delta_g(L)+\sup_{v \in L} \sum_{s \geq 1} 2^{s/2} \cdot \Delta_g(G_{s-1}(v)) \\
		&= u \cdot \Big(\Delta_e(L)+ \sup_{v \in L} \sum_{s \geq 0} 2^{s+1} \cdot \Delta_e(E_{s}(v))\Big) + \Delta_g(L)+\sup_{v \in L} \sum_{s \geq 0} 2^{(s+1)/2} \cdot \Delta_g(G_{s}(v))  \\*
		&\leq u \cdot (1+4) \cdot \gamma_1(L,\gennorm_e)+ (1+2\sqrt{2})\cdot \gamma_2(L,\gennorm_g),
	\end{align}
	where we have used in the last line that $(\mathcal{E}_s)_{s \in \N}$ and $(\mathcal{G}_s)_{s \in \N}$ approximate $\gamma_1(L,\gennorm_e)$ and $\gamma_2(L,\gennorm_g)$ up to a factor of~$2$, respectively.
\end{proof}
\begin{remark}\label{rem_log_concave}
	The proof of Lemma~\ref{lemma_lambda_tilde} is inspired by \citep[Subsec.~4.3]{men16}, where the upper bound  $\tilde{\Lambda}_u(L,x) \lesssim u \cdot \gamma_1(L, \gennorm_\infty)+\gamma_2(L, \gennorm_2)$ is derived under the assumption that $x$ obeys an unconditional, isotropic, log-concave distribution. This assumption implies that $x$ is stochastically dominated by a random vector with i.i.d.\ standard exponential coordinates, which enables a bound for $\norm{v^*}_{(q)}$ in terms of $\norm{v}_2$ and $\norm{v}_\infty$ (cf.~Subsection~\ref{subsec:discussion:indep_se_features}).
\end{remark}

The estimate from Lemma~\ref{lemma_lambda_tilde} leads us to our final result for the multiplier process:
\begin{proposition} \label{prop_multiplier_bound}
	Let $L\subset t\Sphere^{p-1}$ for some $t > 0$ and let $(x, \xi) \in \R^p \times \R$ be a random pair such that $\norm{\xi}_{\psi_1} < \infty$ and $x$ exhibits generic Bernstein concentration with respect to $(\gennorm_g,\gennorm_e)$. We assume that $(x_1, \xi_1), \dots, (x_n, \xi_n)$ are independent copies of $(x, \xi)$.
	Then there exist universal constants $C, C' >0$ such that for every $u \geq 8$, the following holds true with probability at least $1-2\exp(-C\cdot \sqrt{n})-4\exp(-C\cdot u^2)$:
	\begin{equation}
		\sup_{v \in L} \bigg| \sum_{i=1}^n (\xi_i \cdot \scpr{x_i,v} - \E[\xi \cdot v^*]) \bigg| \leq C' \cdot t \cdot u^2 \cdot \sqrt{n} \cdot \norm{\xi}_{\psi_1} \cdot m_t^{(g,e)}(L).
	\end{equation}
\end{proposition}
\begin{proof}
According to the definition of the local m-complexity, there exists a set $\tilde{S}\subset \R^p$ with
\begin{equation}
	\convh(\tilde{S}) \supset (L \cap t \Sphere^{p-1}) \cup \{0\} \ (= L \cup \{0\})
\end{equation}
that satisfies
\begin{equation}\label{prop_multiplier_bound_2_approx_property}
\gamma_1(\tilde{S}, \gennorm_e)+\gamma_2(\tilde{S}, \gennorm_g) \leq 2t \cdot m_t^{(g,e)}(L).
\end{equation}
By Theorem~\ref{thm_men_subexp_bound}, with probability at least $1-2\exp(-C\cdot \sqrt{n})-4\exp(-C\cdot u^2)$, we have that
\begin{equation}\label{prop_multiplier_bound_upper_bound}
\sup_{v \in \tilde{S}} \bigg| \sum_{i=1}^n (\xi_i \cdot \scpr{x_i,v} - \E[\xi \cdot v^*]) \bigg| \lesssim u \cdot \sqrt{n} \cdot \norm{\xi}_{\psi_1} \cdot \tilde{\Lambda}_{\tilde{C} u}(\tilde{S},x).
\end{equation}
Now let $v\in L$. Since $L \subset \convh(\tilde{S})$, the point $v$ can be expressed as a convex combination of points $s_1,\dots,s_M \in \tilde{S}$ as in \eqref{v_convex_combination}. Conditioning on the random variables $x_i$ and $\xi_i$, the function
\begin{equation}
	h:\R^p \to \R, \quad w \mapsto \bigg| \sum_{i=1}^n (\xi_i \cdot \scpr{x_i,w} - \E[\xi \cdot w^*])\bigg|=\bigg|\SCPR[\Big]{\sum_{i=1}^n(\xi_i x_i - \E[\xi x]),w}\bigg|
\end{equation}
is a composition of a linear function and the convex function $z \mapsto |z|$. Hence, $h$ is convex and we can apply Jensen's inequality to obtain
\begin{equation}
	h(v) = h\bigg(\sum_{j=1}^M \lambda_j s_j\bigg)\leq\sum_{j=1}^M \lambda_j h(s_j)\leq\sum_{j=1}^M \bigg(\lambda_j \cdot \sup_{w \in \tilde{S}}h(w)\bigg)=\sup_{w \in \tilde{S}}h(w).
\end{equation}
Since $v\in L$ was arbitrarily chosen, we can conclude that the following bound holds true if the event from \eqref{prop_multiplier_bound_upper_bound} has occurred:
\begin{equation}
	\sup_{v \in L} \bigg| \sum_{i=1}^n (\xi_i \cdot \scpr{x_i,v} - \E[\xi \cdot v^*]) \bigg| \lesssim u \cdot \sqrt{n} \cdot \norm{\xi}_{\psi_1} \cdot \tilde{\Lambda}_{\tilde{C} u}(\tilde{S},x).
\end{equation}
Finally, Lemma~\ref{lemma_lambda_tilde} implies
\begin{equation}
	\tilde{\Lambda}_{\tilde{C} u}(\tilde{S},x) \lesssim \tilde{C} \cdot u\cdot \Big(\gamma_1(\tilde{S},\gennorm_e)+ \gamma_2(\tilde{S},\gennorm_g)\Big) \stackrel{\eqref{prop_multiplier_bound_2_approx_property}}{\lesssim} u\cdot t \cdot m_t^{(g,e)}(L),
\end{equation}
where we have also used that $u \geq 8 > 1$.
\end{proof}

\subsubsection{Controlling the Excess Risk}
\label{subsec:proofs:errorbounds}

With the results of Proposition~\ref{prop_quadratic_bound} and~\ref{prop_multiplier_bound} at hand, we are now ready to prove Theorem~\ref{thm_eg_local}.
Let us first consider the case $t>0$. According to Fact~\ref{fact_excess}, it suffices to show that $\mathcal{E}(\beta,\beta^\natural)>0$ for all $\beta \in K_{\beta^\natural, t}$.
Remarkably, this argument is actually the only point in our proof where we rely on the convexity of the hypothesis set $K$.
The remainder of the proof is divided into several substeps.

\paragraph{Step 1 (quadratic process):} 
Applying Proposition~\ref{prop_quadratic_bound} to $L \coloneqq K_{\beta^\natural,t} - \beta^\natural = (K - \beta^\natural) \cap t\Sphere^{p-1}$, the following holds with probability at least $1-\exp(-u^2 / 2)$: For every $\beta \in K_{\beta^\natural,t}$, we have that
\begin{align}
	\sqrt{\mathcal{Q}(\beta - \beta^\natural)}&= \bigg(\tfrac{1}{n} \sum_{i=1}^n \scpr{x_i, \beta-\beta^\natural}^2\bigg)^{\frac{1}{2}} \\*
	&\geq t \cdot \bigg(\tau \cdot Q_{2\tau}(K^\Delta,x) -\frac{C_Q\cdot q_{t,n}^{(g,e)}(K - \beta^\natural)+\tau \cdot u}{\sqrt{n}}\bigg) \\*
	&\geq \tfrac{1}{2} \cdot t \cdot \tau \cdot Q_{2\tau}(K^\Delta,x),
\end{align}
where the last step follows from the condition \eqref{thm_eg_local_condition_n} for an appropriate hidden constant.

\paragraph{Step 2 (multiplier process):} 
Applying Proposition~\ref{prop_multiplier_bound} to $L \coloneqq K_{\beta^\natural,t} - \beta^\natural = (K - \beta^\natural) \cap t\Sphere^{p-1}$ and $\xi \coloneqq \scpr{x,\beta^\natural}-y$, the following holds with probability at least $1-2\exp(-C\cdot \sqrt{n})-4\exp(-C\cdot u^2)$: For every $\beta \in K_{\beta^\natural,t}$, we have that
\begin{align}
	\tfrac{1}{2}\cdot \mathcal{M}(\beta, \beta^\natural) &= \tfrac{1}{n} \sum_{i=1}^n(\scpr{x_i,\beta^\natural}-y_i) \scpr{x_i, \beta-\beta^\natural} \\*
	&\geq \E\big[(\scpr{x,\beta^\natural}-y)\scpr{x, \beta-\beta^\natural}\big] -C' \cdot u^2 \cdot \norm{\scpr{x,\beta^\natural}-y}_{\psi_1} \cdot \frac{t \cdot m_t^{(g,e)}(K - \beta^\natural)}{\sqrt{n}} \\
	&= -t \cdot \bigg( \E\Big[(y - \scpr{x,\beta^\natural})\SCPR[\Big]{x, \tfrac{\beta - \beta^\natural}{t}}\Big]+C' \cdot u^2 \cdot \norm{\scpr{x,\beta^\natural}-y}_{\psi_1} \cdot \frac{m_t^{(g,e)}(K - \beta^\natural)}{\sqrt{n}} \bigg)\\
	&\geq -t \cdot \bigg(\rho_t(\beta^\natural) + C' \cdot u^2 \cdot \sigma(\beta^\natural)\cdot \frac{m_t^{(g,e)}(K - \beta^\natural)}{\sqrt{n}}\bigg)\\*
	&\geq -t \cdot \max\{1,C'\}\cdot \bigg(\rho_t(\beta^\natural) + u^2 \cdot \sigma(\beta^\natural)\cdot \frac{m_t^{(g,e)}(K - \beta^\natural)}{\sqrt{n}}\bigg),
\end{align}
where the second inequality is due to
\begin{equation}
	\E\Big[(y - \scpr{x,\beta^\natural})\SCPR[\Big]{x, \tfrac{\beta - \beta^\natural}{t}}\Big] = \SCPR[\Big]{\E\big[(y - \scpr{x,\beta^\natural}) x\big], \underbrace{\tfrac{\beta - \beta^\natural}{t}}_{\mathclap{\in \tfrac{1}{t} (K - \beta^\natural) \cap \Sphere^{p-1} =  K^t}}} \leq \rho_t(\beta^\natural).
\end{equation}
If the aforementioned event has occurred and if $t$ satisfies the condition \eqref{thm_eg_local_condition_t} for an appropriate hidden constant, it follows that
\begin{equation}
	\mathcal{M}(\beta, \beta^\natural) \geq -\frac{t^2}{8} \cdot \big(\tau \cdot Q_{2\tau}(K^\Delta,x)\big)^2 \quad \text{for all $\beta \in K_{\beta^\natural,t}$.}
\end{equation}

\paragraph{Step 3 (excess risk):} Finally, we assume that the events from Step~1 and Step~2 have occurred jointly, which indeed happens with probability at least $1 - 2\exp(-C\cdot \sqrt{n})-5\exp(-C\cdot u^2)$ for an appropriately chosen constant $C > 0$.
Then, we obtain
\begin{align}
	\mathcal{E}(\beta, \beta^\natural)&= \mathcal{Q}(\beta - \beta^\natural)+\mathcal{M}(\beta, \beta^\natural) \\*
	&\geq  \frac{t^2}{4}  \cdot \big(\tau \cdot Q_{2\tau}(K^\Delta,x)\big)^2 - \frac{t^2}{8} \cdot \big(\tau \cdot Q_{2\tau}(K^\Delta,x)\big)^2 > 0
\end{align}
for all $\beta \in K_{\beta^\natural,t}$, which concludes the proof for $t>0$. It remains to consider the case $t=0$.

\paragraph{Step 4 ($t = 0$):} In this case, $q_{t,n}^{(g,e)}(K - \beta^\natural)$ and $m_t^{(g,e)}(K-\beta^\natural)$ correspond to the conic complexities from Definition~\ref{def_conic_q_m_complexity}. 

Applying Proposition~\ref{prop_quadratic_bound} and Proposition~\ref{prop_multiplier_bound} simultaneously (as in the preceding steps) to $L \coloneqq \cone{K-\beta^\natural}\cap \Sphere^{p-1}$ and radius $\tilde{t} \coloneqq 1$, we have that with probability at least $1 - 2\exp(-C\cdot \sqrt{n})-5\exp(-C\cdot u^2)$, both
\begin{align} 
\inf_{v \in L}\mathcal{Q}(v)\stackrel{\eqref{thm_eg_local_condition_n}}{\geq}\frac{1}{4} \cdot \big(\tau \cdot Q_{2\tau}(K^\Delta,x)\big)^2 > 0 \label{thm_eg_t_eq_0}
\end{align}
and
\begin{align}
\inf_{v \in L}\mathcal{M}(v+\beta^\natural, \beta^\natural) \geq -1 \cdot\max\{1,C'\}\cdot \bigg(\rho_0(\beta^\natural) + u^2 \cdot \sigma(\beta^\natural)\cdot \frac{m_0^{(g,e)}(K - \beta^\natural)}{\sqrt{n}}\bigg)\stackrel{\eqref{thm_eg_local_condition_t}}{\geq} 0,
\end{align}
where we have also used that $q_{1,n}^{(g,e)}(L) = q_{0,n}^{(g,e)}(K-\beta^\natural)$ and $m_{1}^{(g,e)}(L) = m_{0}^{(g,e)}(K-\beta^\natural)$.
Finally, let us assume that this event has occurred and let $\beta \in K \setminus \{\beta^\natural\}$. Then, we have
\begin{equation}
\mathcal{E}(\beta,\beta^\natural)=\norm{\beta-\beta^\natural}^2_2 \cdot \mathcal{Q}\Big(\underbrace{\tfrac{\beta-\beta^\natural}{\norm{\beta-\beta^\natural}_2}}_{\in L}\Big)+\norm{\beta-\beta^\natural}_2 \cdot \mathcal{M}\Big(\underbrace{\tfrac{\beta-\beta^\natural}{\norm{\beta-\beta^\natural}_2}}_{\in L}{}+{}\beta^\natural,\beta^\natural\Big)>0,
\end{equation}
which implies that $\beta^\natural$ is the only solution to \eqref{LS_K}. \qed

\subsection{Proof of Corollary~\ref{cor_eg_conic}}
\label{subsec:proofs:cor_eg_conic}

For $t>0$ and $L \subset \R^p$, we have that $\tfrac{1}{t} L \subset \cone{L}$. Due to the homogeneity of the semi-norms $\gennorm_g$ and $\gennorm_e$, we can rewrite the definition of $q_{t,n}^{(g,e)}(L)$ as
\begin{equation}
	q_{t,n}^{(g,e)}(L)=\inf\Big\{\tfrac{\gamma_1(S, \gennorm_e)}{\sqrt{n}}+\gamma_2(S, \gennorm_g+\gennorm_e)\MID[\big] S \subset \R^p, \convh(S) \supset \tfrac{1}{t} L \cap \Sphere^{p-1} \Big\},
\end{equation}
which coincides with the definition of $q_{0,n}^{(g,e)}(L)$ except that the infimum is taken over an inclusion-wise larger domain of sets. Therefore, it holds that $q_{t,n}^{(g,e)}(L) \leq q_{0,n}^{(g,e)}(L)$, and analogously, we have $m_t^{(g,e)}(L) \leq m_0^{(g,e)}(L)$. It follows that the replacement of the local complexities by the conic complexities (and $\rho_t(\beta^\natural)$ by $\rho_0(\beta^\natural)$) leads to stronger conditions in \eqref{thm_eg_local_condition_n} and \eqref{thm_eg_local_condition_t}, which obviously cannot harm the validity of the theorem. \qed

\subsection{Proof of Lemma~\ref{lemma_q_and_m_global}}
\label{subsec:proofs:lemma_q_and_m_global}

The claim of~\ref{lemma_q_and_m_global_1} follows from the fact that the infima in the global complexity parameters are taken over inclusion-wise smaller domains of sets. The claim of~\ref{lemma_q_and_m_global_2} follows from the fact that the affine term $v$ does not affect the pseudo-metrics induced by the semi-norms $\gennorm_g$ and $\gennorm_e$.

For the claim of~\ref{lemma_q_and_m_global_3}, we first observe that
\begin{equation}\label{gamma_functional_split}
	\gamma_2(L, \gennorm_g+\gennorm_e) \lesssim \gamma_2(L, \gennorm_g)+\gamma_2(L, \gennorm_e),
\end{equation}
which is stated as an exercise by \citet[Exc.~2.2.24]{tal14}; its proof is based on the same strategy as the proof of Lemma~\ref{lemma_lambda_tilde}: based on two admissible partition sequences that approximate $\gamma_2(L, \gennorm_g)$ and $\gamma_2(L, \gennorm_e)$ up to a factor of $2$, a third partition sequence is defined as in \eqref{tal_bern_bound_def_f}. 
Making use of \eqref{gamma_functional_split}, we obtain
\begin{align}
	q_n^{(g,e)}(L) &\lesssim \inf_S\Big\{\gamma_1(S, \gennorm_e)+ \gamma_2(S, \gennorm_g)+\gamma_2(S, \gennorm_e)\Big\} \\* &\lesssim \inf_S\Big\{\gamma_1(S, \gennorm_e)+ \gamma_2(S, \gennorm_g)\Big\} \\*
	&= m^{(g,e)}(L),
\end{align} 
where the second inequality is due to the fact that $\gamma_2(L,d) \leq \gamma_1(L,d)$ holds true for all sets~$L$ and pseudo-metrics $d$.

The claims of~\ref{lemma_q_and_m_global_4} and~\ref{lemma_q_and_m_global_5} follow directly from the respective definitions. \qed

\subsection{Proof of Corollary~\ref{cor_eg_global}}
\label{subsec:proofs:cor_eg_global}

We apply Theorem~\ref{thm_eg_local} for the precision level
\begin{equation}
	t \coloneqq \tilde{C}\cdot \max\Big\{1,\big(\tau\cdot Q_{2\tau}(K^\Delta,x)\big)^{-2}\Big\} \cdot \pospart[\bigg]{\rho_0(\beta^\natural) + \max\{1,u^2 \cdot \sigma(\beta^\natural)\}\cdot \frac{\sqrt{m^{(g,e)}(K)}}{n^{1/4}}},
\end{equation}
where the universal constant $\tilde{C} \geq 1$ will be specified later.

For this specific choice of $t$, Lemma~\ref{lemma_q_and_m_global}\ref{lemma_q_and_m_global_1}--\ref{lemma_q_and_m_global_3} imply that
\begin{align}
	q^{(g,e)}_{t,n}(K-\beta^\natural) &\leq \frac{1}{t} \cdot q_n^{(g,e)}(K) \\*
	&\leq \tilde{C}\cdot\min\Big\{1,\big(\tau\cdot Q_{2\tau}(K^\Delta,x)\big)^2\Big\}\cdot\frac{n^{1/4}}{\sqrt{m^{(g,e)}(K)}} \cdot q_n^{(g,e)}(K)\\&\lesssim \sqrt{\tau\cdot Q_{2\tau}(K^\Delta,x)} \cdot\frac{n^{1/4}}{\sqrt{q_n^{(g,e)}(K)}}\cdot q_n^{(g,e)}(K)\\*
	&= \sqrt{\tau\cdot Q_{2\tau}(K^\Delta,x)} \cdot n^{1/4} \cdot \sqrt{q_n^{(g,e)}(K)}, \label{cor_eg_global_q-global-local}
\end{align}
where we have used that
\begin{equation}
	\min\Big\{1,\big(\tau\cdot Q_{2\tau}(K^\Delta,x)\big)^2\Big\} \leq  \sqrt{\tau\cdot Q_{2\tau}(K^\Delta,x)}.
\end{equation}
This implies
\begin{align}
	q^{(g,e)}_{t,n}(K-\beta^\natural)+\tau \cdot u &\lesssim \sqrt{\tau\cdot Q_{2\tau}(K^\Delta,x)}  \cdot n^{1/4} \cdot\sqrt{q_n^{(g,e)}(K)}+ \tau \cdot u \\*
	&\stackrel{\mathllap{\eqref{cor_eg_global_condition_n}}}{\lesssim} \sqrt{\tau\cdot Q_{2\tau}(K^\Delta,x)} \cdot n^{1/4} \cdot \sqrt{\tau\cdot Q_{2\tau}(K^\Delta,x)} \cdot n^{1/4}+\tau\cdot Q_{2\tau}(K^\Delta,x)\cdot \sqrt{n} \\*
	&\leq 2 \cdot \sqrt{n}\cdot \tau\cdot Q_{2\tau}(K^\Delta,x),
\end{align}
and if the hidden constant in \eqref{cor_eg_global_condition_n} is appropriately chosen, then the condition \eqref{thm_eg_local_condition_n} is indeed fulfilled.

Analogously to the estimates in \eqref{cor_eg_global_q-global-local}, we obtain from Lemma~\ref{lemma_q_and_m_global} that (note that $0 \in K-\beta^\natural$)
\begin{equation}
	m^{(g,e)}_t(K-\beta^\natural)\leq \sqrt{\tau\cdot Q_{2\tau}(K^\Delta,x)} \cdot n^{1/4} \cdot \sqrt{m^{(g,e)}(K)}.
\end{equation}
This implies
\begin{align}
&\frac{1}{(\tau\cdot Q_{2\tau}(K^\Delta,x))^2}\cdot\pospart[\bigg]{\rho_t(\beta^\natural) + u^2 \cdot \sigma(\beta^\natural)\cdot \frac{m_t^{(g,e)}(K-\beta^\natural)}{\sqrt{n}}} \\*
\leq {} & \max\Big\{1,\big(\tau\cdot Q_{2\tau}(K^\Delta,x)\big)^{-2}\Big\} \cdot \pospart[\bigg]{\rho_0(\beta^\natural) + u^2 \cdot \sigma(\beta^\natural)\cdot \frac{\sqrt{m^{(g,e)}(K)}}{n^{1/4}}} \\*
\leq {} & \tilde{C}^{-1} \cdot t,
\end{align}
where we have used the definition of $t$ and the fact that
\begin{equation}
	\frac{\sqrt{\tau\cdot Q_{2\tau}(K^\Delta,x)}}{\big(\tau\cdot Q_{2\tau}(K^\Delta,x)\big)^2}= \big(\tau\cdot Q_{2\tau}(K^\Delta,x)\big)^{-\frac{3}{2}} \leq \max\Big\{1,\big(\tau\cdot Q_{2\tau}(K^\Delta,x)\big)^{-2}\Big\}.
\end{equation}
Finally, if the constant $\tilde{C}$ is chosen sufficiently large, then the condition \eqref{thm_eg_local_condition_t} is fulfilled and Theorem~\ref{thm_eg_local} yields the claim. \qed

\subsection{Proof of Proposition~\ref{prop_intro}}
\label{subsec:proofs:prop_intro}

Proposition~\ref{prop_intro} is a direct consequence of Corollary~\ref{cor_eg_global} with the following specifications. Since $\beta^\natural \coloneqq \beta^*$, we have $\rho_0(\beta^*) \leq 0$ according to Appendix~\ref{subsec:app:rem_on_mc}. The global complexity terms can be bounded according to Proposition~\ref{prop_polytopes}. Finally, we make use of the bound \eqref{eq_paley_zygmund} from Appendix~\ref{subsec:app:rem_main}\ref{rmk_local_smallball}. Let $\alpha$ and $\delta$ be as in \eqref{paley_alpha_delta} and set $\tau \coloneqq \alpha / 4$. Since $x$ is isotropic, we have $\delta=1$. To see that $\alpha \gtrsim \kappa^{-3}$, observe that
\begin{equation}
    1 = \E[\scpr{x,v}^2]^2 \leq \E[|\scpr{x,v}|] \cdot \E[|\scpr{x,v}|^3] \leq \E[|\scpr{x,v}|] \cdot 3^3 \cdot \kappa^3 \quad \text{for all } v\in \Sphere^{p-1},
\end{equation}
where we have used the isotropy of $x$, the Cauchy-Schwarz inequality, and Proposition~\ref{prop_sg_se_properties}\ref{prop_sg_se_mom}. Now, \eqref{eq_paley_zygmund} implies that $\tau \cdot Q_{2\tau}(K^\Delta,x) \gtrsim \kappa^{-9}$. Since $\tau=\alpha / 4$, \eqref{eq_paley_zygmund} is equivalent to $Q_{2\tau}(K^\Delta,x) \geq \alpha^2 / (4\delta)$, which implies that $Q_{2\tau}(K^\Delta,x) \gtrsim \kappa^{-6}$. Finally, note that due to the isotropy of $x$, we have that $\kappa \gtrsim 1$. This concludes the proof.
\qed

%%% acknowledgements
\section*{Acknowledgments}
{\smaller
	M.G.\ acknowledges support by the German Federal Ministry of Education and Research (BMBF) through the Berlin Center for Machine Learning (BZML), Project~AP4.
	C.K.\ is supported by the Berlin Mathematics Research Center MATH+.
	
}

%%% Biblography
\renewcommand*{\bibfont}{\smaller}
% \nocite{*} % auch nichtzitierte Werke im Verzeichnis
%\begin{refcontext}[sorting=nyt]
%	\printbibliography[heading=bibintoc]
%\end{refcontext}
% \nocite{*}
 \bibliographystyle{abbrvnat}
 \bibliography{references.bib}

\appendix
%\clearpage
%\addappheadtotoc
\begin{appendices}
	
\section{Basic Facts on Sub-Gaussian and Sub-Exponential Random Variables}
\label{sec:app:facts}

The following proposition provides two characterizations of sub-exponential and sub-Gaussian random variables. The first one concerns the exponential decay behavior of their tails, while the second one addresses the growth of their absolute moments. Note that the dependence of the constants on $\alpha$ can be dropped here, since we consider only two values of $\alpha$.
\begin{proposition}[{\citep[Prop.~2.5.2, Prop.~2.7.1]{ver18}}] \label{prop_sg_se_properties} 
	For $\alpha \in \{1,2\}$, let $Z\in L_{\psi_\alpha}$. Then the following holds true:
	\begin{thmproperties}
		\item \label{prop_sg_se_ci} $Z$ satisfies the concentration inequality 
		\begin{equation}
			\Prob(|Z| \geq t) \leq 2 \exp\Big(-\Big(\tfrac{t}{C_\alpha\cdot\norm{Z}_{\psi_\alpha}}\Big)^\alpha\Big)\quad \text {for all $t \geq 0$,}
		\end{equation}
		where $C_\alpha > 0$ is a constant depending on $\alpha$.
		\item \label{prop_sg_se_mom} The moments of $Z$ satisfy 
		\begin{equation}
			\norm{Z}_{L^q} \leq C_\alpha\cdot\norm{Z}_{\psi_\alpha} \cdot q^{1/\alpha}\quad \text {for all $q \geq 1$,}
		\end{equation}
		where $C_\alpha > 0$ is a constant depending on $\alpha$.
	\end{thmproperties}
\end{proposition}

The following two results are well-known inequalities for sub-Gaussian and sub-expo\-nen\-tial random vectors, respectively.
A comparison of both shows that the (weighted) sum of independent sub-exponential variables exhibits a mixed-tail behavior, as if it \emph{``were a mixture of sub-gaussian and sub-exponential distributions''}; quote from \citep[p.~35]{ver18}.
\begin{theorem}[\protect{Hoeffdings's inequality; \citep[Thm.~2.6.3]{ver18}}]\label{thm_hoeff_ineq}
	Let $x = (x_1, \dots, x_p)\in \R^p$ be a centered random vector with independent, sub-Gaussian coordinates and $R \coloneqq \max_{1 \leq j \leq p} \norm{x_j}_{\psi_2}$. Then, for every $v \in \R^p$ and $t \geq 0$, we have that
	\begin{equation} \label{thm_hoeff_ineq_bound}
		\Prob (|\scpr{x,v}| \geq t ) \leq 2 \exp\Big(- C_H \cdot \tfrac{t^2}{R^2\norm{v}_2^2}\Big),
	\end{equation}
	where $C_H > 0$ is a universal constant.
\end{theorem}

\begin{theorem}[\protect{Bernstein's inequality; \citep[Thm.~2.8.2]{ver18}}]\label{thm_bern_ineq}
	Let $x = (x_1, \dots, x_p)\in \R^p$ be a centered random vector with independent, sub-exponential coordinates and $R \coloneqq \max_{1 \leq j \leq p} \norm{x_j}_{\psi_1}$. Then, for every $v \in \R^p$ and $t \geq 0$, we have that
	\begin{equation} \label{thm_bern_ineq_bound}
		\Prob (|\scpr{x,v}| \geq t ) \leq 2 \exp\Big(- C_B \cdot \min\Big\{\tfrac{t^2}{R^2\norm{v}_2^2},\tfrac{t}{R\norm{v}_\infty}\Big\}\Big),
	\end{equation}
	where $C_B > 0$ is a universal constant.
\end{theorem}

\section{Further Details on Section~\ref{sec:mainresults}}
\label{sec:app:remarks}

\subsection{Remarks on the Mismatch Covariance (Definition~\ref{def_mismatch})}
\label{subsec:app:rem_on_mc}

In this part, we adopt the notation of Definition~\ref{def_mismatch}, which has introduced the mismatch parameters.
Since $K^t \subset K^0 \subset \Sphere^{p-1}$, we observe that the mismatch covariance satisfies
\begin{equation}\label{eq:mismatchcovar:loc-vs-glob}
	\rho_t(\beta^\natural) \leq \rho_0(\beta^\natural) \leq \rho(\beta^\natural).
\end{equation}
These bounds imply that all estimation guarantees presented in Section~\ref{sec:mainresults} remain true when replacing the local mismatch covariance by its global variant.
However, there exist relevant scenarios where the second inequality in \eqref{eq:mismatchcovar:loc-vs-glob} becomes strict, so that considering $\rho(\beta^\natural)$ leads to suboptimal results.
To this end, it is useful to first relate the mismatch covariance to the expected risk minimization problem \eqref{LS_K_exp}: Let $\mathcal{L}(\beta) \coloneqq \E[(y-\scpr{x,\beta})^2]$ be the objective function in \eqref{LS_K_exp} and assume that $K \subset \R^p$ is compact and convex. Then, we have that $\nabla\mathcal{L}(\beta^\natural) = 2 \E[(\scpr{x,\beta^\natural} - y) x]$ and therefore
\begin{equation}
	\rho_t(\beta^\natural) = \sup_{v \in K^t} \SCPR[\big]{{-\tfrac{1}{2}\nabla\mathcal{L}(\beta^\natural)}, v } \quad \text{ and }\quad \rho(\beta^\natural) = \norm[\big]{\tfrac{1}{2}\nabla\mathcal{L}(\beta^\natural)}_2.
\end{equation}
Now, let $\beta^\natural \in K$ be an expected risk minimizer on $K$, i.e., a solution to \eqref{LS_K_exp}.
A well-known optimality condition in convex analysis then implies that $\rho_t(\beta^\natural) \leq 0$.
On the other hand, if~$K$ does not contain a global expected risk minimizer, i.e., a solution to $\min_{\beta \in \R^p}\mathcal{L}(\beta)$, we have that $\nabla\mathcal{L}(\beta^\natural) \neq 0$ and therefore $\rho(\beta^\natural) > 0$; and vice versa, $\rho(\beta^\natural)= 0$ implies that $K$ contains a global expected risk minimizer. We refer to Figure~\ref{fig:mismatchcovar} for an illustration of this argument when $x$ is isotropic.\footnote{In the isotropic case, there also exists a nice functional-analytic interpretation: The mapping $\beta \mapsto \scpr{x,\beta}$ is an isometric embedding of the Hilbert space $\R^p$ into $L_2(\Omega, \Prob)$. Then the components $x_1, \dots, x_p \in L_2$ of the random vector $x$ constitute an orthonormal basis for the subspace $G \coloneqq \{\scpr{x,\beta} \mid \beta \in \R^p\}$. This implies $\rho(\beta^\natural) = \norm{(\scpr{y-\scpr{x,\beta^\natural},x_j}_{L_2})_{j = 1}^p}_2 = \norm{P_G(y-\scpr{x,\beta^\natural})}_{L_2}$, where $P_G$ is the orthogonal projection onto $G$; in other words, $\rho(\beta^\natural)$ corresponds to the ``linear component'' of the expected risk of $\beta^\natural$.}

In view of our main result, Theorem~\ref{thm_eg_local}, the local mismatch covariance measures the asymptotic impact of the model mismatch. When positive, $\rho_t(\beta^\natural)$ can be seen as an asymptotic bias term, while a negative value can have favorable effects on the estimation performance of \eqref{LS_K}; see Appendix~\ref{subsec:app:rem_main}\ref{rmk_local_mendelson}.

\begin{figure}
	\centering
	\begin{tikzpicture}[scale=2.5]
	\coordinate (K1) at (-.5,-.3);
	\coordinate (K2) at (.3,-.8);
	\coordinate (K3) at (-.6,-1.4);
	\coordinate (K4) at (-.6,-1.4);
	\coordinate (K5) at (-1,-.8);
	\coordinate[below right=0 and 0 of K1] (betatarg);
	\coordinate (betastar) at ($(betatarg) + (100:.5cm)$);
	
	\filldraw pic[top color=blue!30, bottom color=white, angle radius=3cm] {angle=K5--K1--K2};
	\node at ([xshift=.2cm,yshift=-.4cm]K2) {$\cone{K - \beta^\natural} + \beta^\natural$};
	
	\draw[fill=lightgray, name path = K] (K1) .. controls (K2) and (K3) .. (K4) -- (K5) -- cycle;
	\node[draw,circle,minimum size=1.4cm,dashed] at (betatarg) {}; 
	\begin{scope}
	\clip (K1) -- (K2) -- (K3) -- (K4) -- (K5) -- cycle;
	\node[draw=red,circle,minimum size=1.4cm,ultra thick] at (betatarg) {};
	\end{scope}
	\draw[thick] (K1) .. controls (K2) and (K3) .. (K4) -- (K5) -- cycle;
	
	\node at (barycentric cs:K1=1,K2=1,K3=1,K4=1,K5=1.6) {$K$};
	
	\path (betatarg) -- ++(-60:.3cm) coordinate (anchorSphere);
	\node[above right=0 and 1 of anchorSphere] (anchorSpherelabel) {$K^0 + \beta^\natural$};
	\path[<-,shorten <=3pt,>=stealth,bend right] (anchorSphere) edge (anchorSpherelabel);
	
	\coordinate (gradient) at ($(betatarg)!.4!(betastar)$);
	\node[left=1 of gradient] (gradientlabel) {$-\tfrac{1}{2}\nabla\mathcal{L}(\beta^\natural)$};
	\path[<-,shorten <=3pt,>=stealth,bend left] (gradient) edge (gradientlabel);
	
	\draw[thick,red,latex-] (betastar) -- (betatarg);
	\draw[dashed,red] ($(betatarg) + (190:1cm)$) -- ($(betatarg) + (10:1cm)$);
	\node[blackdot,label={[label distance=-4pt]above right :$\beta^\natural$}] at (betatarg) {};
	\node[blackdot,label={[label distance=-4pt]above right :$\beta^* = \E[yx]$}] at (betastar) {};
	
	\end{tikzpicture}
	\caption{An illustration of the argument in Appendix~\ref{subsec:app:rem_on_mc} when $x$ is isotropic: In this case, we have $\tfrac{1}{2}\nabla\mathcal{L}(\beta) = \beta - \E[y x]$, implying that $\beta^* \coloneqq \E[y x]$ is the (unique) global expected risk minimizer. The above figure shows a situation where $\beta^* \not\in K$ and $\beta^\natural$ is the expected risk minimizer on~$K$. This implies that the negative gradient at~$\beta^\natural$ points out of~$K$ in the direction of~$\beta^*$, or more geometrically, the dashed supporting hyperplane separates $K$ and $\beta^*$. Hence, we have that $\scpr{-\nabla\mathcal{L}(\beta), v} \leq 0$ for all $v \in K - \beta^\natural$, and in particular, $\rho_t(\beta^\natural) \leq 0$. On the other hand, it holds that $\rho(\beta^\natural) = \norm{\tfrac{1}{2}\nabla\mathcal{L}(\beta^\natural)}_2 = \norm{\beta^* - \beta^\natural}_2 > 0$.}
	\label{fig:mismatchcovar}
\end{figure}

\subsection{Remarks on Theorem~\ref{thm_eg_local}}
\label{subsec:app:rem_main}

This part compiles several additional remarks on our main result, Theorem~\ref{thm_eg_local}.
\begin{rmklist}
    \item\label{rmk_local_extensions}
		\textbf{Possible extensions.} Theorem~\ref{thm_eg_local} is amenable to various extensions and generalizations.
		For instance, replacing the $\l{2}$-norm by an arbitrary semi-norm $\gennorm$ in Fact~\ref{fact_excess} would lead to an error bound in terms of $\gennorm$; note that such a step would also require an adaptation of the spherical intersections in the q- and m-complexities and the small-ball condition \eqref{small_ball_condition}.
		This extension becomes particularly useful when the covariance matrix of the input vector~$x$ is poorly conditioned or even degenerate.\footnote{In principle, Assumption~\ref{model_setup} imposes no explicit conditions on the covariance structure of~$x$, but if it becomes too degenerate, the small-ball condition \eqref{small_ball_condition} might become unrealizable for $\gennorm = \gennorm_2$.} In this case, an appropriate linear transform of the $\l{2}$-error can account for the underlying covariance structure; see \citep[Chap.~3 and Sec.~4.3]{gen19} for a detailed discussion of this issue.
		
		Apart from that, it is possible to incorporate different loss functions or adversarial noise into Theorem~\ref{thm_eg_local}; cf.~\citep{gen16} and \citep[Chap.~3]{gen19}. Working out the details goes beyond the scope of this paper, but is expected to be relatively straightforward.
		Furthermore, one might show similar estimation guarantees for the ``basis-pursuit'' version or the unconstrained version of the generalized Lasso; cf.~\citep[Chap.~3]{gen19} and \citep{lm16,lm17}.
		Finally, it is worth mentioning that the sub-exponentiality of $y$ in Assumption~\ref{model_setup} could be replaced by a less restrictive tail condition.
		This modification would concern the analysis of the multiplier process in Subsection~\ref{subsec:proofs:multiplierprocess}; for example, a finite moment assumption for $y$ would be sufficient when using Theorem~\ref{thm_men_general_bound} instead of Theorem~\ref{thm_men_subexp_bound}.
	\item\label{rmk_local_smallball}
		\textbf{Small-ball condition.} The small-ball condition \eqref{small_ball_condition} in Assumption~\ref{model_setup} is based on Mendelson's more general condition (see \citep[Asm.~3.1]{men15}), which reads
		\begin{equation}
			\exists u>0: Q_{\mathcal{F}-\mathcal{F}}(u) \coloneqq \inf_{f \in \mathcal{F}-\mathcal{F}} \Prob(|f(x)| \geq u \norm{f}_{L^2})>0,
		\end{equation}
		where $\mathcal{F}$ is a class of (not necessarily linear) hypothesis functions. We emphasize that the small-ball condition \eqref{small_ball_condition} is stated relative to the hypothesis set $K$, and it particularly suffices for $x$ to be non-degenerate relative to the subspace $\spann(K-K)$. This reflects the fact that the input vectors~$x_i$ are only of interest to us insofar as they enable us to discern differences between the hypotheses in $K$. 
		
		Furthermore, one can easily replace the small-ball function $Q_{2\tau}(K^\Delta,x)$ in Theorem~\ref{thm_eg_local} by a more explicit expression.
		For instance, the Paley-Zygmund inequality implies the following lower bound (cf.~\citep[Subsec.~2.6.5]{tro15}): Let
		\begin{equation}\label{paley_alpha_delta}
			\alpha \coloneqq \inf_{v \in K^\Delta} \E[|\scpr{x,v}|] \quad \text{and} \quad \delta \coloneqq \sup_{v \in K^\Delta} \E[\scpr{x,v}^2]
		\end{equation}
		and set $\tau \coloneqq \alpha / 4$. As long as $\alpha > 0$, we have that
		\begin{equation}\label{eq_paley_zygmund}
			\tau\cdot Q_{2\tau}(K^\Delta,x) \geq \frac{\alpha^3}{16\delta}.
		\end{equation}
		This lower bound is a more convenient expression, since $\alpha$ measures the degeneracy of $x$ relative to~$K^\Delta$ while $\delta$ can be seen as an (an-)isotropy parameter.
	\item\label{rmk_local_mendelson}
		\textbf{Low- and high-noise regime.} The statement of Theorem~\ref{thm_eg_local} could be further refined by specifying the smallest value of~$t$ such that the conditions \eqref{thm_eg_local_condition_n} and \eqref{thm_eg_local_condition_t} still hold true, while all other model parameters remain fixed.
		Such an optimization strategy is elaborated in the general learning framework of \citet{men15}. 
		Although the latter has certainly a wider scope than ours, there are important conceptual overlaps. 
		Indeed, \eqref{thm_eg_local_condition_n} is closely related to what Mendelson refers to as the ``low-noise'' regime: this condition is intrinsic to the hypothesis set $K$ and does not depend on the model mismatch (the ``noise'') $y-\scpr{x,\beta^\natural}$; in particular, it specifies how many samples are required for \eqref{LS_K} to recover a linear hypothesis function exactly.
		In contrast, the condition \eqref{thm_eg_local_condition_t} is associated with the ``high-noise'' regime, as it strongly depends on the model mismatch in terms of $\rho_t(\beta^\natural)$ and $\sigma(\beta^\natural)$.
		A remarkable conclusion is possible when $\rho_0(\beta^\natural) < 0$ (cf.~Appendix~\ref{subsec:app:rem_on_mc}): in this case, we may simply set $t = 0$, while \eqref{thm_eg_local_condition_t} can be even satisfied if $\sigma(\beta^\natural) > 0$. In other words, exact recovery of $\beta^\natural$ is feasible in certain scenarios, despite the presence of noise or model misspecifications.
	\item\label{rmk_local_radius}
		\textbf{Alternative m-complexity.} An important detail in the definition of the m-complexity $m_t^{(g,e)}$ is that $\convh(S)$ needs to contain the origin as well.
		Without this minor modification, an extra term would have to be added to $m_t^{(g,e)}$ in the error bound of \eqref{thm_eg_local_condition_t}, capturing the radii of $(K-\beta^\natural) \cap t\Sphere^{p-1}$ with respect to $\gennorm_g$ and $\gennorm_e$ (see the proof of Lemma~\ref{lemma_lambda_tilde}).
		The appearance of such an additive term is in fact quite common in the literature, e.g., see~\citep{dir15, men16}.
\end{rmklist}

\section{Proof of Proposition~\ref{prop_upper_bounds_complexities}}
\label{sec:app:complexities}

For the claims of~\ref{prop_upper_bounds_complexities_20} and~\ref{prop_upper_bounds_complexities_2inf}, we first apply Lemma~\ref{lemma_q_and_m_global}\ref{lemma_q_and_m_global_3}--\ref{lemma_q_and_m_global_5} to observe that
\begin{equation}\label{prop_upper_bounds_complexities_local_global}
q_{t,n}^{(g,e)}(K-\beta^\natural) \lesssim m_{t}^{(g,e)}(K-\beta^\natural) \leq m_{0}^{(g,e)}(K-\beta^\natural) = m^{(g,e)}((\underbrace{\cone{K-\beta^\natural}}_{\eqqcolon \mathcal{D}_1(\beta^\natural)} {} \cap {} \Sphere^{p-1}) \cup \{0\}),
\end{equation}
where the second inequality is due to $\tfrac{1}{t}(K-\beta^\natural) \subset \cone{K-\beta^\natural}$ for all $t > 0$.
Moreover, according to \eqref{prop_sg_eq} and \eqref{bound_exponential_width}, it holds that
\begin{equation}
m^{(2,0)}((\mathcal{D}_1(\beta^\natural)\cap \Sphere^{p-1}) \cup \{0\}) \asymp w((\mathcal{D}_1(\beta^\natural)\cap \Sphere^{p-1}) \cup \{0\})
\end{equation}
and
\begin{equation}
m^{(2,\infty)}((\mathcal{D}_1(\beta^\natural)\cap \Sphere^{p-1}) \cup \{0\}) \lesssim \sqrt{\log(p)} \cdot w((\mathcal{D}_1(\beta^\natural)\cap \Sphere^{p-1}) \cup \{0\}).
\end{equation}
Since $\mathcal{D}_1(\beta^\natural)$ corresponds to the descent cone of $\gennorm_1$ at a $k$-sparse vector $\beta^\natural$, the claims of~\ref{prop_upper_bounds_complexities_20} and~\ref{prop_upper_bounds_complexities_2inf} now follow from a standard bound for the conic Gaussian width (e.g., see \citep[Eq.~2.11]{tro15}):
\begin{equation}
w((\mathcal{D}_1(\beta^\natural)\cap \Sphere^{p-1}) \cup \{0\}) \lesssim \sqrt{k \log\Big(\frac{p}{k}\Big)}.
\end{equation}

The proof of~\ref{prop_upper_bounds_complexities_02} is a little more involved. In a first step, we intend to show that
\begin{equation}\label{prop_upper_bounds_complexities_02_conebound}
\mathcal{D}_1(\beta^\natural) \cap \Sphere^{p-1} \subset \convh(S)
\end{equation}
with $S \coloneqq \{v \in \R^p \MID \norm{v}_0 \leq k, \norm{v}_2 \leq 3\}$. In order to verify this inclusion, we adapt the proof of \citet[Lem.~13]{pw14}:
Writing $A \coloneqq \mathcal{D}_1(\beta^\natural) \cap B_2^p$, we define the support functions $h_A,h_S: \R^p \to \R$ by
\begin{equation}
h_A(z) \coloneqq \sup_{v \in A} \scpr{z,v} \quad \text{ and } \quad h_S(z) \coloneqq \sup_{v \in S} \scpr{z,v}=\sup_{v \in \convh(S)} \scpr{z,v}.
\end{equation}
If there would exist a point $v \in A \setminus \convh(S)$, then a variant of the hyperplane separation theorem implies that there exists some $z \in \R^p$ with $\scpr{z,v} > h_S(z)$ and therefore $h_A(z)>\nobreak h_S(z)$. Consequently, it suffices to show that $h_A(z) \leq h_S(z)$ holds true for all $z \in \R^p$.

To this end, let $z \in \R^p \setminus \{0\}$ be fixed and let $I_1 \subset \{1,\dots,p\}$ contain $k$ indices whose corresponding entries in $z$ have the $k$ largest absolute values. Then, it is not hard to see that $h_S(z) = \scpr{z,v}$ for
\begin{equation}
v = \frac{3}{\norm{P_{I_1}(z)}_2} \cdot P_{I_1}(z) \in S,
\end{equation}
where $P_{I_1}(z) \in \R^{p \times p}$ denotes the orthogonal projection onto the coordinate space associated with~$I_1$.
Consequently, we obtain
\begin{equation}
h_S(z)=3\cdot \sqrt{\sum_{j \in I_1}z_j^2},
\end{equation}
and for every $j \not\in I_1$, we have that
\begin{equation}
|z_j| \leq  \frac{1}{k} \cdot \sum_{j' \in I_1}|z_{j'}| \leq \frac{1}{\sqrt{k}} \cdot \sqrt{\sum_{{j'} \in I_1}z_{j'}^2}.
\end{equation}
Now let $v \in A$. We denote by $I_2 \subset \{1,\dots,p\}$ the set of indices corresponding to the non-zero entries of $\beta^\natural$, which satisfies $|I_2|\leq k$.
Since $v$ belongs to descent cone of $\gennorm_1$ at $\beta^\natural$ and to the Euclidean unit ball, it follows that
\begin{equation}
\sum_{j \not\in I_2}|v_j| \leq \sum_{j \in I_2}|v_j| \leq \sqrt{k} \sqrt{\sum_{j \in I_2}v_j^2} \leq \sqrt{k}.
\end{equation}
Combining these observations, we obtain
\begin{align}
\scpr{z,v}&=\sum_{j \in I_2}z_j v_j+\sum_{\substack{j \not\in I_2\\ j \in I_1}}z_j v_j+\sum_{\substack{j \not\in I_2\\ j \not\in I_1}}z_j v_j\\*
&\leq 2 \cdot \norm{v}_2 \cdot \sqrt{\sum_{j \in I_1}z_j^2} + \max_{j \not\in I_1} |z_j| \cdot \sqrt{k} \leq 3 \cdot \sqrt{\sum_{j \in I_1}z_j^2} =h_S(z),
\end{align}
which concludes the proof of \eqref{prop_upper_bounds_complexities_02_conebound}.

Next, we observe that $S \subset 3 \sqrt{k} B_1^p$, which is a consequence of the Cauchy-Schwarz inequality.
Now, consider the finite set $F \coloneqq 3 \sqrt{k} \cdot \{\pm u_1, \dots, \pm u_p\}$, where $u_1, \dots, u_p \in \R^p$ denote the Euclidean unit vectors.
Then, we have that
\begin{equation}\label{prop_upper_bounds_complexities_02_conebound_2}
	\mathcal{D}_1(\beta^\natural) \cap \Sphere^{p-1} \subset \convh(S) \subset \convh(3 \sqrt{k} B_1^p) = 3 \sqrt{k} B_1^p = \convh(F).
\end{equation}
Finally, we make use of \eqref{eq_complexity_finiteset} to obtain the following bound:
\begin{equation}
	\gamma_1(F, \gennorm_2) \lesssim \Delta_2(F) \cdot \log(|F|) \lesssim \sqrt{k} \cdot \log(2p).
\end{equation}
The claim of~\ref{prop_upper_bounds_complexities_02} follows directly from a combination of this bound with \eqref{prop_upper_bounds_complexities_local_global}, \eqref{prop_upper_bounds_complexities_02_conebound_2}, and the definition of the m-complexity.
\qed

\section{Proofs for Subsection~\ref{subsec:proofs:implications}}
\label{sec:app:implications}

\begin{proof}[Proof of Lemma~\ref{lemma_bc_norm_bound}] 
	Using the generic Bernstein concentration of $x$, we observe that
	\begin{align}
		\norm{v^*}_{L^q}^q = \E[|\scpr{v,x}|^q] &= \int_0^\infty \Prob(|\scpr{v,x}|^q \geq s) \ ds \\*
		&= \int_0^\infty \Prob(|\scpr{v,x}| \geq r) \cdot q r^{q-1} \ dr
		\\ &\leq \int_0^\infty 2  \exp\Big(- \min\Big\{\tfrac{r^2}{\norm{v}_g^2},\tfrac{r}{\norm{v}_e}\Big\}\Big) \cdot q r^{q-1} \ dr \\ &= \int_0^\infty 2  \max \Big\{\exp\Big(-\tfrac{r^2}{\norm{v}_g^2}\Big),\exp\Big(-\tfrac{r}{\norm{v}_e}\Big) \Big\} \cdot q r^{q-1} \ dr \\
		&\leq \int_0^\infty 2 \exp\Big(-\tfrac{r^2}{\norm{v}_g^2}\Big) \cdot q r^{q-1} \ dr + \int_0^\infty 2\exp\Big(-\tfrac{r}{\norm{v}_e}\Big) \cdot q r^{q-1} \ dr \\
		&= 2q \cdot \norm{v}_g^q\cdot \int_0^\infty \exp(-r^2) \cdot r^{q-1} \ dr + 2q \cdot \norm{v}_e^q \cdot \int_0^\infty  \exp(-r) \cdot r^{q-1} \ dr \\
		&\stackrel{\mathllap{\text{Definition of $\Gamma$-function}}}{=} q \cdot \norm{v}_g^q \cdot \Gamma(q/2) +  2q \cdot \norm{v}_e^q \cdot \Gamma(q) \\*
		& \stackrel{\mathllap{\text{Stirling's approximation}}}{\leq} 2q \cdot \norm{v}_g^q \cdot (q/2)^{q/2}+ 2q \cdot \norm{v}_e^q \cdot q^q.
	\end{align}
	Since $(a+b)^{1/q} \leq a^{1/q}+b^{1/q}$ for all $a,b \geq 0$ and $q \geq 1$, this implies
	\begin{equation}
		\norm{v^*}_{L^q} \leq \sqrt[q]{2q} \cdot \Big(\norm{v}_g \cdot \frac{\sqrt{q}}{\sqrt{2}}+ \norm{v}_e \cdot q\Big).
	\end{equation}
	Observing that $q \mapsto \sqrt[q]{2q}$ is bounded on the interval $[1,\infty)$, the claim follows. 
\end{proof}

\begin{proof}[Proof of Lemma~\ref{lemma_bern_sum}]
	Our basic proof strategy is adapted from the proof of Bernstein's inequality by \citet[Thm.~2.8.1]{ver18}. Let $\tilde{x}_i \coloneqq \varepsilon_i x_i$ and $S \coloneqq \sum_{i=1}^n \tilde{x}_i$. Then, the generic Chernoff bound (for arbitrary $v \in \R^p$, $\lambda \in \R$, and $t \geq 0$) reads as follows:
	\begin{equation}\label{lemma_bern_gen_chernoff}
		\Prob(\scpr{v,S} \geq t)=\Prob\bigg(\sum_{i=1}^n \scpr{v,\tilde{x}_i} \geq t\bigg) \leq \exp(-\lambda t) \cdot \prod_{i=1}^n \E[\exp(\lambda \scpr{v,\tilde{x}_i})].
	\end{equation}
	
	To proceed, we need an upper bound for the moment generating function $\E[\exp(\lambda \scpr{v,\tilde{x}_1})]$. Since $\tilde{x}_1$ is centered, we have that
	\begin{equation}
		\E[\exp(\lambda \scpr{v,\tilde{x}_1})] =\E\bigg[ 1+ \lambda \scpr{v,\tilde{x}_1}+ \sum_{q=2}^\infty \frac{(\lambda \scpr{v,\tilde{x}_1})^q}{q!}\bigg]=1+\sum_{q=2}^\infty \frac{\lambda^q \cdot \E[\scpr{v,\tilde{x}_1}^q]}{q!}.
	\end{equation}
	Due to symmetry, \eqref{def_gen_bern_conc_ineq} also holds true for the symmetrized random vector $\tilde{x}_1$ and from the proof of Lemma~\ref{lemma_bc_norm_bound}, we therefore obtain
	\begin{align}
		\E[\exp(\lambda \scpr{v,\tilde{x}_1})] &\leq 1+\sum_{q=2}^\infty \frac{\lambda^q \cdot \E[|\scpr{v,\tilde{x}_1}|^q]}{q!} \\* &\leq 1+\sum_{q=2}^\infty \frac{\lambda^q \cdot \big(2q \cdot \norm{v}_g^q \cdot (q/2)^{q/2}+ 2q \cdot \norm{v}_e^q \cdot q^q\big)}{q!} \\
		&=\underbrace{1+\sum_{q=2}^\infty \frac{\lambda^q \cdot 2q \cdot \norm{v}_e^q \cdot q^q}{q!}}_{\eqqcolon A}+\underbrace{\sum_{q=2}^\infty \frac{\lambda^q \cdot 2q \cdot \norm{v}_g^q \cdot (q/2)^{q/2}}{q!}}_{\eqqcolon B}.
	\end{align}
	
	\paragraph{Upper bound for $A$:} According to Stirling's approximation, we have $q!\geq (q/e)^q$, which implies
	\begin{equation}
		A \leq 1+\sum_{q=2}^\infty \frac{ 2q \cdot (\lambda\norm{v}_e)^q \cdot q^q}{(q/e)^q} = 1+ \sum_{q=2}^\infty 2q \cdot (e\lambda\norm{v}_e)^q.
	\end{equation}
	If $\lambda < \frac{1}{2e \norm{v}_e}$, the above series is convergent and we have
	\begin{equation}
		A \leq 1 +2\cdot \frac{(e\lambda\norm{v}_e)^2 \cdot (2-e\lambda\norm{v}_e)}{(1-e\lambda\norm{v}_e)^2} \leq 1 + 16  (e\lambda\norm{v}_e)^2\leq \exp(16 (\lambda \norm{v}_e)^2). 
	\end{equation}
	
	\paragraph{Upper bound for $B$:} 
	Since we do not want to introduce further restrictions on $\lambda$, let us distinguish two cases: If $\lambda < 1 / (2e \norm{v}_g)$, we use that $(q/2)^{q/2}\leq q^q$ (due to $q\geq 2$) and apply the same strategy as for the bound for $A$, which yields $B \leq \exp(16 (\lambda \norm{v}_g)^2)$. Now, let  $\lambda \geq 1 / (2e \norm{v}_g)$. 
	In this case, we use the basic inequality $2q(q / 2)^{q/2} \leq 3q! / \floor{q/2}!$, which holds true for all $q \geq 2$. This implies
	\begin{equation}
		B \leq 3\sum_{q=2}^\infty \frac{(\lambda\norm{v}_g)^q}{\floor{\frac{q}{2}}!}.
	\end{equation}
	Splitting this series into even and odd indices then yields
	\begin{equation}
		B \leq 3\bigg(\sum_{j=1}^\infty \frac{(\lambda\norm{v}_g)^{2j}}{\floor{\frac{2j}{2}}!}+\sum_{j=1}^\infty \frac{(\lambda\norm{v}_g)^{2j+1}}{\floor{\frac{2j+1}{2}}!} \bigg) =
		3\bigg(\underbrace{\sum_{j=1}^\infty \frac{(\lambda\norm{v}_g)^{2j}}{j!}}_{=\exp((\lambda\norm{v}_g)^2)-1}+\underbrace{\sum_{j=1}^\infty \frac{(\lambda\norm{v}_g)^{2j+1}}{j!}}_{\leq\lambda\norm{v}_g \exp((\lambda\norm{v}_g)^2)}\bigg).
	\end{equation}
	By $(1+\lambda\norm{v}_g)\leq \exp((\lambda\norm{v}_g)^2)$, it follows that
	\begin{equation}
		B \leq 3 \exp(2(\lambda\norm{v}_g)^2)-3.
	\end{equation}
	Since $\lambda\norm{v}_g \geq 1 / (2e)$, there exists a universal constant $C'>0$ with $3 \leq \exp(C'\cdot(\lambda\norm{v}_g)^2)$, so that we obtain
	\begin{equation}
		B \leq \exp\big((2+C')\cdot(\lambda\norm{v}_g)^2\big)-3.
	\end{equation}
	
	Next, we combine the above bounds for $A$ and $B$: the basic inequality $\exp(a)+\exp(b)-\nobreak 3\leq \exp(a+b)$ for all $a, b \geq 0$ implies that
	\begin{equation} 
		\E[\exp(\lambda \scpr{v,\tilde{x}_1})] \leq \exp\Big(\max\{2+C',16\} \cdot (\lambda\norm{v}_g)^2+16(\lambda\norm{v}_e)^2\Big) \leq \exp(C\lambda^2(\norm{v}_g+\norm{v}_e)^2)
	\end{equation} 
	for all $\lambda< 1 / (2e \norm{v}_e)$ and a universal constant $C > 0$.
	Plugging this into \eqref{lemma_bern_gen_chernoff} and assuming that $\lambda< 1 / (2e \norm{v}_e)$, we have that
	\begin{equation}
		\Prob(\scpr{v,S} \geq t) \leq \exp(-\lambda t) \cdot \prod_{i=1}^n \E[\exp(\lambda \scpr{v,\tilde{x}_1})] = \exp\big(-\lambda t + n C\lambda^2(\norm{v}_g+\norm{v}_e)^2\big).
	\end{equation}
	Setting
	\begin{equation}
		\lambda \coloneqq \min \Big\{\tfrac{t}{2n C(\norm{v}_g+\norm{v}_e)^2},\tfrac{1}{4e \norm{v}_e} \Big\} < \frac{1}{2e \norm{v}_e},
	\end{equation}
	we finally obtain
	\begin{equation}
		\Prob(\scpr{v,S} \geq t) \leq \exp\Big(-\tilde{C} \cdot \min\Big\{\tfrac{t^2}{n(\norm{v}_g+\norm{v}_e)^2},\tfrac{t}{\norm{v}_e} \Big\}\Big)
	\end{equation}
	for a universal constant $\tilde{C} > 0$. Now, the claim follows if we replace $t$ by $\sqrt{n}t$.
\end{proof}

\end{appendices}

%\newpage
%\listoftodos

\end{document}